\documentclass[11pt]{article}

\usepackage{graphicx}

\usepackage{amssymb,amsmath,amsthm,enumerate,hyperref,bm}

\usepackage{tikz}

\usepackage{fullpage}

\usetikzlibrary{decorations.markings}

\newtheorem{lemma}{Lemma}
\newtheorem{remark}{Remark}
\newtheorem{proposition}[lemma]{Proposition}
\newtheorem{definition}[lemma]{Definition}
\newtheorem{theorem}[lemma]{Theorem}

\newtheorem{corollary}[lemma]{Corollary}
\newtheorem*{Proposition*}{Proposition}

\newcommand{\BIN}{{\mathrm{Bin}}}

\newcommand{\tr}{{\rm tr}}
\newcommand{\Tr}{{\rm Tr}}
\newcommand{\dE}{\mathbb {E}}
\newcommand{\dP}{\mathbb{P}}

\newcommand{\dR}{\mathbb {R}}
\newcommand{\dC}{\mathbb {C}}
\newcommand{\cF}{\mathcal {F}}
\newcommand{\dZ}{\mathbb{Z}}

\newcommand{\cW}{\mathcal {W}}

\newcommand{\cS}{\mathcal {S}}

\newcommand{\SUPP}{ \mathrm{supp}}
\newcommand{\INT}{ \mathrm{full}}

\newcommand{\ABS}[1]{{{\left| #1 \right|}}} 
\newcommand{\BRA}[1]{{{\left\{#1\right\}}}} 
\newcommand{\SBRA}[1]{{{\left[#1\right]}}} 
\newcommand{\NRM}[1]{{{\left\| #1\right\|}}} 
\newcommand{\PAR}[1]{{{\left(#1\right)}}} 

\newcommand{\uB}{{\underline B}}

\newcommand{\1}{1\!\!{\sf I}}\newcommand{\IND}{\1}
\newcommand{\veps}{\varepsilon}
\newcommand{\si}{\sigma}

\newcommand{\BEAS}{\begin{eqnarray*}}
\newcommand{\EEAS}{\end{eqnarray*}}
\newcommand{\BEA}{\begin{eqnarray}}
\newcommand{\EEA}{\end{eqnarray}}
\newcommand{\BEQ}{\begin{equation}}
\newcommand{\EEQ}{\end{equation}}
\newcommand{\BIT}{\begin{itemize}}
\newcommand{\EIT}{\end{itemize}}
\newcommand{\BNUM}{\begin{enumerate}}
\newcommand{\ENUM}{\end{enumerate}}
\newcommand{\AND}{\quad \mathrm{ and } \quad}

\title{Eigenvalues of random lifts and polynomials of random permutation matrices}
\author{Charles Bordenave, Beno\^\i t Collins}

\begin{document}
\maketitle

\begin{abstract}
Let $(\sigma_{1}, \ldots, \sigma_d)$ be a finite sequence of independent random permutations, chosen uniformly 
either among all permutations or among all matchings on $n$ points. We show that,  in  probability, as 
$n\to\infty$, these permutations viewed as operators on the $n-1$ dimensional vector space 
$\{(x_1,\ldots, x_n)\in \dC^n, \sum x_i=0\}$, are asymptotically strongly free. 
Our proof relies on the development of a matrix version of the non-backtracking operator theory and a refined trace method.

As a byproduct, we show that the non-trivial eigenvalues of random $n$-lifts of a fixed based graphs approximately achieve the 
Alon-Boppana bound with high probability in the large $n$ limit. This result generalizes Friedman's Theorem 
stating that with high probability,
the Schreier graph generated by a finite number of independent random permutations is close to Ramanujan. 

Finally, we extend our results to tensor products of random permutation matrices. This extension is especially relevant in the context of quantum expanders.
\end{abstract}

\section{Introduction} 

\subsection{Weighted sum of permutations}
\label{sec:A}

Let $X$ be a countable set and $r,d$ positive integers. We consider $(a_0, \ldots, a_d)$  elements in $M_r (\dC)$ and let 
$\sigma_{i}\in S(X), i\in \{1,\ldots ,d\}$, be permutations of the set $X$. 
Let $\ell^2 (X)$ be the Hilbert space spanned by an orthonormal basis $\delta_x$ indexed by the elements of $x\in X$.
The permutation $\sigma_{i}$ acts naturally as a unitary operator $S_i$ on $\ell^2 (X)$
by $\sigma_i$, $S_i (g)(x) = g( \sigma_i (x))$ for all $g \in \ell^2 (X)$.
Let $1$ be the identity operator on $\ell^2 (X)$. 
We are interested in the operator on $\dC^r \otimes \ell^2 (X)$,  
\begin{equation}\label{eq:defA}
A = a_0 \otimes 1 + \sum_{i=1}^d a_i \otimes S_i.
\end{equation}
If $X$ is a finite set with $n$ elements, $A$ may be viewed as an $rn\times rn$ matrix.

If $X$ is finite,  the constant vector $\IND \in X$ is in $\ell^2 (X)$. In addition, it is
an eigenvector of any permutation matrix of $X$ associated to the eigenvalue $1$. Therefore, 
$$
H_1 = \dC^r \otimes \IND
$$ 
is an invariant vector space of $A$ and $A^*$ of dimension $r$. The restriction of $A$ to $H_1$ is given by
\begin{equation}\label{eq:defA1}
A_1 = a_0 + \sum_{i=1}^d a_i.
\end{equation}
When $X$ is finite, we are interested in the spectrum of $A$ on $H_0 = H _1^\perp$. The space $H_0$ is the set of 
$g \in \dC^r  \otimes \ell^2 (X)$ such that $\sum_x g(x) = 0 \in \dC^r$ where 
$g(x)$ denotes the orthogonal projection of $g$ on $\dC^r \otimes \{ \delta_x\}$
-- canonically identified to $\dC^r$.
We note that $A$ leaves $H_0$ invariant, i.e. $A H_0  \subset H_0$, therefore it also defines an operator on $H_0$. 
We will denote by $A_{|H_0}$ the restriction of $A$ to $H_0$, which we see as an element of 
$B( H_0)$.

We use the following standard notation. For a positive integer  $n$, wet set $ [n] = \{1, \ldots, n\}$.  
The \emph{absolute value} of a bounded operator $T$ is  
$
|T| = \sqrt{ T T^*}, 
$
$\sigma(T)$  denotes the \emph{spectrum} of $T$ and $\|T\|$ its \emph{operator norm}. 
Finally, $\rho(T)$ is the \emph{spectral radius} of $T$ and if $T$ is self-adjoint we set $s(T) = \sup \sigma(T)$ (the right edge of the spectrum). 
For example,
$$
 \| A _{|H_0} \| = \sup_{ g: g \in H_0\setminus\{0\}} \frac{ \| A g \| } { \|g \|  }. $$

It will be often useful to work with the operator $A$ when it is self-adjoint. 
Self-adjointness is ensured by replacing the operators $S_i$ by generic algebraically free unitary operators and checking
self-adjointness. This condition is essentially sufficient (it is sufficient if the cardinal of $X$ is large enough and if one 
requires the property to hold for any choice of permutation $S_i$).
For practical purposes, let us assume that the set $\{1, \ldots, d\}$ is endowed with an involution 
$i \mapsto i^{*}$ (and $i^{**} = i$ for all $i$). 
Then, the symmetry condition is fulfilled as soon as
\begin{equation}\label{eq:sym}
\hbox{$a_0^*=a_0$ \quad and \quad  for all  $i \in \{1, \ldots, d\}$, } \quad (a_i ) ^ * = a_{i^*}   \AND \sigma_{i^ *} = \sigma_i ^{-1}.  
\end{equation}
If the symmetry condition \eqref{eq:sym} holds, then $A$ is self-adjoint. 
In this case, 
$$s(A_{|H_0}) = \sup_{ g: g \in H_0\setminus\{0\}} \frac{ \langle g ,  A g \rangle } { \|g \|^2  } .$$

We are interested in the spectrum of $A$ when $X = [n]$, the permutations $\sigma_i$ are random and $n$ is large.
The operator $A$ becomes a random matrix which we study in the case when $\sigma_i$, $i \in [d]$
are random permutations whose distribution are described 
below.

\begin{definition}[Symmetric random permutations]\label{def:ps}
For some integer $0 \leq q \leq d/2$, we have for $1 \leq i \leq q$, $i^* = i + q$, for $q+1 \leq i \leq 2q$, $i^* = i - q$ 
and for $2q+1 \leq i \leq d$, $i^* = i$. The permutations $(\sigma_i), i \in  \{ 1, \ldots , q \}  \cup \{ 2q+1 , \ldots, d \}$, 
are independent, for $1 \leq i \leq 2q$, $\sigma_i$ is uniformly distributed in $\cS_n$ and $\sigma_{i^*} = \sigma_i^{-1}$. 
If $2q < d$, then we assume that $n$ is even and for $2q + 1 \leq i \leq d$, $\sigma_i$ is a uniformly distributed matching 
on $[n]$ (where a matching is a permutation $\sigma$ such that $\sigma^2 (x) = x$ and $\sigma( x) \ne x$ for all $x \in [n]$). 
\end{definition}

\subsection{Large $n$ limit, non-commutative probability spaces}

The following operator describes the local limit of $A$ (in the sense of Benjamini-Schramm, see \cite{MR2354165,MR3792625}) 
as $n$ grows large. For symmetric random permutations, 
let $q$ be as in Definition \ref{def:ps} and $X = X_\star= \dZ * \cdots * \dZ * \dZ_2* \cdots * \dZ_2$ be the free product 
generated by  $q$ copies of $\dZ$ and 
$d -2q$ copies of $\dZ_2$. We denote by $g_1, \ldots, g_d$ its 
generators, where  if $1 \leq i \leq q$, $(g_i, g_{i+q}) $ 
generates the $i$-th copy of $\dZ$.  
In $\dC^r \otimes \ell^{2} (X_\star)$, we define the convolution operator
\begin{equation}\label{eq:defAfree}
A_\star= a_0 \otimes 1 + \sum_{i=1}^d a_i  \otimes\lambda (g_i) ,
\end{equation}
with $a_0, \ldots , a_d$ from Equation \eqref{eq:defA}, and where $\lambda(g)$ is the left regular representation (left multiplication). 

In the non-commutative probability vocabulary, $A_\star$ is called a \emph{non-commutative random variable}, namely,
it is an element of $\mathcal{A}$ where $\mathcal{A}$ is a unital $^*$-algebra and $\tau$ is a faithful trace on it. 
Here, $\mathcal{A}$ is $M_r(C_r (X_\star))$, where $C_r (X_\star)$ is the reduced $C^*$-algebra of the group $X_\star$
and the trace is $r^{-1}\Tr\otimes \tau$, where $\tau (\lambda (g) )  = \IND (  g= e)$.

Recall that a sequence of complex random variables $(Y_n)$ converges \emph{in probability} to $y \in \dC$, if for any 
$\veps >0$, $ \dP ( |Y_n  - y | \geq \veps)$ converges to $0$ as $n$ goes to infinity.

\subsection{Linear or not linear?}

In this paper we study the spectrum of the operator $A$ defined in Equation 
\eqref{eq:defA}. The spectrum of the limiting operator defined in Equation \eqref{eq:defAfree} 
gives a candidate for the limiting spectrum, which we will show to be correct with high probability. 
This operator $A$ is a \emph{linear} combination of the permutation matrices $S_i$'s with matrix coefficients. 
On the other hand, the abstract of this manuscript mentions strong asymptotic freeness. As defined below, this is a property 
that involves the behaviour of any non-commutative \emph{polynomial} in the variables $S_i$'s with scalar coefficient, i.e. it is \emph{not} 
 necessarily a linear combination of the $S_i$'s. So there is no obvious a priori implication between the two problems. 
It turns out that these questions are actually equivalent. This fact is an important phenomenon that has been
widely used in random matrix theory in the last two decades, known as the \emph{linearization trick}. 
Details are provided in section \ref{sec:cor}.

\subsection{Large $n$ limit, main result}

For symmetric random permutations, it follows from results of Nica \cite{MR1197059} that the operators $(S_i), i \in [d],$ 
are \emph{asymptotically free} in 
probability, in the sense that for any polynomial $P$ in unitaries $\lambda(g_i)$ ($(g_i), i \in [d],$ symmetric generators of 
the group $X_\star$, as per the definition 
above Equation \eqref{eq:defAfree}), the corresponding polynomial $P_n$ obtained by replacing $\lambda(g_i)$ by $S_i$ 
(seen as a random variable permutation on $M_n(\mathbb{C} )$) satisfies that the random variable
$n^{-1}\Tr (P_n) \to \tau (P)$, where this convergence holds in probability. 

This notion is a particular case of the concept of \emph{asymptotic freeness}. A good and modern introduction can be found in 
\cite{MR3585560}. Although the results of this paper can also be seen as a contribution to the asymptotic theory of freeness,
a non-expert reader can safely assume that the explanations developed in this paragraph cover the necessary background
in free probability. 

This notion of convergence proved by Nica for the permutations operators $S_i$ as the dimension $n$ grows to infinity 
shows that for any self adjoint  polynomial $P_n$ in $S_i$, the percentage of eigenvalues
in a given real interval $[a,b]$ converges to the spectral mesure of the limiting polynomial $P$ on the group 
$X_\star$ on the same
interval $[a,b]$. In particular, if $[a,b]$ does not intersect the limiting spectrum, it shows that the percentage of 
eigenvalues in this interval tends to zero. But it does not rule out the possibility for $o(n)$ eigenvalues being in 
this interval. If such eigenvalues exist they are called \emph{outliers}. 
As a matter of fact, in our model, outliers can be made to exist
by taking an appropriate polynomial and the constant vector $\IND$.  For example, $S_1+ S_1^* + \ldots +S_k + S_k^*$ 
has always an eigenvalue $2k$ and this is an outlier as soon as $k\ge 2$.

It is very natural to ask whether there are more outliers than those potential obvious ones. For some 
random matrix models, it was shown that this is not the case. 
For example, a (negative) answer to the unitary version of this problem was achieved by the second author 
and Male \cite{MR3205602}, 
as a continuation of the seminal result of Haagerup and Thorbj{\o}rnsen \cite{MR2183281}. The proofs are based on 
the  linearization trick which reduces such question to the analog question on polynomials of degree one and with 
matrix coefficients, as our operator $A$ in \eqref{eq:defA}. 

Whenever there are no outliers in a limit of a (random) matrix model, one says that it converges {\em strongly}. 
Mathematically, it is equivalent to saying that the norm of any polynomial converges to the supremum of its limiting spectrum, 
specifically, beyond assuming the existence of a limit of
$n^{-1}\Tr (P_n)$ for any polynomial, one assumes in addition that
\begin{equation}\label{eq-strong-convergence}
\lim_n ||P_n||=\lim_\ell (\lim_n n^{-1}\Tr ((P_nP_n^*)^\ell)^{(2\ell)^{-1}}
\end{equation}
(note that this notion is not probabilistic -- when the operators $P_n$ are random, e.g. because they are constructed 
out of random unitaries, then one may consider such notions of convergences to a constant in some probabilistic sense,
for example in probability, or almost surely).

The above ideas are well captured by stating that the spectrum of a self-adjoint 
operator is not far from its limiting spectrum in the 
sense of the Hausdorff distance. 
Recall  that the \emph{Hausdorff distance} between two subsets $S$ and $T$ of $\dR$ is the infimum over all $\veps >0$ 
such that $S \subset T + [-\veps, \veps]$ and $T \subset S + [-\veps, \veps]$. 
Let us also remark that it is not completely obvious at first sight that the notion introduced in Equation 
\eqref{eq-strong-convergence} (i.e. convergence of the operator norm) and the notion of convergence of spectrum in Hausdorff distance are equivalent. One has to check  the absence of outliers \emph{between} two connected components of the limiting spectrum if
it is not connected. This happens to be equivalent because the quantifier for strong convergence runs over \emph{every} polynomial. 
We refer to Section \ref{sec:cor} for additional details.

\begin{theorem}\label{th:main}
If the symmetry condition \eqref{eq:sym} holds, for symmetric random permutations, as $n$ goes to infinity,  the 
Hausdorff distance between $\sigma(A_{|H_0})$ and $\sigma(A_\star)$ converges in probability to $0$. In particular  
$s(A_{|H_0})$  converges in probability to $s(A_\star)$. 
\end{theorem}

We note that there is an explicit expression for $\| A_\star \|$ and $s (A_\star)$
 in the self-adjoint case. The scalar case $r = 1$ is due to Akemann and Ostrand \cite{MR0442698} and 
the general case for any $r$ and $a_i$ Hermitian is due to Lehner 
 \cite{MR1738412}.

A corollary of this result is

\begin{theorem}\label{thm:strong}
For symmetric random permutations, the permutation operators restricted to $\IND^\perp$, $((S_i)_{|\IND^\perp}), i \in [d],$ 
are asymptotically strongly free in probability.
\end{theorem}

\subsection{Spectral gaps of random graphs}

In Equation \eqref{eq:defA}, consider the special case where $a_0 = 0$ and for any $i \in [d]$, $a_i = E_{u_i v_i}$, 
for some $u_i,v_i$ in $[r]$ (where $(E_{uv})_{u'v'} =  \IND_{(u,v) = (u',v') }$). Then $A$ is the adjacency matrix of a colored 
graph on the vertex set $[n] \times [r]$ and whose directed edges with color $i \in [d]$ are $((x,u_i) ,(\sigma_i(x),v_i))$, 
for all $x \in [n]$. If the symmetry condition  \eqref{eq:sym} holds, then this graph is undirected. This graph is called 
a {\em $n$-lift of the base graph} whose adjacency matrix $A_1 = \sum_i a_i$ is given by \eqref{eq:defA1}, see Figure \ref{fig:nlift} for a concrete example.

\begin{figure}[htb]
\begin{center}  
\resizebox{12cm}{!}{
\begin{tikzpicture}[main node/.style={circle, draw , fill = gray!20, text = black, thick}]
\tikzstyle{root}=[inner sep=7pt,  draw, black , circle, fill= gray!20]
\node[main node]  at (0,0) (1) {1} ;
\node[main node] at (2,1) (2) {2} ;
\node[main node] at (2,0) (3) {3} ;
\node[main node] at (2,-1 ) (4) {4} ;
\node[main node] at (4,0) (5) {5} ;

\draw[thick,<->] (1) to (2) ; 
\draw[thick,<->] (1) to (3) ; 
\draw[thick,<->] (1) to (4) ; 
\draw[thick,<->] [out = 80, in = 180] (1) to (2,2) [out = 0 , in = 100 ] to (5) ; 
\draw[thick,<->] (5) to (2) ; 
\draw[thick,<->] (5) to (3) ; 
\draw[thick,<->] (5) to (4) ;

\node at (-0.1,0.5) {\scriptsize{8}} ; 
\node at (3.74,0.5) {\scriptsize{1}} ; 

\node at (0.35,0.39) {\scriptsize{9}} ; 
\node at (1.64,0.61) {\scriptsize{2}} ;

\node at (3.65,-0.39) {\scriptsize{7}} ; 
\node at (2.36,-0.61) {\scriptsize{14}} ;

    \coordinate (shift) at (10,0);
    \begin{scope}[shift=(shift)]
\node[root]  at (0,0) (01) {} ;
\node[root,label={\small{$(\sigma_1(x),5)$}}] at (-2,0) (02) {} ;
\node[root,label={\small{$(\sigma_3(x),3)$}}] at (2,0) (03) {} ;
\node[root,label={\small{$(\sigma_2(x),2)$}}] at (0,2 ) (04) {} ;
\node[root,label=below:{\small{$(\sigma_4(x),4)$}}] at (0,-2) (05) {} ;

\node  at (01) {\small{$(x,1)$}};

\draw[thick,<->] (01) to (02) ; 
\draw[thick,<->] (01) to (03) ; 
\draw[thick,<->] (01) to (04) ; 
\draw[thick,<->] (01) to (05) ;

\node at (-0.5,-0.2) {\scriptsize{8}} ; 
\node at (-1.5,0.2) {\scriptsize{1}} ; 

\node at (0.5,0.2) {\scriptsize{10}} ; 
\node at (1.5,-0.2) {\scriptsize{3}} ;

\node at (-0.2,0.5) {\scriptsize{9}} ; 
\node at (0.2,1.5) {\scriptsize{2}} ;

\node at (0.2,-0.5) {\scriptsize{11}} ; 
\node at (-0.2,-1.5) {\scriptsize{4}} ;

\draw[dashed,thick] (03) to (3.1,0) ; 
\draw[dashed,thick] (04) to (1.1,2) ; 
\draw[dashed,thick] (05) to (1.1,-2) ; 
\draw[dashed,thick] (02) to (-3.1,0) ;
\draw[dashed,thick] (02) to (-3.1,-0.3) ;
\draw[dashed,thick] (02) to (-3.1,0.3) ;

\end{scope}

    \coordinate (shift) at (16,-1);
    \begin{scope}[shift=(shift)]
\node[main node]  at (0,0) (11) {1} ;
\node[main node] at (2,1) (12) {2} ;
\node[main node] at (2,0) (13) {3} ;
\node[main node] at (2,-1 ) (14) {4} ;
\node[main node] at (2,2) (15) {5} ;

\draw[<->,thick] (11) to (12) ; 
\draw[<->,thick] (11) to (13) ; 
\draw[<->,thick] (11) to (14) ; 
\draw[<->,thick] (11) to (15) ;

\node at (0.25,0.5) {\scriptsize{8}} ; 
\node at (1.75,1.5) {\scriptsize{1}} ; 

\node at (2.25,2.5) {\scriptsize{7}} ; 
\node at (3.75,3.5) {\scriptsize{14}} ; 

\node at (0.35,-0.45) {\scriptsize{11}} ; 
\node at (1.65,-0.6) {\scriptsize{4}} ; 

\node at (2.5,1.2) {\scriptsize{12}} ; 
\node at (3.5,0.8) {\scriptsize{5}} ;

\node at (2.5,0.2) {\scriptsize{13}} ; 
\node at (3.5,-0.2) {\scriptsize{6}} ; 

\node at (2.5,-0.8) {\scriptsize{14}} ; 
\node at (3.5,-1.2) {\scriptsize{7}} ;

\node[main node] at (4,1) (21) {5} ;
\node[main node] at (4,0) (22) {5} ;
\node[main node] at (4,-1 ) (23) {5} ;
\node[main node] at (4,2 ) (24) {2} ;
\node[main node] at (4,3 ) (25) {3} ;
\node[main node] at (4,4 ) (26) {4} ;

\draw[<->,thick] (21) to (12) ; 
\draw[<->,thick] (22) to (13) ; 
\draw[<->,thick] (23) to (14) ; 
\draw[<->,thick] (15) to (24) ; 
\draw[<->,thick] (15) to (25) ; 
\draw[<->,thick] (15) to (26) ;

\draw[dashed,thick] (21) to (5,1.3) ; 
\draw[dashed,thick] (21) to (5,1) ;
\draw[dashed,thick] (21) to (5,0.7) ;

\draw[dashed,thick] (22) to (5,0.3) ; 
\draw[dashed,thick] (22) to (5,0) ;
\draw[dashed,thick] (22) to (5,-0.3) ;

\draw[dashed,thick] (23) to (5,-1.3) ; 
\draw[dashed,thick] (23) to (5,-1) ;
\draw[dashed,thick] (23) to (5,-0.7) ;

\draw[dashed,thick] (24) to (5,2) ; 
\draw[dashed,thick] (25) to (5,3) ;
\draw[dashed,thick] (26) to (5,4) ;

\end{scope}

\end{tikzpicture}
}

\caption{Left: an undirected base graph with $r=5$ vertices and $q = d/2 = 7$ edges, with $q$ as in Definition \ref{def:ps}. We have $a_1  = E_{15}$, $a_2 = E_{12}$, $a_3 = E_{13}$, $a_4 = E_{14}$, $a_5 = E_{25}$, $a_6 = E_{35}$, $a_7 = E_{45}$ and for $8 \leq i \leq 14$, $a_{i} =  a^*_{i-7} = a^*_{i^*}$. The  subscripts on the arrows are the index of the corresponding $i \in [d]$ (there are not all represented for the sake of readability). Middle: neighborhood in the $n$-lift of a vertex $(x,1)$. Right: picture of the common universal covering tree of the base graph and the $n$-lifts.} \label{fig:nlift}

\end{center}\end{figure}

For random symmetric permutations, the $n$-lift is random. Since the work of Amit and Linial 
\cite{MR1883559,MR2216470} and Friedman \cite{MR1978881}, this class of random graphs has attracted a 
substantial attention \cite{MR2674623,AB-S,MR2799807,MR3385636,MR2799213,FrKo}.  The Alon-Boppana 
lower bounds asserts that for any $\veps >0$, for all $n$ large enough and any permutations 
$(\sigma_i), i \in [d]$, in $\cS_n$, with  the symmetry condition \eqref{eq:sym}, we have
\begin{equation}\label{eq:AB}
s ( A_{|H_0} ) \geq s(A_\star) - \veps,
\end{equation}
(due in this context to Greenberg \cite{greenberg}).
Then,  Theorem \ref{th:main} proves that random $n$-lifts achieve the Alon-Boppana lower bound \eqref{eq:AB} 
up to a vanishing term. It follows that $A$ has up to vanishing terms the largest possible spectral gap (the difference 
between the largest eigenvalue and the second largest). It settles the conjecture in Friedman \cite{MR1978881}, and proves an even stronger statement, i.e. \emph{all} eigenvalues of $A_{|H_0}$ are $\veps$-close to the spectrum of $A_\star$, see Figure \ref{fig:nlift2} for a numerical illustration. In some 
cases, it was already proved in Friedman \cite{MR2437174} ($r=1$, $a_i =1$), Friedman and Kohler \cite{FrKo} and 
Bordenave \cite{bordenaveCAT} ($r \geq 1$, $a_i = E_{u_i v_i}$, $A_1 = \sum_i a_i$ constant row sum) and, up to a multiplicative 
factor, in Puder \cite{MR3385636}.

\begin{figure}[htb]
\begin{center}  
\includegraphics[height = 6cm,width = 12cm]{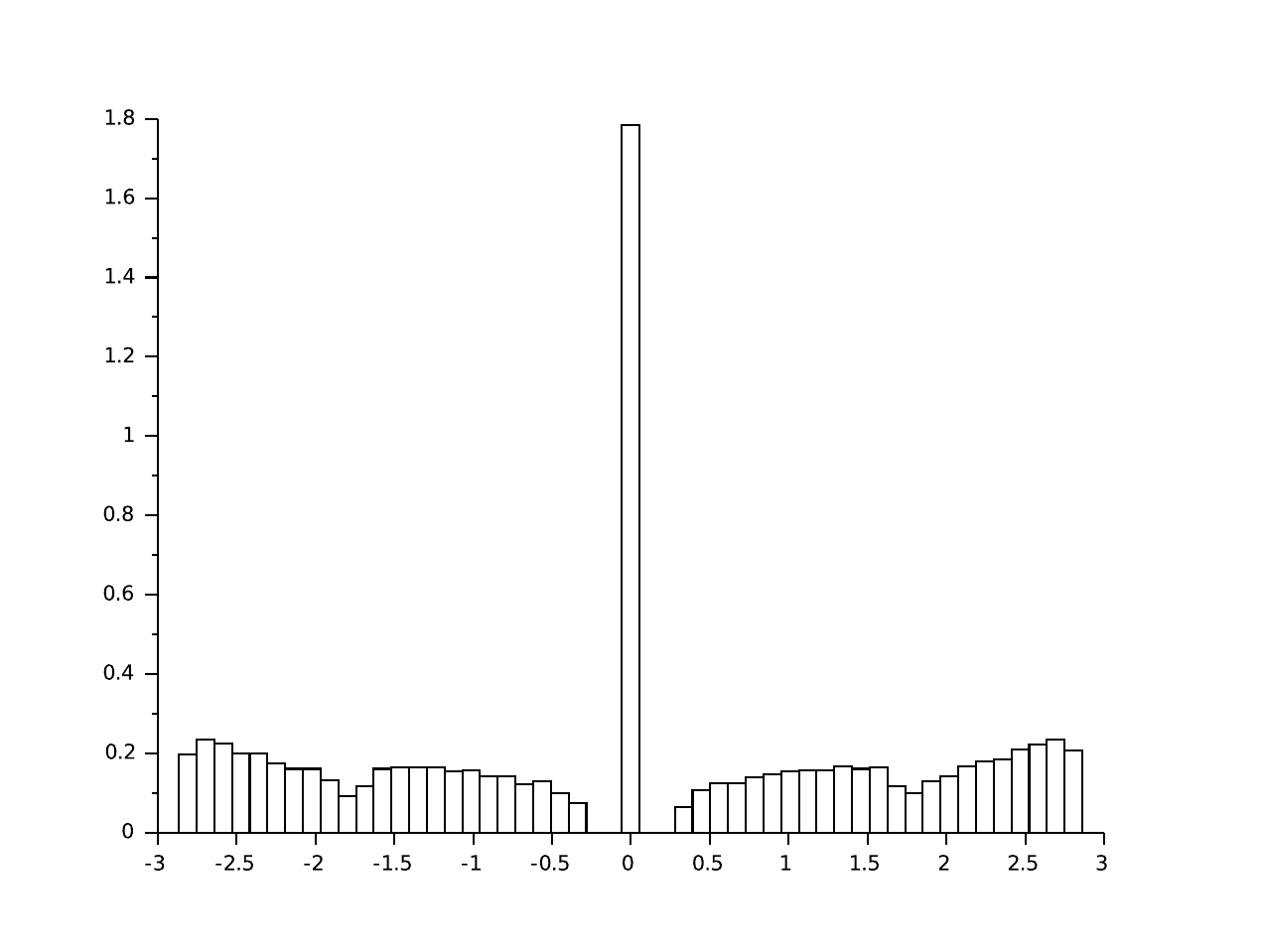} 
\caption{Histogram of the eigenvalues of $A_{|H_0}$ for $n=500$ and a random $n$-lift of the base graph depicted in Figure \ref{fig:nlift}. The spectrum of $A_\star$ is the spectrum of the universal covering tree of the base graph: we have $\sigma(A_\star) = [-a,-b] \cup \{0\} \cup [b,a]$ with $a\simeq 2.866$ and $b \simeq 0.283$.} \label{fig:nlift2}
\end{center}\end{figure}

Now, consider the case, $r=1$, $a_i \geq 0$ and $\sum_i a_i = 1$, then $A$ is the Markov transition matrix of 
an {\em anisotropic random walk}. For the properties of this random walk on the free group, see the monograph by 
Fig\`a-Talamanca and Steger \cite{MR1219707}. More generally, for any integer $r \geq 1$, if $A_1  = \sum_i a_i$ is a stochastic matrix, 
then $A$ is also a stochastic matrix which can be interpreted as a Markov chain on the $n$-lift graph. The Alon-Boppana 
lower bounds \eqref{eq:AB} holds also in this context, see Ceccherini-Silberstein, Scarabotti and 
Tolli \cite{CECCHERINISILBERSTEIN2004735}. Thus, Theorem \ref{th:main} asserts that $A$ has, up to vanishing terms, 
the largest possible spectral gap. Interestingly, our argument will actually show that \eqref{eq:upboundA} is achieved 
with probability tending to one, jointly for all $a_i$ with $\max \| a_i \| \leq 1$.

In the same vein, assume that all $a_i$, $i \in [d]$, have non-negative entries and $a_0  = -\sum_i a_i$. Then $A$ is a 
Laplacian matrix and it is the infinitesimal generator of a continuous time random walk on the $n$-lift. Theorem 
\ref{th:main} proves again that, up to vanishing terms,  random permutations maximize the spectral gap of such processes.

\subsection{Tensor product of random permutation matrices}

We now discuss an extension of our work which is notably relevant in the context of quantum expanders, 
see  \cite{PhysRevA.76.032315,Hastings:2009:CQT:2011781.2011790} and in cryptography \cite{Friedman1996}. 
Let $X$ be a finite set and $r,d$ positive integers. Let $\ell^2 (X ^2) $ be the Hilbert space spanned by an orthonormal basis 
$\delta_{(x,y)}$ indexed by the elements of $(x,y) \in X^2 = X \times X$. 
 We consider $(a_0, \ldots, a_d)$  elements in $M_r (\dC)$ and let 
$\sigma_{i}\in S(X)$ be a permutation of the set $X$ whose corresponding unitary operator on $\ell^2 (X)$ is $S_i$.  
We are now interested in the operator on $\dC^r \otimes \ell^2 (X^2)$, 
\begin{equation} \label{eq:defAT}
A^{(2)} = a_0 \otimes 1 \otimes 1 + \sum_{i=1}^d a_i \otimes S_i \otimes S_i.
\end{equation}
Note that \eqref{eq:defAT} is again an operator of the form \eqref{eq:defA} since $S_i \otimes S_i$ is a unitary operator associated 
to the permutation on $\ell^2 (X^2)$ defined by, for all $(x,y) \in X^2$, $\sigma^{(2)}_i ( x,y)  = (\sigma_i (x), \sigma_i (y))$.  
Note also that we may identify $\ell^2 (X^2)$ with linear operators on $\ell^2(X)$ endowed with the Hilbert-Schmidt scalar 
product $\langle a, b \rangle = \Tr ( a^* b)$. Then $S_i \otimes S_i$ is identified with the linear map $T \mapsto S_i T S^*_i$.

We consider the following orthogonal elements of $\ell^2 (X^2)$, defined in coordinates, for all $(x,y) \in X^2$,
\begin{equation}\label{eq:defIJ}
J_{xy} =  \IND (x \ne y) \AND I_{xy} = \IND(x = y).
\end{equation}
It is immediate to check that for any permutation operator $S$  on $\ell^2 (X)$
$$
( S \otimes S ) (J ) = J \AND  ( S \otimes S ) (I ) = I.
$$
Let $V$ be the vector space spanned by $I$ and $J$. We introduce the vector space of codimension $2r$, 
$$H^{(2)}_0 = \dC^r \otimes V^\perp.$$ 

For random permutations, we have the following analog to Theorem \ref{th:main}. 

\begin{theorem}\label{th:mainT}
Theorem \ref{th:main} holds with $A$ replaced by $A^{(2)}$ defined by \eqref{eq:defAT} and $H_0$ replaced by $H_0^{(2)}$.
\end{theorem}

In Section \ref{sec:mainT}, we will explain how to adapt the proof of Theorem \ref{th:main} to deal with tensor products. 
Interestingly, the analog of Theorem \ref{th:mainT} for random unitaries is not known and it cannot be deduced 
from  \cite{MR3205602,MR2183281}. As corollary of Theorem \ref{th:mainT}, we have the following

\begin{theorem}\label{thm:strongT}
For symmetric random permutations, the permutation operators restricted to 
$V^\perp$, $((S_i \otimes S_i)_{|V^\perp}), i \in [d],$ are asymptotically strongly free in probability.\end{theorem}

Note that asymptotic freeness follows, among others, from \cite{MR3573218}.
In particular, when one restricts oneself to sum of generators $\sum_i S_i \otimes S_i$, one obtains that the family 
$S_i\otimes S_i$ viewed as an operator on $V^\perp$ is a nearly optimal quantum expander in the sense of 
Hastings \cite{MR2486279} and Pisier \cite{MR3226740}.

\subsection{Brief overview}

The proof of Theorem \ref{th:main} is divided into two parts. For any $\veps >0$, we will prove that with probability tending 
to one:  
\begin{enumerate}
\item
The spectrum of $A_\star$ is contained in an arbitrarily small neighbourhood of the spectrum of $A_{|H_0}$:
\begin{equation}\label{eq:lowboundA}
\sigma( A_\star ) \subset \sigma (A_{|H_0}) + [-\veps , \veps],
\end{equation}
and 
\item  
The spectrum of $A_{|H_0}$ is contained in an arbitrarily small neighbourhood of the spectrum of $A_\star$:
\begin{equation}\label{eq:upboundA}
 \sigma (A_{|H_0})\subset  \sigma( A_\star )  + [-\veps , \veps].
\end{equation}
\end{enumerate}
The proof of the spectrum inclusion \eqref{eq:lowboundA} is standard and follows from the asymptotic freeness of  
independent permutations, which follows from the already mentioned reference, Nica \cite{MR1197059}.
However we give an alternative argument by supplying a general deterministic criterion which guarantees 
that \eqref{eq:lowboundA} holds and we prove that the symmetric random permutations meet this criterion.

The proof of \eqref{eq:upboundA} is much more involved and it is the main contribution of this work. 
It relies on a novel use of {\em non-backtracking operators}. These operators are defined on an enlarged vector 
space and, despite the fact that they are non-normal, they are much easier to work with. 
Indeed, they have a very simple form on the free product of groups
 $X_\star$. Notably, in Theorem \ref{prop:edgeAB}, 
we  prove that the set $\sigma(A) \backslash \sigma(A_\star)$ is controlled by the spectral radii of a one parameter 
family of non-backtracking operators. This will allow us to reduce the proof of Theorem \ref{th:main} to the proof of 
Theorem \ref{th:mainB} which is an analogous statement for non-backtracking operators. 
Then, the proof of Theorem \ref{th:mainB} follows a strategy similar to the new proof of Friedman's Theorem 
in \cite{bordenaveCAT} but with non-negligible refinements.  

Indeed, the main technical novelty will be the presence of matrix-valued weights $a_i$, $i \in [d]$. 
In particular, we are able to relate directly the expectation of the trace of a power of a non-backtracking matrix on $[n]$ 
to powers of the corresponding non-backtracking operator on the free group 
$X_\star$ (forthcoming Lemma \ref{le:sumpath}).
Another important issue will be that we will need a refined net argument 
to control jointly the spectral radii of the non-backtracking matrices associated to all possible weights $a_i$ with 
uniformly bounded norms. Due to the non-normality of non-backtracking matrices, we have to deal with the bad 
regularity of spectral radii in terms of matrix entries. See Remark \ref{rq:CAT} for a more precise comparison with previous works.

The remainder of this manuscript is organized as follows. In section  \ref{sec:AB}, we prove the spectrum inclusion 
\eqref{eq:lowboundA}. In Section \ref{sec:nonback}, we introduce non-backtracking operators and we reduce the proof 
of Theorem \ref{th:main} to Theorem \ref{th:mainB} on non-backtracking operators. In section \ref{sec:mainB}, we prove 
this last theorem on non-backtracking operators for random permutation matrices. In Section \ref{sec:mainT}, 
we adapt the previous arguments to prove Theorem \ref{th:mainT}. The proof of Theorem \ref{thm:strong} and 
Theorem \ref{thm:strongT} is contained in Section \ref{sec:cor}. It is based on the   linearization trick.
It was developed simultaneously in various areas of mathematics, and applied in operator algebras by Haagerup and 
Thorbj\o rnsen 
and subsequently improved by Anderson \cite{MR3098069}.
For a synthetic introduction we will refer to Mingo and Speicher \cite[p256]{MR3585560}. Finally, proofs of auxiliary 
results are gathered in Section \ref{sec:aux}.

\paragraph{Notation} 
We  use the usual notation $o(\cdot)$ and $O(\cdot)$. 
We  denote by $\dP( \cdot) $ and $\dE( \cdot) $ 
the corresponding probability measures on $\cS_n^d$, 
corresponding to the Definition of Equation \eqref{eq:defA}. Note that $\dP(\cdot)$ and $\dE (\cdot )$ depend implicitly on $n$.  
The coordinate vector at $x \in X$ is denoted $\delta_x$.  It will be convenient to describe our operators as 
matrix-valued operators. For an operator $M$ on $\dC^r \otimes \ell^2 (X)$,  we set for all $x,y \in X$ 
\begin{equation}\label{eq:matrixvalued}
M_{xy} = (1 \otimes \langle \delta_x , \, \cdot \, \delta_y \rangle) (M) \in M_r (\dC).
\end{equation}
In other words, we may see $M$ as an infinite block matrix indexed by $X\times X$ (of matrices in $M_r (\dC)$), and we may 
reformulate the above equation as $M=(M_{xy})_{(x,y)\in X\times X}$.
Finally, we will use the convention that a product over an empty set is equal to $1$ and the sum over an empty set is $0$.  

\paragraph{Acknowledgements} CB was supported by ANR-16-CE40-0024-01 and ANR-14-CE25-0014. He would like to thank Doron Puder 
for early discussions on this problem. BC was supported by JSPS KAKENHI 17K18734, 17H04823, 15KK0162 
and ANR-14-CE25-0003. 
He would like to thank Mikael de la Salle, Camille Male and Amir Dembo for 
enriching discussions on random permutations and free probability.  Both authors are grateful to Gilles Pisier for 
discussions during the finalization phase of the paper, including useful 
references and questions that encouraged us to write Theorem \ref{thm:strongT}. The final version of this paper was completed during the visit of CB to Kyoto University under a JSPS short term
professorship. Both authors gratefully acknowledge the support of JSPS and Kyoto University.  The authors are indebted to Yu Shang-Chun, Ryan O'Donnell and Xinyu Wu for pointing typos and inaccuracies in their careful reading. Finally, the authors would like to thank an anonymous referee for an extremely detailed and insightful report that allowed them to improve
greatly their initial manuscript. 

\section{Inclusion of the spectrum of $A_\star$}
\label{sec:AB}

Assume that $X = [n]$ and that the symmetry condition \eqref{eq:sym} holds. We start by some elementary definitions 
from graph theory and define a natural colored graph $G^\sigma$ associated to the permutations $\sigma_i$, $i \in [d]$. 
\begin{definition}\label{def0}
\begin{enumerate}[-]
\item
A {\em colored edge} $[x,i,y]$ is an equivalence class of  triplets $(x,i,y) \in [n] \times [d] \times [n]$ endowed with 
the equivalence $(x,i,y) \sim (x',i',y')$ if $(x',i',y') \in \{ (x,i,y), (y,i^*,x) \}$. A {\em colored graph} $H$ is a graph whose 
vertices is a subset of $[n]$ and whose edges are a set of colored edges. 

\item
A {\em path} in $H$ of length $k$  from $x$ to $y$ is a sequence $(x_1, i_1, \ldots, x_{k+1})$ such that 
$x_1 = x$, $ x_{k+1} = y$  and $[x_t, i_t , x_{t+1}] $ is an edge of $H$ for any $1 \leq t \leq k$. 
The path is {\em closed} if $x_1 = x_{k+1}$. A {\em cyclic path} in $H$ is a closed path in $H$ such that 
$(x_1, \ldots, x_{k})$ are pairwise distinct. A {\em cycle}
is the colored graph spanned by a cyclic path (that is, the vertex set is $\{x_t: 1 \leq t \leq k\}$ and the edge set is 
$\{[x_t,i_t,x_{t+1}]: 1 \leq t \leq k\}$).  
\item
If $x \in [n]$ and $h$ is a non-negative integer, $(H,x)_h$ denotes the subgraph of $H$ spanned by all edges belonging to a path starting from $x$ and of length at most $h$.

\item
$G^\sigma$ is the colored graph whose vertex set is $[n]$ and whose edges are the set of $[x,i,y]$ such that 
$\sigma_i (x) = y$ (and $\sigma_{i^*} (y) =x$).  
\end{enumerate}
\end{definition}

The inclusion \eqref{eq:lowboundA} is  a direct consequence of the following deterministic proposition whose 
proof can be found in Section \ref{sec:aux} for completeness.

\begin{proposition}\label{prop:inclusion}
Assume that $X = [n]$ and that the symmetry condition \eqref{eq:sym} holds. Let  $A$ and $A_\star$ be defined 
by \eqref{eq:defA} and \eqref{eq:defAfree}.   
For any $\veps > 0$, there exists an integer $h \geq 1$ such that if $(G^\sigma,x)_{h}$ contains no cycle for some 
$x \in [n]$ then
$
\sigma( A_\star ) \subset \sigma (A_{|H_0}) + [-\veps , \veps].
$ 
\end{proposition}

As a corollary, we obtain the first half of Theorem \ref{th:main}. 

\begin{corollary}\label{cor:inclusion}
Let $A$ and $A_\star$ be as in \eqref{eq:defA} with the symmetry condition \eqref{eq:sym}. For symmetric 
random permutations, for any $\veps >0$, with probability tending to one as $n$ goes to infinity, $
\sigma( A_\star ) \subset \sigma (A_{|H_0}) + [-\veps , \veps].
$ 
\end{corollary}

\begin{proof}
If $x \in [n]$ is such that $(G^\sigma, x)_h$ 
 contains a cycle, we claim that there exists a cycle of length $k \leq 2h$ contained in  $(G^\sigma, x)_h$. Indeed, assume that 
 $(G^\sigma, x)_h$ contains a cycle and let $y$ be a vertex on this cycle which is at maximal distance $t \leq h$ from $x$ 
 (where the distance is the minimal $t$ such that  there exists  a path of length $t$ from $x$ to $y$). Fix a path $\gamma$ from 
 $x$ to $y$ of length $t$. At least one of two neighboring edge of $y$ on the cycle is not on $\gamma$, we call it $[y,i,y']$. 
 Let $\gamma'$ be a path from $y'$ to $x$ of minimal length $t'\leq t$. By construction $[y,i,y']$ is not an edge of $\gamma'$. 
 It follows that the sequence $(\gamma,i,\gamma')$ forms a closed path in $(G^\sigma, x)_h$ which contains a cycle.  It is of length 
 at most $t + t' + 1 \leq 2h$ (if $t = h$ then $t' < t$ by the definition of $(G^\sigma, x)_h$).

Now, for an integer $k\geq 1$, let $N_k$ be the number of distinct cycles of length $k$ in $G^\sigma$. In the forthcoming 
 Lemma \ref{le:tanglefree}, we will check that the expectation of $N_k$ is $O ( (d-1)^k)$. Also, for a given cycle in 
 $G^\sigma$ of length $k \leq 2 h $, there are at most $d (d-1)^{h-1}$ vertices $x\in [n]$ such that $(G^\sigma, x)_h$ 
 contains this cycle.   It follows that the expected number of $x \in [n]$ such that $(G^\sigma, x)_h$ 
 contains a cycle is upper bounded by 
 $$
 \sum_{k=1}^{2h}d (d-1)^{h-1}  \dE N_k  =  O ( (d-1)^{3h} ),
 $$
(this bound is very rough). Hence, by Markov inequality, for any $h \leq (\log n )  / \kappa$ with $\kappa > 3 \log (d-1)$, 
there exists with probability tending to one, a vertex $x \in [n]$ such that $(G^\sigma, x)_h$ contains no cycle. 
On the latter event, we may conclude by applying Proposition \ref{prop:inclusion}.
\end{proof}

\section{Non-backtracking operator}

\label{sec:nonback}

\subsection{Spectral mapping formulas}

Let $A$ be as in \eqref{eq:defA}. We consider a vector space $ U$ of finite codimension in $\ell^2 (X)$ which is left invariant by 
all permutation operators $S_i$, $i \in [d]$: $S_i U = U$. We set  
$$
H = \dC ^r \otimes U.
$$
This vector space $H$ is left invariant by $A$. 

We define $E = X \times [d]$. If  $A$ is thought of as a weighted adjacency operator of a 
directed graph on the vertex set $X$, an element $(x,i)$ of $E$ can be thought as a directed edge from $x$ to 
$\sigma_i(x)$ with weight $a_i$ (where all vertices have a loop edge of weight $a_0$).  The non-backtracking operator $B$ associated 
to $A$ is the operator on 
$\dC^r \otimes \ell^{2} (E) = \dC^r \otimes \ell^2 (X) \otimes \dC^d$ defined by
\begin{equation}
\label{eq:defB}
B = \sum_{j \ne i^*} a_j \otimes S_i \otimes E_{ij},
\end{equation}
where  $E_{ij} \in M_d( \dR)$ is the  matrix defined, for all $j,k \in [d]$ by $(E_{ij} )_{k l } =  \IND_{( i ,j)  = (k,l) }.$ Equivalently, 
writing $B$ as a matrix-valued operator on $\ell^ 2 (E)$, we  have for any $e , f \in  E$ and 
$e = (x,i) \in E$, $f = (y,j)$,
$$
B_{e f}  = a_j \IND( \sigma_i(x) =  y) \IND ( j \ne i^*).
$$
See Figure \ref{fig:AvsB} for an informal illustration of the operator $A$ and its non-backtracking operator $B$. 

Note that $B$ does not depend on the matrix element $a_0$. We set 
$$
K =  \dC ^r \otimes U  \otimes \dC^d.
$$We observe from \eqref{eq:defB} that $B $ defines an operator on $K$, that is 
$BK\subset K$.  As  before, we denote by $B_{|K}$ the restriction of  
$B$ to  $K$.

\begin{figure}[htb]
\begin{center}  
\resizebox{12cm}{!}{
\begin{tikzpicture}[main/.style={circle, draw , fill = black!15, text = black, thick},high/.style={circle, draw , fill = red!90, text = black, thick}]

\tikzset{->-/.style={decoration={
  markings,
  mark=at position .5 with {\arrow{>}}},postaction={decorate}}}

\draw[fill = cyan] (0,0) circle(0.1)  ; 
\draw[fill = cyan] (0,1) circle(0.1)  ; 
\draw[fill = black!20] (1,0) circle(0.1)  ; 
\draw[fill = black!20] (2,0) circle(0.1)  ; 
\draw[fill = black!20] (1,1) circle(0.1)  ;

\draw[red!100,thick,->-] [out = -30 , in = -150] (0,0) to (1,0)  ;  
\draw[red!88,thick,->-] [out = 30 , in = 150] (1,0) to (2,0)  ;  
\draw[red!76,thick,->-] [out = -150, in = -30] (2,0) to (1,0)  ;  
\draw[red!64,thick,->-] [out = 120 , in = -120] (1,0) to (1,1)  ;  
\draw[red!52,thick,->-] [out = -60 , in = 60] (1,1) to (1,0)  ;  
\draw[red!40,thick,->-] [out = 150 , in = -150] (1,0) to (1,1)  ;  
\draw[red!28,thick,->-] [out = 150 , in = 30] (1,1) to (0,1)  ;

\node at (0,-0.3) {\scriptsize{$x$}} ; 
\node at (0,1.3) {\scriptsize{$y$}} ; 
\node at (1,-0.3) {\scriptsize{$x_2$}} ; 
\node at (2,-0.3) {\scriptsize{$x_3$}} ; 
\node at (1,1.3) {\scriptsize{$x_5$}} ; 
\node at (0,-0.3) {\scriptsize{$x$}} ; 
\node at (0.4,-0.3) {\scriptsize{$i_1$}} ; 
\node at (1.6,-0.3) {\scriptsize{$i_2^*$}} ; 
\node at (1.4,0.3) {\scriptsize{$i_2$}} ;
\node at (0.7,0.4) {\scriptsize{$i_4$}} ;
\node at (1.3,0.6) {\scriptsize{$i_4^*$}} ; 
\node at (0.6,1.3) {\scriptsize{$i_7$}} ;

\node at (1,-0.8) {\small{$A$}} ;

    \coordinate (shift) at (4,0);
    \begin{scope}[shift=(shift)]

\draw[fill = black!20] (0,0) circle(0.1)  ; 
\draw[fill = black!20] (0,1) circle(0.1)  ; 
\draw[fill = black!20] (1,0) circle(0.1)  ; 
\draw[fill = black!20] (2,0) circle(0.1)  ; 
\draw[fill = black!20] (2,1) circle(0.1)  ; 
\draw[fill = black!20] (1,1) circle(0.1)  ; 

\draw[cyan,->,ultra thick] (0,0) to (0.5,0)  ; 
\draw[cyan,->,ultra thick] (0,1) to (-0.5,1)  ;

\draw[red!100,thick,->-] [out = -30 , in = -150] (0,0) to (1,0)  ;  
\draw[red!88,thick,->-] [out = 120 , in = -120] (1,0) to (1,1)  ;  
\draw[red!76,thick,->-] [out = 30, in = 150] (1,1) to (2,1)  ;  
\draw[red!64,thick,->-] [out = -60 , in = 60] (2,1) to (2,0)  ;  
\draw[red!52,thick,->-] [out = -150 , in = -30] (2,0) to (1,0)  ;  
\draw[red!28,thick,->-] [out = 150 , in = 30] (1,1) to (0,1)  ;  
\draw[red!40,thick,->-] [out = 150 , in = -150] (1,0) to (1,1)  ;  

\node at (0,-0.3) {\scriptsize{$x$}} ; 
\node at (0,1.3) {\scriptsize{$y$}} ; 
\node at (1,-0.3) {\scriptsize{$x_2$}} ; 
\node at (2,-0.3) {\scriptsize{$x_5$}} ; 
\node at (1,1.3) {\scriptsize{$x_3$}} ; 
\node at (2,1.3) {\scriptsize{$x_4$}} ; 
\node at (0,-0.3) {\scriptsize{$x$}} ; 
\node at (0.4,-0.3) {\scriptsize{$i_1$}} ; 
\node at (1.6,-0.3) {\scriptsize{$i_5$}} ; 
\node at (0.7,0.4) {\scriptsize{$i_2$}} ;
\node at (0.6,1.3) {\scriptsize{$i_7$}} ; 
\node at (1.4,1.3) {\scriptsize{$i_3$}} ;
\node at (2.3,0.6) {\scriptsize{$i_4$}} ;

\node at (1,-0.8) {\small{$B$}} ;

\end{scope}

\end{tikzpicture}
} 
\end{center}
\caption{Left: For $x,y \in X$ and $k$ integer, $(A^k)_{xy} \in M_r(\dC)$ is the sum of all weighted paths $(x_1,i_1,x_2,\cdots, i_k,x_{k+1})$ 
in $G^\sigma$ (as per Definition \ref{def0} with extra loop edges of weight $a_0$  at all vertices) from $x=x_1$ to $y=x_{k+1}$  with 
weight $\prod_{t=1}^{k} a_{i_t}$. In the example, we have $k= 7$ and the path is such that $i_3 = i_2^*$, $i_5 = i_4^*$, $i_6 = i_4$. 
Right: For $e = (x,i), f = (y,j) \in E = X \times [d]$ and $k$ integer, $(B^k)_{ef} \in M_r(\dC)$ is the sum of all weighted paths 
$(x_1,i_1,x_2,\cdots, i_k,x_{k+1})$ in $G^\sigma$ with $ (x_1,i_1) = e$, $(x_{k+1},i_{k+1})=  f$ where  $i_{k+1} = j$, and for all 
$1 \leq t \leq k$, $i_{t+1} \ne i_t^*$. The weight of the path is $\prod_{t=1}^{k} a_{i_{t+1}}$.  In the example, we have $k= 7$ and $i_6 = i_2$. 
The condition $i_{t+1} \ne i_t^*$ is viewed as a non-backtracking constraint of the path. If $X = X_\star$  with generators 
$(g_i)_{i \in [d]}$ as defined above \eqref{eq:defAfree} and $\sigma_i = \lambda(g_i)$, the condition $i_{t+1} \ne i_t^*$ asserts that 
$g_{i_1} \cdots g_{i_{k+1}}$ is in reduced form.} \label{fig:AvsB}
\end{figure}

The following statement relates the spectrum of $B$ with the spectrum of an operator of the same type as $A$. 
In the scalar case $r = 1$, the next proposition is contained in Watanabe and Fukumizu \cite{NIPS2009_0420}.
\begin{proposition}\label{prop:nonback}
Let $A$ be as in \eqref{eq:defA} with associated non-backtracking operator $B$ and let  $\lambda \in \dC$ satisfy 
$\lambda^2  \notin  \{ \si(  a_i a_{i^*} ): i \in [d] \}$. Define the operator $A_\lambda$ on $\dC^r \otimes \ell^2(X)$  through
\begin{equation} \label{def_AM}
A_\lambda  = a_0 ( \lambda) +  \sum_{i=1}^d a_i (\lambda) \otimes S_i \,, \qquad a_i (\lambda) = \lambda  a_{i} ( \lambda^2 - a_{i^*} a_{i} )^{-1}   \AND a_0(\lambda) =  - 1 -  \sum_{i=1}^d  a_{i}(  \lambda^2 - a_{i^*} a_i )^{-1} a_{i^*}  .
\end{equation} 
Then $\lambda \in \sigma (B)$ if and only if  $0 \in \sigma ( A_\lambda)$. Moreover, $\lambda \in \sigma (B_{|K})$ if and only if  $0 \in \sigma  ( (A_\lambda)_{|H} )$.
\end{proposition}

\begin{proof}[Proof of Proposition \ref{prop:nonback}] 
We first assume $\lambda$ is in the discrete spectrum of $B$. We show that $0$ is in the discrete spectrum of $A_\lambda$. 
Then, there is an eigenvector $v  \in  \dC^r \otimes \ell^2 (E)$ such that $B v =  \lambda v$, which reads in the coordinates of 
$\ell^2 (E)$, for all $e = (x,i) \in E$
\begin{equation}\label{eq:eigH}
 \lambda  v(x,i) =  \sum_{j \ne i^*} a_{j}  v(\sigma_i (x), j),
\end{equation}
with $v(e) \in \dC^r$. 
We define $u \in\dC^r \otimes \ell^2 (X)$ by, for each $x \in X$,  
 \begin{equation}\label{eq:ydefx}
 u(x) =  \sum_j a_j  v(x, j)\,. 
 \end{equation}
(If elements $(x,i)$ of $E$ are thought as derivatives at $x$ in a discrete direction $i$, the vector $u$ can be thought as a divergence vector). 
The eigenvalue equation \eqref{eq:eigH} reads
 $$
 \lambda v(x,i) = u(\sigma_i(x)) - a_{i^*} v(\sigma_i (x), i^*)\,. 
 $$
Applying the above identity to $e = ( \sigma_ {i} (x) , i^*) =  ( \sigma_ {i_*}^{-1} (x) , i^*) $, we find  
 $$
  \lambda v(\sigma_{i} (x),i^*) = u (x )  - a_i v(x,i)\,.
 $$
We deduce
$$
 \lambda^2  v(x,i)  =   \lambda u(\sigma_i(x)) -a_{i^*}  u ( x ) +  a_{i^*} a_i v(x,i)\,.
 $$
By assumption, $\lambda^2   - a _{i^*} a_i$  is invertible. Hence, 
\begin{equation} \label{x_defy}
 v(x,i) =  ( \lambda^2   - a _{i^*} a_i )^{-1} ( \lambda u( \sigma_i(x))-a_{i^*}  u ( x) ) \,. 
\end{equation}
Let us note that Equations 
\eqref{eq:ydefx} and \eqref{x_defy}, when restricted, define a map between $H_0$ and $K_0$.
We see from this last expression that $u \ne 0$ if $v \ne 0$. We now check that $u$ is in the kernel of $A_\lambda$. 
Let $y \in X$, $i \in [d]$ and set $x  = \sigma_i ^{-1} (y) = \sigma_{i^*} (y)$. We plug \eqref{x_defy} into \eqref{eq:eigH} and get
 \begin{align*}
& \lambda^2  ( \lambda^2   - a _{i^*} a_i )^{-1} u(y)   -  \lambda( \lambda^2   - a _{i^*} a_i )^{-1}  a_{i^*}  u (x)  \\
& =  \sum_{j \ne i^*} \lambda  a_{j} ( \lambda^2 - a_{j^*} a_{j} )^{-1} u(\sigma_j(y))  - \sum_{j \ne i^*} a_{j}(  \lambda^2 - a_{j^*} a_j )^{-1} a_{j^*} u(y) .
\end{align*}

Since $\sigma_{i^*} (y) = x$, we find 
  $$
\lambda  ( \lambda^2   - a _{i} a_{i^*} )^{-1} a_i u(y)  =  \sum_{j} \lambda  a_{j} ( \lambda^2 - a_{j^*} a_{j} )^{-1} u(\sigma_j(y))  - \sum_{j \ne i^*}  a_{j}(  \lambda^2 - a_{j^*} a_j )^{-1} a_{j^*} u(y) .$$
 Since $1=
\lambda^2   \PAR{ \lambda^2 -  a_{i^*} a_i }^{-1}    -a_{i^*}   \PAR{ \lambda^2 - a_i a_{i^*}  }^{-1} a_i $, we conclude that
 $$
\PAR{ 1 + \sum_j a_{j}(  \lambda^2 - a_{j^*} a_j  )^{-1} a_{j^*}  }u(y)  = \sum_{j} \lambda a_j ( \lambda^2 - a_{j^*} a_{j} )^{-1} u(\sigma_j (y))\,,
 $$
which proves that $u$ is the kernel of $A_\lambda$.

Conversely, if $0$ is in the discrete spectrum of $A_\lambda $ with eigenvector $u$, we define $v$ through \eqref{x_defy} 
(note that $v\ne 0$ because of \eqref{eq:ydefx}). Then the above computation also implies that $v$ satisfies \eqref{eq:eigH}, 
i.e.\ $Bv = \lambda v$, so that $\lambda$ is in the discrete spectrum of $B$. Note also
that $u \in H_0$ if and only if $v \in K_0$. 

Finally, if $\lambda$ is in the essential spectrum of $B$, then, for any $\veps >0$, there exists 
$v \in   \dC^r \times \ell^2 (E)$ such that $\| v \|_2 = 1$ and $\| B v - \lambda  v \|_2 \leq \veps$. The above argument 
shows that $\| A_ \lambda u \|_2 = O ( \veps)$ and \eqref{x_defy} implies that $\| v \|_2 = O ( \| u \|_2)$. It implies that $0$ 
is in the spectrum of $A_\lambda$. Conversely, if $0$ is in the essential spectrum of $A_\lambda$ then $\lambda$ is in 
the spectrum of $B$. Finally, the spectrum of $B$ is contained in the union of the essential spectra of $B$ and $B^*$. The above arguments can also be applied directly to $B^* = \sum_{j \ne i^*} a^*_j \otimes S_i^{-1} \otimes E_{ji}$.  
\end{proof}

This proposition could be used for studying the spectrum of the operator $A$. To this end, we should for a given $A$ and 
$\mu$ find a corresponding $B_\mu$ such that $\mu$ is the spectrum of $A$ if and only if $\lambda = 1$ is in the spectrum of 
$B_\mu$.  Assume that there are $q$ pairs $\{i,i^*\}$ such $i \ne i^*$ and $p$ elements of $[d]$ such that $i = i^*$, with 
$2 q + p  =d $. For concreteness, as in Definition \ref{def:ps}, we may assume without loss of generality that for 
$1 \leq i \leq q$, $i^* = i + q$, for $q+1 \leq i \leq 2q$, $i^* = i - q$ and for $2q+1 \leq i \leq d$, $i^* = i$.  
Now, let $A_\star$ be as in \eqref{eq:defAfree}. 
We relate the spectra of $A$ and $B$ through the resolvent of $A_\star$. More precisely,  for 
$\mu \notin \sigma( A_\star)$, we set
$$
G(\mu) = ( \mu - A_\star  )^{-1}.
$$

For technical reasons, we further introduce the operator $A_\star^{(o)} $ on $\dC^r \otimes \ell^2( X_\star)$ defined by 
\begin{equation}\label{eq:defAo}
A_\star - A_\star^{(o)} =   \sum_{j} a_j  \otimes \delta_{g_j} \otimes \delta_o + a_{j^*} \otimes \delta_{o} \otimes \delta_{g_j} .
\end{equation}
In words, $A_\star^{(o)}$ is the operator associated to the  Cayley graph of $X_\star$ where the the unit $o$ has been 
isolated. We define $\hat \sigma(A_\star) =  \sigma(A_\star) \cup \sigma(A_\star^{(o)})$. We note that if the symmetry condition \eqref{eq:sym} holds then $\hat \sigma(A_\star)$ is contained in the convex hull of $\sigma(A_\star)$ and $\sigma(A^{(o)}_\star)\backslash  \sigma(A_\star)$ is a finite set of at most $2dr$ points.

For the next proposition and for the sequel, we recall that we use a matrix notation with indices in $X\times X$ for
operators on $\dC^r \otimes \ell^2 (X)$, as per Equation \eqref{eq:matrixvalued}. In particular, $G(\mu) $ defined above 
fits in this context with $X = X_\star$ and it will be written as $G_{xy} (\mu)= (G(\mu) )_{xy}$. We will be interested in the case where
$x$ and $y$ are $o$ (the neutral element of $X_\star$ for its group structure) or $g_i$ (the $i$-th generator of $X_\star$ defined above \eqref{eq:defAfree}).  
In the symmetric and scalar case $r = 1$, the next proposition is a formula derived in 
Anantharaman \cite[Section 7]{MR3649482}.


\begin{proposition}\label{prop:nonback2}
Let $A$ be as in \eqref{eq:defA}  and $\mu \notin  \hat \sigma ( A_\star)$. Define the operator 
$A_\mu$ on $\dC^r \otimes \ell^2(X)$  through
\begin{equation} \label{def_AM}
A_\mu  =   \sum_{i=1}^d \hat a_i (\mu) \otimes S_i \,, \qquad \hbox{with } \quad \hat a_i (\mu) = G_{oo}  ( \mu  )^{-1} G_{ g_i o} ( \mu ). 
\end{equation} 
Let $B_\mu$ be the corresponding non-backtracking operator. Then $\mu \notin \sigma (A)$ if and only if  
$1  \notin \sigma ( B_\mu)$. Moreover, $\mu \notin \sigma (A_{|H})$ if and only if  
$1 \notin \sigma  ( (B_\mu)_{|K} )$.
\end{proposition}

Our first lemma proves that $\hat a_i (\mu)$ is well-defined for $\mu \notin \hat \sigma(A_\star)$ and establishes a classical expression for $\hat a_i (\mu)$ related to the recursive equations satisfied by resolvent operators on trees. 
\begin{lemma}\label{le:resrec}
Let $A_\star$ be as above, $\mu \notin  \hat \sigma ( A_\star)$. Then $G_{oo}(\mu)$ is invertible in $M_r(\dC)$ and
$\hat a_i  (\mu) = G_{oo}  ( \mu  )^{-1} G_{o g_i} ( \mu )$ is well-defined.
Moreover  the following identities hold 
\begin{equation*}
G_{oo}  = \PAR{ \mu I_r - a_0  - \sum_{i} \hat a_i a_{i^*} }^{-1} ,
\end{equation*}
and
$$
a_{i} G_{oo} = \hat a_i (I_r- \hat a_{i^*} \hat a_i)^{-1}.
$$ 
\end{lemma}
\begin{proof}
 We denote by  $G^{(o)}$ the resolvent of $A_\star^{(o)}$ defined in \eqref{eq:defAo}. We set $\gamma_i (\mu) =  G^{(o)}_{g_i g_i}(\mu)$. Then, omitting $\mu$ for ease of notation, the resolvent identity reads 
$$
G = G^{(o)} + G ( A_\star - A_\star^{(o)} ) G^{(o)} = G^{(o)} + G^{(o)} ( A_\star - A_\star^{(o)} ) G.
$$ 

Observe that $ G^{(o)} _{o g_i}= 0$ and $ G^{(o)} _{g_j g_i}= 0$ for $j \ne i$ (since there is a direct decomposition of 
$A_\star^{(o)}$ on $\ell^2 (X_\star) = \dC \delta_o \oplus_i \ell^{2} (  X_\star g_i)$). 
Thus,  composing the resolvent identity by $(  1 \otimes \langle \delta_{o} , \, \cdot \,    \delta_{g_i} \rangle )$ 
and $(  1 \otimes \langle \delta_{g_i} , \, \cdot \,    \delta_{o} \rangle )$ we obtain
$
G_{o g_i} =  G_{o o} a_{i^*} \gamma_i$ and $G_{g_i o}  = \gamma_i a_{i} G_{oo}.
$
Then applying the last inequality to $g_{i^*} = g_i^{-1}$ and using that $G_{ xg ,  yg} = G_{x y}$ for any $x,y ,g$ in $X_\star$, 
it gives the formula 
\begin{equation}\label{eq:resrec0}
 G_{og_i}   = G_{oo} a_{i^*} \gamma_i  = \gamma_{i^*}  a_{i^*} G_{oo}.
 \end{equation}

We may now prove that $G_{oo}$ is invertible and prove the first formula of the lemma. We use that $G^{(o)}_{oo} = (\mu - a_0)^{-1}$ and compose the resolvent identity by 
$(  1 \otimes \langle \delta_{o} , \, \cdot \,    \delta_{o} \rangle )$, we obtain
$$
G_{oo} = G^{(o)}_{oo} + \sum_{j} G_{o  g_j}  a_{j} G^{(o)}_{o o} = (\mu - a_0)^{-1} +  \sum_{j } G_{oo} a_{j^*} \gamma_j  a_{j} (\mu - a_0)^{-1}.
$$
Multiplying on the right by $\mu - a_0$, we obtain
$$
G_{oo}  \PAR{  \mu - a_0  - \sum_{j}  a_{j^*} \gamma_j  a_{j} } = 1.
$$
Hence, $G_{oo}$ is invertible and $\hat a_i = a_{i} \gamma_{i^*}$. The first formula follows.

For the second formula, we first repeat the above argument with 
$$A_\star^ {(og_i)} = A^{(o)}_\star - \sum_{j \ne i ^*}  a_{j^*} \otimes \delta_{g_i} \otimes \delta_{g_j g_i}  + a_{j} \otimes \delta_{g_j g_i} \otimes \delta_{g_i}.$$
 We use the resolvent identity between $A_\star^{(o)}$ and $A_\star^{(og_i)}$. Using that   $G^{(og_i)}_{g_j g_i ,g_j g_i} = \gamma_j$, we find that 
 $
 \gamma_i = \PAR{ \mu - a_0  - \sum_{j \ne i^*} a_{j^*} \gamma_j a_{j} }^{-1}.
 $
 It implies that
$$
\gamma_i (1- a_{i} \gamma_{i^*} a_{i^*} \gamma_i )^{-1} =  ( \gamma_i ^{-1} -a_{i} \gamma_{i^*} a_{i^*} )^{-1} =  \PAR{ \mu - a_0 - \sum_j a_{j^*} \gamma_j a_{j}}^{-1} = G_{oo}.
$$
From \eqref{eq:resrec0}, it concludes the proof. \end{proof}

\begin{proof}[Proof of Proposition \ref{prop:nonback2}]
Consider the operator $B = B_\mu$ and let $ A_1$ be as in Proposition \ref{prop:nonback}: for all $i \in [d]$, we have $
a_i(1) = \hat a_i ( 1 - \hat a_{i ^*}   \hat a_i )^{-1}.
$ It is sufficient to prove that $A_1 =( A - \mu ) ( G_{oo} \otimes 1)$. 
Using Lemma \ref{le:resrec}, we deduce that 
$
a_i(1) = a_{i} G_{oo}. 
$
In particular, 
$$
a_0(1) =  - 1 -  \sum_{i} \hat a_{i}(  1 - \hat a_{i^*} \hat a_i )^{-1} \hat a_{i^*}  =  - 1 -  \sum_{i}  a_{i} G_{oo} G_{oo}^{-1} G_{o g_{i}}.
$$
Then using \eqref{eq:resrec0} in Lemma \ref{le:resrec} and $\hat a_i = a_i \gamma_{i^*}$, we find
$$
a_0(1) = - 1 -  \sum_{i} a_{i}  G_{ o g_{i}} =  - \PAR{G^{-1}_{oo} + \sum_{i}  \hat a_i  a_{i^*} }G_{oo} = (a_0 - \mu) G_{oo},
$$
as requested.
\end{proof}

\subsection{Spectral radius of non-backtracking operators}

The next theorem is very important in our argument. It gives a sharp criterion to guarantee in terms of non-backtracking 
operators that the spectrum of an operator $A$ is in a neighbourhood of the spectrum of the operator $A_\star$.

Let $D$ be a bounded set in $\dC$, we define $\INT( D ) = \dC \backslash U$ where $U$ is the unique infinite component of 
$\dC \backslash D$ (in loose words, $\INT(D)$ fills the holes of $D$). For example, if $D \subset \dR$ 
or $D$ is simply connected, then $\INT(D) = D$. 

\begin{theorem}\label{prop:edgeAB}
Let $A$ be as in \eqref{eq:defA} and $A_\star$ the corresponding free operator defined by \eqref{eq:defAfree}. The following two results hold true:
\begin{enumerate}[(i)]
\item
For any $\mu \notin \INT ( \hat \sigma (A_\star) ) $, we have $\rho( (B_\star)_\mu ) < 1$, 
where $B_\star$ is defined as $B$ in Equation \eqref{eq:defB} with $S_i$ replaced by $\lambda (g_i )$.
\item
For any $\veps >0$, there exists $\delta > 0$ such that if for all complex $\mu$, 
$$
\rho(B_\mu) < \rho( (B_\star)_\mu ) + \delta, 
$$
then $\INT(\sigma(A))$ is in an $\veps$-neighbourhood of $\INT( \hat \sigma (A_\star) )$. 
\end{enumerate}
Moreover, 
the same holds with $A_{|H}$ and $(B_\mu)_{|K}$.
\end{theorem}

\begin{lemma}\label{le:cercle}
Let $A_\star$ be as in \eqref{eq:defAfree} and $B_\star$ its corresponding non-backtracking operator. We have 
$$
\{ z \in \dC: |z| = \rho ( B_\star) \} \subset \sigma( B_\star).
$$
\end{lemma}

\begin{proof}
If $\rho(B_\star)   = 0$, there is nothing to prove. We may thus assume $$ \rho ( B_\star)   = 1.$$ 

 We start by a consequence of Gelfand's formula on the spectral radius.  
 Let $k \geq 1$ be an integer and let $M$ be a bounded operator on $\dC^k \otimes \ell^2 ( X_ \star )$ in the 
 $C^\star$-algebra generated by 
 operators of the form $b \otimes \lambda(g)$, $b \in M_k(\dC)$, $g \in X_\star$. We introduce the standard tracial state $\tau$ defined by 
\begin{equation}\label{eq:deftau}
\tau(M) = \PAR{ \frac 1 k \Tr (\cdot) \otimes \langle \delta_o  ,  \cdot \, \delta_o \rangle } (M). 
\end{equation}
  Gelfand's formula asserts that $\rho(M) = \lim_n \|  M^n \|^{1/n}$
 (see for example \cite[Theorem 1.3.6]{MR1074574}). 
 Moreover, since $\| M \|^2 \geq \tau( |M|^2)$, we find
$$
\lim\sup_n \tau(|M^n|^2 )^{\frac 1 {2n} } \leq \rho(M ). 
$$
On the other hand, Haagerup's inequality (for matrix valued operators, see Buchholz \cite{MR1476122}) asserts that 
$$\| M \|^2 \leq L(m) \sum_{ x \in X_\star} \| M_{ox}\|^2  \leq k L(m)\tau (|M|^ 2) , $$ 
where $m = \sup \{ |x|: M_{ox} \ne 0 \}$ and $L(m)$ grows polynomially with $m$ (and depends implicitly on $d$). Hence, 
$$
\rho (M)^{2n}  \leq \| M^{n} \| ^2 \leq k L(nm) \tau ( |M^n |^2 ). 
$$
Since $L(nm)^{1/n}$ converges to $1$ as $n$ grows to infinity, we find
$$
 \rho (M) \leq \liminf_n \tau(  |M^n|^2 )^{\frac{1}{2n}} . 
$$
So finally 
$$
\rho( M) = \lim_{n} \tau(  |M^n|^2 )^{\frac{1}{2n}}. 
$$

As a consequence, writing $B_\star$ as a convolution operator on $\dC^r \otimes \ell^2 (X_\star) \otimes \dC^d$, we get that 
\begin{equation}\label{eq:gelfand}
1 = \rho( B_\star) = \lim_{n}  \tau(  |B_\star^n|^2 )^{\frac{1}{2n}}. 
\end{equation}
In particular, there exist $c >0$ and $\ell \geq 0$, such that for any integer $n \geq 0$, 
\begin{equation}\label{eq:tauBstarn}
\tau(  |B_\star^n|^2 ) \geq  (rd L(n))^{-1} \geq c (n+1)^{-\ell}.
\end{equation}

Now, we observe that for any $f \in \dC^r$,
\begin{equation}\label{eq:Bstarn}
B_\star^{n} ( f \otimes \delta_{(o,i)}  ) = \sum_{x} \PAR{\prod_{s=2}^{n+1} a_{x_s} f } \otimes \delta_{x},
\end{equation}
where the sum is over all reduced words in $X_\star$ $x = (x_1, \ldots, x_{n+1})$ of length $n+1$ with $x_1 = g_i$ 
and where we have set for $j \in [d]$,  $a_{g_j} = a_{j}$.  The lemma is ultimately a consequence of the fact that the vectors 
$B_\star^{n} ( f \otimes \delta_{(o,i)})$, $n \geq 0$, are orthogonal. More precisely, we find for $n \ne m$,
\begin{equation}\label{eq:puissorth}
\tau ( B_\star^{n} (B_\star^{m})^*  ) = 0.
\end{equation}

Now, let $z \in \dC$ with $|z|  = 1 + \veps > \rho ( B_\star) = 1 $ and  $R = ( z - B_\star )^{-1}$ be the resolvent of 
$B_\star$.  Since $R(z)$ is analytic on the complement of the spectrum, this is also true for its $k$-th derivative for any integer $k\geq 1$. It is thus sufficient to check that for some well-chosen $k$ and $c_1 >0$,
$\tau ( | \partial^k R(z)|^2  ) \geq c_1 \veps^{-1}$ for all $z$  with $|z| = 1 + \veps$.  
Since $|z| > \rho( B_\star)$, we have the converging Taylor expansion
$$
\partial^k R(z)  = \sum_{n=0}^\infty (n+1)(n+2) \cdots (n+k) z^{-n-k-1} B_\star^{n}.
$$
From \eqref{eq:tauBstarn}-\eqref{eq:puissorth}, we find 
$$
\tau (\partial^k |R (z) |^2 )  =    \sum_{n=0}^\infty (n+1)^2(n+2)^2 \cdots (n+k)^2 |z|^{-2n-2k-2} \tau ( |B_\star^{n}|^2 ) \geq c  \sum_{n=0}^\infty  (n+1)^{2k - \ell} (1+\veps)^{-2n-2k-2 }.
$$
If $2k \geq \ell$, this last expression is at least $c_1 \veps^{-1}$. This concludes the proof. 
\end{proof}

\begin{lemma}\label{le:contrho}
The map $(a_1, \cdots, a_d) \mapsto \rho(B_\star)$ is continuous for any norm on $M_r(\dC)^d$.
\end{lemma}

Before proving Lemma \ref{le:contrho} which requires some preliminaries, let us prove Theorem \ref{prop:edgeAB}.

\begin{proof}[Proof of Theorem of \ref{prop:edgeAB}]

Let us first prove (ii) assuming (i). By Lemma \ref{le:contrho}, there exists 
$\delta >0$ such that $\rho ( (B_\star)_\mu ) < 1 - \delta$ for all $\mu$ at distance larger than $\veps$ from 
$\INT ( \hat \sigma (A_\star ))$. Hence for all $\mu$ at distance larger than $\veps$ from $ \INT ( \hat \sigma (A_\star ))$, 
we have $\rho ( B_\mu) < 1$ and thus by Proposition \ref{prop:nonback2}, $\mu \notin \sigma(A)$.

As for (i),
assume on the contrary that there exists $\mu_0 \notin \INT(\hat \sigma(A_\star )) $ such that 
$\rho(B_{\mu_0}) \geq 1$. Since $\mu_0 \notin \INT(\hat \sigma(A_\star )) $, there exists a continuous function 
$t \mapsto \mu_t$, from $[0,\infty)$ to $\dC \backslash \INT(\hat \sigma(A_\star ))$  such that $|\mu_t|$ goes to infinity as $t$ 
goes to infinity. Note that also $ \rho( (B_\star)_\mu )  \leq \| (B_\star)_{\mu} \|$ goes to $0$ with $|\mu|$ going to infinity. 
Therefore Lemma \ref{le:contrho} and the intermediate value Theorem imply that there exists $t \geq 0$ such that 
$\mu_t  \notin \INT(\hat \sigma(A_\star ))$ and $ \rho( (B_\star)_{\mu_t} ) = 1 $. 
Then, by Lemma \ref{le:cercle}, $1\in \sigma( (B_\star)_{\mu_t})$ and, by Proposition \ref{prop:nonback2}, 
we deduce that $\mu_t \in \sigma( A_\star)$. This is a contradiction since 
$\sigma( A_\star) \subseteq \INT(\hat \sigma(A_\star ))$. 
\end{proof}

Let us now prove  Lemma \ref{le:contrho}. We start with a preliminary statement which is a non-commutative 
finite-dimensional Perron-Frobenius Theorem due to Krein and Rutman  \cite{MR0038008}. Let $m \geq 1$ be an integer 
and let $L$ be a linear operator on $M_m (\dC)$. We endow $M_m ( \dC)$ with the standard inner product 
$$
\langle x ,  y \rangle = \tr ( x^* y ). 
$$
The Frobenius norm is the associated hilbertian norm: $\| x \|_ 2 = \sqrt{ \tr ( x^* x )}$.  We say that $L$ is of non-negative 
type if for any $x$ positive semi-definite, $Lx$ is also positive semi-definite. We start with an elementary property of 
non-negative operators.

\begin{lemma}\label{le:stabnn}
Assume that $L$ is of non-negative type.  Then $L$ maps hermitian matrices to hermitian matrices. Moreover, 
for any  integer $n \geq 0$, $L^n$ and $L^*$ are of non-negative types.  
\end{lemma}

\begin{proof}
Since $L$ maps positive semi-definite matrices to positive semi-definite matrices, it also maps  negative semi-definite 
matrices to negative semi-definite matrices. Consequently, writing an hermitian matrix as $x = a - b$ with $a,b$ positive 
semi-definite, we find $L$ maps hermitian  matrices to hermitian matrices. 

By induction on $n$, it is immediate from the definition that $L^n$ is also of non-negative type. For $L^*$, let us first check 
that it maps hermitian matrices to hermitian matrices. First, from what precedes $(Ly)^* = L(y^*)$ 
(we write $y = y_1 + i y_2$ with $y_1, y_2$ hermitian and use linearity and $Ly_i$ hermitian). 
Note that a matrix $x$ is hermitian if and only if for any matrix $y$, $\tr(x^* y )= \tr (x y)$. Since $\tr(xy) = \tr(yx)$, 
we get that $x$ is hermitian if and only if $\langle x, y \rangle = \langle y^* , x\rangle$. 
Let $x$ be hermitian. For any $y$ we have 
$\langle L^* x , y \rangle = \langle  x ,  L y \rangle =  \langle   (Ly)^* ,   x \rangle  =  \langle   L(y^*) ,   x \rangle  =  \langle   y^* ,  L ^* x \rangle$ 
where we have used at the second identity that for any $y$, $\langle x , y \rangle = \langle y^* , x^* \rangle$ and $x = x^*$. 
We thus have proved that $L^*$ maps hermitian matrices to hermitian matrices. 

Similarly, an hermitian matrix $x$ is positive semi-definite if and only if for any $y$ positive semi-definite 
$\langle x , y \rangle = \tr (x y)  = \tr (y^{1/2} x y^{1/2} )\geq 0$. However, if $x,y$ are positive semi-definite 
$\langle L ^* x , y \rangle = \langle x , L y \rangle  \geq 0$, since $L$ is non-negative. It concludes the proof. \end{proof}

The following theorem is a direct consequence of the Krein-Rutman Theorem \cite{MR0038008}, see e.g. 
Deimling \cite[Theorem 19.1]{MR787404}.  For completeness, we are have included a proof in Section \ref{sec:aux}. 

\begin{theorem}\label{th:PF}
Assume that $L$ is of non-negative type and let $\rho$ be its spectral radius. 
\begin{enumerate}[(i)]
\item \label{PF1} $\rho$ is an eigenvalue of $L$ and it has a positive semi-definite eigenvector. 
\item \label{PF2} If $x$ is positive definite, we have 
$$
\lim_{n \to \infty} \| L^n x \|_{2}^{\frac 1 n} = \rho .
$$
 \end{enumerate}
\end{theorem}
 
We are ready for the proof of Lemma  \ref{le:contrho}.

\begin{proof}[Proof of Lemma \ref{le:contrho}]
From \eqref{eq:gelfand}, we have $\rho (B_\star) = \lim_{n \to \infty} \| \tau_i ( | B_\star^n |^2 )  \|^{1/{2n}}$ for some
 $i \in [d]$. With the notation of \eqref{eq:Bstarn}, we find 
$$   \tau_i( |B_\star^{n} |^2  )  =  \tau_i (   B_\star^n (B_\star^n )^*  ) = \sum_{x} \ABS{\prod_{s=2}^{n+1} a_{x_s}}^2,$$ 
where  the sum is over all reduced words $x = (x_1, \ldots, x_{n+1})$ in $X_\star$ of length $n+1$  with 
$x_1 = g_i$. Let $m = d r$ and let $Z_n $ be the block diagonal matrix in $M_{m} (\dC)$ with diagonal blocks in $M_r(\dC)$, 
$(  \tau_1(  |B_\star^{n} |^2  )  , \ldots ,\tau_d( |B_\star^{n} |^2  ) )$,  we find $Z_0 = 1$ and for integer $n \geq 0$, 
\begin{equation}\label{eq:recG}
Z_{n+1}  = L (Z_n),
\end{equation}
where $L$ is the operator on $M_m(\dC)$ defined as follows. For $x \in M_m (\dC)$, we write in  
$x = (x_{ij})$, $i,j \in [d]$ for its blocks in $M_r (\dC)$. Then $L(x)$ is block diagonal with diagonal blocks 
$(L(x)_{11} , \ldots, L(x)_{dd})$ and for all $i$ in $[d]$,
\begin{equation}\label{eq:defL}
L (x)_{ii}   = \sum_{j \ne i^*}  a_j  x_{jj} a_{j} ^* \in M_r (\dC).
\end{equation}
It is straightforward to check that  $L$ is of non-negative type. Indeed, if $x$ is positive semi-definite then, 
for each $j \in [d]$, $x_{jj}$ is also positive semi-definite in $M_r(\dC)$ and thus $a_j  x_{jj} a_{j} ^*$ is positive semi-definite.

Now, from \eqref{eq:recG}, $Z_n = L^n (1)$. We deduce from Theorem \ref{th:PF}\eqref{PF2} that $\rho (B_\star)$ is 
equal to the square root of the the spectral radius of $L$.  We recall finally that the spectral radius is a continuous function on $M_m(\dC)$ for 
any norm.
\end{proof}

\subsection{Random weighted permutations}
\label{subsec:mainB}

We now consider symmetric random permutations, $X = [n]$. We consider the vector space 
$K_0$ of  vectors $f \in \dC^r \otimes \ell^2 (E)$ such that $\sum_x f(x,i) = 0$ for all $i \in [d]$ (that is the orthogonal of 
$\dC^r \otimes \IND \otimes \dC^d$). The following result is the central technical contribution hidden behind the proof of Theorem \ref{th:main}.

\begin{theorem}\label{th:mainB}
For any $0 < \veps < 1$, for symmetric random permutations, with probability tending to one as $n$ goes to infinity, for all 
$(a_i), i \in [d],$ such that $\max_i ( \| a_i \| \vee  \| a_i ^{-1} \|^{-1} )  \leq \veps^{-1}$, we have
\begin{equation*}
\rho  (B_{|K_0})  \leq \rho( B_\star)+ \veps, 
\end{equation*}
where $B$ is the non-backtracking operator associated to $A$ defined by \eqref{eq:defA} and $\rho( B_\star)$ is the 
spectral radius of the corresponding non-backtracking operator in the free group.
\end{theorem}

Note that in the above Theorem, the assumption $\max_i ( \| a_i \| \vee  \| a_i ^{-1} \|^{-1} )  \leq \veps^{-1}$ entails a control
on the norm of $a_i ^{-1}$ and in particular, the assumption that it is invertible. This is a technical assumption which does appear in the 
main result Theorem \ref{th:main}. This  will however not be a major obstacle for proving Theorem \ref{th:main} in the next subsection 
by using  the fact that invertible matrices in are dense in the space of all matrices $M_r(\dC)$.

As a corollary, in the next subsection we  obtain a proof of Theorem \ref{th:main}. 
The forthcoming Section \ref{sec:mainB} is devoted to the proof of Theorem \ref{th:mainB}.  It relies on a 
refinement of the trace method F\"uredi and Koml\'os \cite{MR637828} which was developed in 
\cite{M13,BLM,bordenaveCAT}.
In special case where $a_i$ is $E_{x_i y_i}$, this theorem is contained \cite{bordenaveCAT} and under an extra 
assumption in \cite{FrKo}.

\subsection{Proof of Theorem \ref{th:main}}

\begin{proof}[Proof of Theorem \ref{th:main}] 
We start by proving the last claim: $s(A_{|H_0})$ converges in probability toward $s(A_\star) $. From Corollary \ref{cor:inclusion}, it suffices to prove the upper bound: for any $\veps >0$, with high probability $s(A_{|H_0}) \leq s(A_\star) + \veps$. The proof is a combination of 
Theorem \ref{th:mainB} and Theorem \ref{prop:edgeAB} applied to $H = H_0$ and $K = K_0$. We may assume that 
$\max_i \| a_i\| \leq 1$. If $A$ and $A'$ are operators of the form \eqref{eq:defA} with associated matrices 
$(a_i)$ and $(a'_i)$ and the same $(S_i)$, we have $\| A - A' \| \leq \sum_i \| a'_i - a_i \|$. Hence, up to modifying 
$\veps$ in $\veps / 2$, in order to prove the upper bound \eqref{eq:upboundA}, it is enough to consider weights 
$(a_i)$ such that for any $i$, $\| a_i \| \leq 1$ and $\| a_i ^{-1} \| \leq 2d / \veps$ and check that on an event of high probability
 the upper bound \eqref{eq:upboundA} holds. We already know that $\| A_{|H_0} \| \leq d$. 
 Let $\mu > s(A_\star) + \veps$ be a real number. Recall from Lemma \ref{le:resrec} that $\hat a_i (\mu) = a_i \gamma_{i^*} (\mu)$. Since  $\mu > s(A_\star) + \veps$ and $\sigma(A_\star^{(o)})$ is contained in the convex hull of $\sigma(A_\star)$,  $ G^{(o)} (\mu) = (\mu - A^{(o)}_\star)^{-1}$ is positive semi-definite. More precisely, from the spectral theorem $ (\mu + d)^{-1} I \leq G^{(o)} (\mu) \leq \veps^{-1} I$.  Thus,  we have  $(\mu + d)^{-1} I_r \leq |\gamma_{i^*} (\mu)| \leq \veps^{-1} I_r$. 
 Hence, if $\mu  \leq d$, we get $\|\hat a_i(\mu) \| \leq \| a_i \| \| \gamma_{i^*} (\mu) \| \leq  1 / \veps$ and  
 $\| \hat a_i( \mu)^{-1}  \| \leq \| a^{-1} _i\| \|  \gamma_{i^*} (\mu)^{-1} \|  \leq (2d )^2 /  \veps $. 
 It remains to use the event of high probability in Theorem \ref{th:mainB} with 
 $\veps' = (\veps / (2d)^2 ) \wedge \delta$ where $\delta >0$ is as in Theorem \ref{prop:edgeAB}. This proves the last claim of Theorem \ref{th:main}. 
 
 It remains to prove the first claim of Theorem \ref{th:main}: the convergence in probability of $\sigma(A_{|H_0})$ toward $\sigma(A_\star)$ for the Hausdorff distance. From Corollary \ref{cor:inclusion}, it suffices to prove the upper bound \eqref{eq:upboundA}.  First,  what precedes applied to minus $A$ implies the convergence in probability of the operator norm of $A_{|H_0}$ when the symmetry condition \eqref{eq:sym} holds. 
 Therefore, from  Pisier in \cite[Proposition 6]{MR1401692}, see Section \ref{sec:cor} below, for any matrix-valued polynomial $P_n $ in the $S_i$'s, $\| (P_n)_{|H_0} \| $ converges in probability towards $\| P_\star \|$, the corresponding polynomial in the $\lambda(g_i)$'s. For $T$ self-adjoint, we have  $s(T) = \| T - c I \| -  c$ for all $c \geq \| T \|$. Hence, if $P_\star$ and $P_n$ are self-adjoint, this also implies the convergence of right-edge of the spectrum $s((P_n)_{|H_0})$ toward $s(P_\star)$ in probability.

We may then prove \eqref{eq:upboundA}. Fix $\veps >0$ and let $J  = \sigma(A_\star) + [-\veps, \veps]$. Since $\sigma(A_\star)$ is compact, $J = \cup_i [\alpha_i,\beta_i]$ is a finite disjoint union of intervals (possibly reduced to a singleton). Consider the polynomial $q (x) = \prod_i (x-\alpha_i)(x-\beta_i)$ and set $P_n = q (A)$, $P_\star = q (A_\star)$. By construction, $q$ is negative on the interior of $J$ and non-negative on $J^c$. Hence $s(P_\star) < 0$. From what precedes, with probability tending to one, $s ((P_n)_{|H_0}) < 0$. It implies that $A_{|H_0}$ has no eigenvalue on $J^c$ with probability tending to $1$. This proves \eqref{eq:upboundA}.
\end{proof}

\section{Proof of Theorem \ref{th:mainB}}

\label{sec:mainB}

\subsection{Overview of the proof}
\label{subsec:o}

Let us first describe the method introduced by F\"uredi and Koml\'os \cite{MR637828} to bound the norm of a random matrix.  
Let  $M$ be a random matrix in $M_n (\dC)$. Imagine that we want to prove that, for some $\rho > 0$, for any $\veps > 0$, 
with probability tending to $1$ as $n$ goes large, 
\begin{equation}\label{eq:normM}
\| M \| \leq  \rho (1 + \veps).
\end{equation}
For integer $k\geq 1$, we write 
$$
\| M \|^{2k} = \| M M^* \|^{k} = \| (M M^*)^k \| \leq \tr \PAR{ ( M M^*)^k }.  
$$
At the last step, we might typically loose a factor proportional to $n$, since the trace is the sum of $n$ eigenvalues. 
Hence, it is reasonable to target a bound of the form
\begin{equation}\label{eq:EtrM}
\dE \tr \PAR{ ( M M^*)^k } \leq n \rho^{2k} \PAR{ 1 + \veps}^{2k}. 
\end{equation}
If we manage to establish such an upper bound, we would deduce from Markov inequality that for any $\delta >0$, the event 
$$
\| M \| \leq \rho (1 +  \veps) ( 1 + \delta) 
$$
has probability at least $$1 - n \PAR{  1 +  \delta }^{-2k} = 1 - \exp\PAR{ - 2 k \log (1 + \delta) + \log n}.$$
  
Hence, the bound on the trace \eqref{eq:EtrM} implies the bound \eqref{eq:normM} with 
$\veps'=  \veps + o(1)$ if $k \gg \log n$. Then, the problem of bounding the norm of $M$ has been reduced to bounding 
the expression
$$
\dE \tr \PAR{ ( M M^*)^k } = \sum_{x_1,\ldots , x_{2k}} \dE \prod_{t=1}^{k} M_{x_{2t-1},x_{2t}} \bar M_{x_{2t+1},x_{2t}},
$$
where the sum is over all sequences $(x_1, \ldots, x_{2k})$ in $[n]$ with the convention $x_{2k+1} = x_1$. 
The right-hand side of the above expression may usually be studied by combinatorial arguments.

As it is described, the method of F\"uredi and Koml\'os cannot be applied directly in this context. 
Indeed, assume for concreteness that $r = 1$ and that $A_1$ and $A$ are stochastic matrices and that the symmetry 
condition \eqref{eq:sym} holds (recall that $A_1 = \sum_i a_i$ is defined in \eqref{eq:defA1})
We are then interested in the matrix $M = A - \frac 1 n \IND \otimes \IND$ and 
$\| A_{|H_0}\| = \| M\|$ and we aim at a bound of the type \eqref{eq:normM} for some $\rho < 1$. 
However, with probability at least  $ c_0 n^{-c_1}$, $\{1,2\}$ is a connected component of $G^{\sigma}$ 
where $G^\sigma$ is the colored graph introduced in Definition \ref{def0} (this can be checked from the forthcoming 
computation leading to \eqref{eq:subgraph}). On this event, say $E$, the eigenvalue $1$ has multiplicity 
at least $2$ in $A$, hence $\| M\|  =  1$. We deduce that 
$$
\dE \tr \PAR{ ( M M^*)^k } \geq \dE \| M \|^{2k} \geq \dP ( E ) \geq c_0n^{-c_1}.
$$
However, the latter is much larger than $\rho^{2k}$ if $k \gg \log n$ and \eqref{eq:EtrM} does not hold. More generally  the 
presence of subgraphs with many edges in $G^\sigma$ prevents the bound \eqref{eq:EtrM} to hold for $ k \gg \log n$. 
Such subgraphs were called {\em tangles} by Friedman \cite{MR2437174}.  We will circumvent this intrinsic difficulty as follows.
 Let $B$ be as in Theorem \ref{th:mainB}. 
\begin{enumerate}[(i)]
\item {\em Bound the spectral radius by the norm of a large power: } we fix a positive integer $\ell$. We recall that $K_0$ is 
the vector space of codimension $rd$ orthogonal to $K_1 = \dC^r \otimes \IND \otimes \dC^d$. We have
\begin{equation}\label{eq:basicrho}
\rho (B_{|K_0}) = \rho (( B^\ell)_{|K_0} )^{1 / \ell} \leq \sup_{g \in K_0, \| g \|_2 = 1 }    \| B^{\ell} g \|_2^{1 / \ell}.
\end{equation}

\item {\em Remove the tangles (Subsection \ref{subsec:PD}): } we  obtain a matrix $B^{(\ell)}$ which coincides with $B^\ell$ 
on an event of high probability.

\item {\em Project on $K_0$ (Subsection \ref{subsec:PD}): } we  project $B^{(\ell)}$ on $K_0$ to evaluate the right-hand side 
of \eqref{eq:basicrho}. We  obtain a matrix $\underline B^{(\ell)}$. However, since $K_0$ is not always an invariant 
subspace of $B^{(\ell)}$,  there will also be some remainder matrices. 

\item {\em Method of F\"uredi and Koml\'os (Subsection \ref{subsec:FK}): } we may then evaluate the norm of 
$\underline B^{(\ell)}$ by taking a trace of power $2m$ and estimate its expectation. The $2m \ell$ plays the role of $2k$ in the above presentation 
of the method of F\"uredi and Koml\'os. We thus need $2 m \ell \gg \log n$ to get a sharp estimate. We  obtain 
$2 m \ell$ of order $(\log n)^2 /( \log \log n)$. We  use at a crucial step that we removed the tangles 
(Lemma \ref{le:enumpath} and Lemma \ref{le:enumpathR}). We  also connect the expected trace of powers of 
non-backtracking matrices to powers of the corresponding non-backtracking operator on the free group (Lemma \ref{le:isopath}).

\item {\em Net argument (Subsection \ref{subsec:net}): } to prove Theorem \ref{th:mainB}, we need to bound all 
spectral radii of $B_{|K_0}$ for all weights $a_i$ with uniformly bounded norm. We will use a net argument on the norm 
of $(B_{|K_0})^{\ell}$ conditioned on the event that there is no tangles in $G^\sigma$.

\end{enumerate}

\begin{remark}\label{rq:CAT} Let us comment on the main differences with \cite{M13,BLM,bordenaveCAT}, notably \cite{bordenaveCAT} 
which is the closest. The steps (i)-(iii) are similar to \cite{bordenaveCAT}. In the analog of step (iv),  the work in \cite{bordenaveCAT} is 
greatly simplified by the  fact that, with the terminology of the present paper, the weights $a_i$ are matrices of the standard 
basis $E_{uv}$, $u,v \in [r]$. In this special case, the spectral radius of $\rho(B_\star)$ has an explicit combinatorial expression 
and products of $a_i$ have a simple combinatorial description. Finally, step (v) is not present in \cite{bordenaveCAT}.\end{remark} 

\subsection{Path decomposition}
\label{subsec:PD}

In this subsection, we set $X = [n]$ and let $A$ be as in \eqref{eq:defA}. We denote by $B$ the non-backtracking 
matrix of $A$. Here we  give  an upper bound on $\rho (B_{|K_0})$ in terms of operator norms of new matrices which will be tuned for 
the use of the method of F\"uredi and Koml\'os.  We fix a positive integer $\ell$.  The right hand side of 
Equation \eqref{eq:basicrho} can be studied by a weighted expansion of paths.  To this end, we will use some definitions 
for the  sequences in $E$ which we will encounter to express the entries of $B^{\ell}$ as a weighted sum of 
non-backtracking paths. Recall the definition of a colored edge $[x,i,y]$ and a colored graph in Definition \ref{def0} 
and see Figure \ref{fig:defg} for an illustration of the new definitions. Recall that  $E = X \times [d]$.

\begin{definition}\label{def1}
For a positive integer $k$, let $ \gamma = (\gamma_1, \ldots, \gamma_{k} ) \in E^k$, with $\gamma_t = ( x_t, i_t)$. 
\begin{enumerate}[-]
\item
The sets of vertices and pairs of colored edges of $\gamma$ are the sets   
$V_\gamma = \{ x_{t}: \hbox{$1 \leq t\leq k$}\}$ and $E_\gamma = \{ [x_{t}, i_t , x_{t+1} ]: 1 \leq t \leq k-1\}$. 
We denote by $G_\gamma$ the colored graph with vertex set $V_\gamma$ and  colored edges $E_\gamma$.  
\item
An {\em (extended) path} of length $k-1$ is an element of $E^k$.  The path $\gamma$ is {\em non-backtracking} if for 
any $t \geq 1$, $  i _{t+1} \ne  i^*_{t}  $.  The subset of non-backtracking paths of $E^k$ is $\Gamma^{k}$.  If $e, f \in E$, 
we denote by $\Gamma^{k}_{ef}$   paths in $\Gamma^k$ such that $\gamma_1 = e$, $\gamma_{k} = f$. 
\item
The {\em weights} of the path $\gamma$ is the element of $M_r(\dC)$,
$$
a (\gamma) = \prod_{t=2}^{k} a_{i_{t}}.
$$
\end{enumerate}
\end{definition}

\begin{figure}[htb]
\begin{center}  
\resizebox{7cm}{!}{
\begin{tikzpicture}[main node/.style={circle, draw , fill = lightgray, text = black, thick}]
\node[main node]  at (0,0) (2) {2} ;
\node[main node] at (2,0) (5) {5} ;
\node[main node] at (4,0) (1) {1} ;
\node[main node] at (4,2 ) (3) {3} ;
\node[main node] at (2,2) (4) {4} ;
\node[main node] at (6,0) (6) {6} ;

\draw[cyan,ultra thick] [out = 60 , in = 0] (2) to (0,1) ; \draw[cyan,ultra thick] [out = 180 , in = 120] (0,1) to (2) ; 
\draw[ magenta,ultra thick] (2) to (5) ; 
\draw[ orange!80,ultra thick] (5) to (1) ;  
  \draw[cyan,  ultra thick] [out = 70 , in = -70]  (1) to (3) ;  
    \draw[olive,  ultra thick]  [out = 110 , in = -110] (1) to (3) ;  
 \draw[magenta,  ultra thick] (3) to (4) ;
 \draw[cyan,  ultra thick] (4) to (5) ;
  \draw[cyan,  ultra thick] (1) to (6) ;

  \node[text = black!80] at (0,1) {1} ;   
  \node[text = black!80] at (1,0) {2} ; 
    \node[text = black!80] at (3,0) {3,7,11} ; 
    \node[text = black!80] at (5,0) {14} ;   
       \node[text = black!80] at (4.45,1) {4,12} ; 
       \node[text = black!80] at (3.55,1) {8,13} ; 
    \node[text = black!80] at (3,2) {5,9} ; 
       \node[text = black!80] at (2,1) {6,10} ;

\end{tikzpicture}
}
{\small $$ \gamma =  (2,1)  (2,2) (5,3) ( 1,1)  (3,2)   (4,1)   (5,3) (1,4) (3,2)  (4,1)  (5,3)  (1,1)  (3,4)    (1,1)  (6,2) $$} 
\vspace{-25pt}
\caption{An example where the involution $i^*$ is the identity. The colored graph $G_\gamma$ associated to a path 
$\gamma \in \Gamma^{15}$. The numbers on the edges are the values of $t$ such that $[x_t,i_t,x_{t+1}]$ is equal to this edge. 
We have $V_\gamma = [6]$ and $E_\gamma = \{ [2,1,2], [2,2,5], [5,3,1] , [1,1,3] , [3,2,4], [4,1,5], [1,4,3], [1,1,6]\}$.} \label{fig:defg}

\end{center}\end{figure}

By construction, from \eqref{eq:defB} we find that
$$
(B^{\ell}) _{e f}=  \sum_{\gamma \in \Gamma^{\ell+1} _{e f}} a(\gamma)  \prod_{t=1}^{\ell} ( S_{i_t} )_{x_{t} x_{t+1} } .
$$
Observe that in the above expression, $\Gamma^{\ell+1}$ and $a(\gamma)$ do not depend on the permutation matrices $S_i, i \in[d]$.  We set 
\begin{equation}\label{eq:defS}
\underline S_{i }   = S_{i}- \frac 1 n  \IND \otimes \IND .
\end{equation}
Note that $\underline S_i$ is the orthogonal projection of $S_i$ on $\IND^\perp$. Hence, setting as in \eqref{eq:defB}, $\underline B = \sum_{j \ne i ^* } a_j \otimes \underline S_i \otimes E_{ij}$, we get  that, if $g \in K_0$, 
\begin{equation}\label{eq:BBBk}
B^\ell g = \underline B^\ell g. 
\end{equation}
Moreover, arguing as above we find
\begin{equation}\label{eq:defdefub}
(\underline B^{\ell}) _{e f} = \sum_{\gamma \in \Gamma^{\ell+1} _{e f}} a(\gamma)  \prod_{t=1}^{\ell} ( \underline S_{i_t} )_{x_{t} x_{t+1} } 
\end{equation}
The matrix $\underline B$ will however not be used. Indeed, as pointed in \S \ref{subsec:o},  due to polynomially 
small events which would have had a big influence on the expected value of $B^\ell$ or $\underline B^\ell$ for large 
$\ell$, we will first reduce  the above sum over $\Gamma^{\ell+1} _{e f}$ to a sum over a smaller subset. 
We will only afterward project on $K_0$, it will create some extra remainder terms. We now introduce a key definition 
(recall the definition of cycles and $(H,x)_\ell$ in Definition \ref{def0}). 

\begin{definition}[Tangles]\label{def2}
A graph $H$ is {\em tangle-free} if it contains at most one cycle, $H$ is {\em $\ell$-tangle-free} if for any vertex $x$, $(H,x)_{\ell}$ 
contains at most one cycle. Otherwise, $H$ is tangled or $\ell$-tangled. We say that $\gamma \in E^k$ is tangle-free or tangled 
if $G_\gamma$ is. Finally,  $F^k$ and $F^{k} _{e f}$ will denote the subsets of tangle-free paths in $\Gamma^k$ and $\Gamma^{k}_{ef}$.
\end{definition}

Now, recall the definition of the colored graph $G^\sigma$ in Definition \ref{def0}. Obviously, if $G^\sigma$ is 
$\ell$-tangle-free  and $0 \leq k \leq 2\ell$ then 
\begin{equation}\label{eq:BkBk}
B^{k}   = B^{(k)},
\end{equation}
where 
$$
(B^{(k)}) _{e f}=  \sum_{\gamma \in F^{k+1} _{e f}} a(\gamma)  \prod_{t=1}^{k} ( S_{i_t} )_{x_{t} x_{t+1} } .
$$
We define similarly the matrix $\underline B^{(k)}$ by 
\begin{equation}\label{eq:defD}
(\underline B^{(k)}) _{e f}=  \sum_{\gamma \in F^{k+1} _{e f}} a(\gamma)  \prod_{t=1}^{k} ( \underline S_{i_t} )_{x_{t} x_{t+1} } .
\end{equation}
Beware that even if $G^\sigma$ is $\ell$-tangle-free  and $2 \leq k \leq \ell$,  $\underline B^k$ is a priori different from 
$ \underline B^{(k)}$ (in \eqref{eq:defdefub} and \eqref{eq:defD} the summand is the same but the sum 
in \eqref{eq:defdefub} is over a larger set). Nevertheless, at the cost of extra terms, as in \eqref{eq:BBBk}, 
we may still express $B^{(\ell)}g$ in terms of $\underline B^{(\ell)}g$ for all $g \in K_0$. We start with the following 
telescopic sum decomposition:
\begin{eqnarray}\label{eq:iopl}
(B^{(\ell)}) _{e f} &  =  & (\uB^{(\ell)} )_{ef}  + \sum_{\gamma \in F^{\ell+1} _{e f}} a(\gamma)    \sum_{k = 1}^\ell \PAR{ \prod_{t=1}^{k-1} ( \underline S_{i_t} )_{x_t x_{t+1}}} \PAR{ \frac 1{ n}  } \PAR{\prod_{t = k+1}^\ell (S_{i_t} )_{x_t x_{t+1}}  },
\end{eqnarray}
which follows from the identity, 
$$
\prod_{t=1}^\ell x_t = \prod_{t=1}^\ell y_t + \sum_{k=1}^{\ell}\PAR{ \prod_{t=1}^{k-1} y_t } ( x_k - y_k) \PAR{\prod_{t = k+1}^{\ell} x_t}.
$$
\begin{figure}[htb]
\begin{center}  
\resizebox{12cm}{!}{
\begin{tikzpicture}[main node/.style={circle,fill , text = black, thick}]

\draw[orange , -,thick] (0.0,0)  [out = 0, in = 0] to (0.45,0.8)  ; 
\draw[orange , -,thick] (0.45,0.8) [out = 180 , in = 180]   to (0.9,0)   ; 
\draw[ -,thick] (0.9,0) to (1.1,0)  ; 
\draw[ cyan,  -,thick] (1.1,0) [out = 0, in = 0]  to (1.55,0.8)  ; 
\draw[ cyan , -,thick] (1.55,0.8) [out = 180 , in = 180]  to (2,0) ;  

\node[above]  at (0,0) (1) {\tiny{$\gamma_1$}} ;
\node[above]  at (2,0) (2)   {\tiny{$\gamma_{\ell+1}$}} ;
\node[above]  at (0.85,-0.05) (3)  {\tiny{$ \gamma_{k}$}} ;
\node[below]  at (1.25,0.05) (4) {\tiny{$\gamma_{k+1}$}} ;

\draw [fill] (0,0) circle [radius=0.03] ;
\draw [fill] (0.9,0) circle [radius=0.03] ;
\draw [fill] (1.1,0) circle [radius=0.03] ;
\draw [fill] (2,0) circle [radius=0.03] ;

\draw[orange , -,thick] (3.0,0)  [out = 0, in = 0] to (3.45,0.8)  ; 
\draw[orange , -,thick] (3.45,0.8) [out = 180 , in = 180]   to (3.9,0)   ; 
\draw[ -,thick] (3.9,0) to (4.1,0)  ; 
\draw[ cyan,  -,thick] (4.1,0) [out = 0, in = 0]  to (4,0.5)  ; 
\draw[ cyan , -,thick] (4,0.5) [out = 180 , in = 90]  to (3.7,0) ;  
\draw[ cyan , -,thick] (3.7,0) [out = -90 , in = -90]  to (5,0) ;  

\node[above]  at (3,0) {\tiny{$\gamma_1$}} ;
\node[above]  at (5,0)  {\tiny{$\gamma_{\ell+1}$}} ;
\node[above]  at (3.85,-0.05)  {\tiny{$ \gamma_{k}$}} ;
\node[below]  at (4.25,0.05)  {\tiny{$\gamma_{k+1}$}} ;

\draw [fill] (3,0) circle [radius=0.03] ;
\draw [fill] (3.9,0) circle [radius=0.03] ;
\draw [fill] (4.1,0) circle [radius=0.03] ;
\draw [fill] (5,0) circle [radius=0.03] ;

\draw[orange , -,thick] (6.0,0)  [out = 0, in = 0] to (6.45,0.8)  ; 
\draw[orange , -,thick] (6.45,0.8) [out = 180 , in = 180]   to (6.9,0)   ; 
\draw[ -,thick] (6.9,0)[out = 0, in = -45]  to (7.15,0.27)  ;
 \draw[ -,thick] (7.15,0.27)[out = 135, in = 90]  to (6.9,0)  ;
\draw[ cyan,  -,thick] (6.9,0) [out = -60, in = -240]  to (8,0) ;  

\node[above]  at (6,0)  {\tiny{$\gamma_1$}} ;
\node[above]  at (8,0)  {\tiny{$\gamma_{\ell+1}$}} ;
\node[above]  at (6.75,-0.05)  {\tiny{$ \gamma_{k}$}} ;
\node[below]  at (7.1,0.0)  {\tiny{$\gamma_{k+1}$}} ;

\draw [fill] (6,0) circle [radius=0.03] ;
\draw [fill] (6.9,0) circle [radius=0.03] ;
\draw [fill] (8,0) circle [radius=0.03] ;

\end{tikzpicture}}
\caption{Tangle-free paths whose union is tangled.} \label{fig:Gamma3}
\end{center}\end{figure}

\vspace{-10pt}

We now rewrite \eqref{eq:iopl} as a sum of matrix products for lower powers of $\uB^{(k)}$ and $B^{(k)}$ up to some 
remainder terms. We decompose a path $\gamma  = (\gamma_1, \ldots, \gamma_{\ell +1} ) \in \Gamma^{\ell+1}$ 
as a path $\gamma'=  (\gamma_1, \ldots, \gamma_{k} ) \in  \Gamma ^{k}$, a path 
$\gamma''= (\gamma_{k}, \gamma_{k+1} ) \in \Gamma^2$ and a path 
$\gamma''' = (\gamma_{k+1}, \ldots, \gamma_{\ell +1}) \in \Gamma ^{\ell - k+1}$. 
If the path $\gamma$ is in $F^{\ell+1}$ (that is, $\gamma$ tangle-free), then  the three paths are tangle-free, 
but the converse is not necessarily true, see Figure \ref{fig:Gamma3}. This will be the origin of the remainder terms.  
For $1 \leq k \leq \ell$, we denote by $F^{\ell+1}_{k}$ the set of $\gamma \in \Gamma^{\ell+1}$ as above such that
$\gamma'  \in F^{k}$, $\gamma'' \in F^{2} = \Gamma^2$ and $\gamma''' \in F^{\ell - k +1}$. 
Then $F^{\ell+1} \subset F^{\ell+1}_{k}$. Setting, $F^{\ell+1}_{k,ef} = F^{\ell+1}_{k} \cap \Gamma^{\ell+1}_{ef}$, 
we write in \eqref{eq:iopl}
\begin{eqnarray*}
\sum_{\gamma \in F^{\ell+1} _{e f}}    \PAR{\star }  = \sum_{\gamma \in F^{\ell+1} _{k,e f}}  \PAR{ \star } - \sum_{\gamma \in F^{\ell+1} _{k,e f} \backslash F^{\ell+1} _{e f}}   \PAR{\star },
\end{eqnarray*}
where $(\star)$ is the summand on the right hand side of  \eqref{eq:iopl}.  We have 
$$
a(\gamma) = a(\gamma') a(\gamma'')  a(\gamma''').
$$
We denote by  $\overline B$ the matrix on $\dC^r \otimes \dC^E$ defined by 
$$\overline B= \sum_{j \ne i^*}  a_j \otimes (\IND \otimes \IND)  \otimes  E_{ij}.$$  
Observe that 
$\overline B_{ef} = \sum a(\gamma) $ where the sum is over all $\gamma $ in $\Gamma^2_{ef} = F^2_{ef}$.  We get 
$$
 \sum_{\gamma \in F^{\ell+1} _{k,e f}} a(\gamma)     \PAR{ \prod_{t=1}^{k-1} ( \underline S_{i_t} )_{x_t x_{t+1}}} \PAR{ \frac 1{ n}  } \PAR{\prod_{t = k+1}^\ell (S_{i_t} )_{x_t x_{t+1}}  } =  \PAR{ \frac{1}{n}} \uB^{(k-1)} \PAR{\overline B  B^{(\ell - k)}}_{ef}. 
$$
We set for all $e,f \in E$, 
\begin{equation}\label{eq:defR}
(R^{(\ell)}_{k} )_{ef}   =  \sum_{\gamma  \in F^{\ell+1} _{k,e f} \backslash F^{\ell+1} _{e f}} a(\gamma)  \PAR{ \prod_{t=1}^{k-1} ( \underline S_{i_t} )_{x_t x_{t+1}}}  \PAR{\prod_{t = k+1}^\ell (S_{i_t} )_{x_t x_{t+1}}  }.
\end{equation}
We have  from \eqref{eq:iopl} that \begin{eqnarray*}
B^{(\ell)}   &= & \uB^{(\ell)}   +    \frac 1 {n}  \sum_{k = 1} ^{\ell}  \uB^{(k-1)} \overline B  B^{(\ell - k)}  -   \frac 1 {n}     \sum_{k = 1}^{\ell} R^{(\ell)}_{k}.\end{eqnarray*}
Now, observe that if $G^\sigma$ is $\ell$-tangle free, then, from  \eqref{eq:BkBk}, 
$\overline B B^{(\ell - k)} = \overline B B^{\ell -k}$. Moreover, the kernel of $\overline B$ contains $K_0$. 
Since $B^{\ell-k} K_0 \subset K_0$, we find that $\overline B B^{\ell-k} = 0$ on $K_0$. 
So finally, if $G^\sigma$ is $\ell$-tangle free, for any $g \in K_0$,  
\begin{eqnarray*}
 B^{\ell} g \  &   =&  \uB^{(\ell)}  g   -  \frac 1{n}   \sum_{k = 1}^\ell R^{(\ell)}_{k}  g .\label{eq:decompBkx}
\end{eqnarray*}
Putting this last inequality in \eqref{eq:basicrho}, the outcome of this subsection is the following lemma. 

\begin{lemma}\label{le:decompBl}
Let $\ell \geq 1$ be an integer and $A$ as in \eqref{eq:defA} be such that $G^\sigma$ is $\ell$-tangle free. Then,
$$
\rho (B_{|K_0}) \leq  \PAR{\|  \uB^{(\ell)} \|  +  \frac 1{ n}   \sum_{k = 1}^\ell \| R^{(\ell)}_{k} \|} ^{1/\ell}.$$
\end{lemma}

\subsection{Estimates on random permutations}

In this subsection, we study some properties of permutations matrices $S_i$ for the symmetric random permutations. 

The first proposition gives a sharp estimate on the expected product of the variables $(\underline S_{i})_{xy}$. 
This estimate will be used to bound entries in products of $\uB^{(\ell)}, R^{(\ell)}_{k}$ and their transposes. 
Note that if $i \ne i^*$, $(\underline S_{i})_{xy}$ is centered: $\dE (\underline S_i )_{xy}  = 0$ while if $i = i^*$, 
$(\underline S_{i})_{xy}$  is almost centered, for $x \ne y$, $\dE (\underline S_{i} )_{xy}  = 1 / (n-1)- 1 / n = O (1 / n^2)$. 

We start with some definitions on colored graphs (as defined in Definition \ref{def0}). 
\begin{definition}\label{def:consistent} 
\begin{enumerate}[-]
\item Let $H$ be a colored graph with colored edge set $E_H$.  A colored edge $e = [x,i,y] \in E_H$ is {\em consistent} if for any $e' = [x',i',y'] \in E_H$, $(x,i) = (x',i')$ or $(y,i^*) = (x',i')$ implies that 
$e = e'$ (recall that $[x,i,y] = [y,i^*,x]$).   It is {\em inconsistent} otherwise.

\item For a sequence of colored edges $ ( e_1, \ldots ,e_\tau )$, the {\em multiplicity} of  $e \in \{ e_t :  1 \leq t \leq \tau \}$ is 
$\sum_t \IND (e_t  =e )$. The edge $e$ is consistent or inconsistent if it is consistent or inconsistent in the colored graph spanned by $ \{ e_t :  1 \leq t \leq \tau \}$.

\end{enumerate}
\end{definition}

 In  Figure \ref{fig:defg}, the edges 
$[1,1,6]$ and $[1,1,3]$ are inconsistent.

\begin{proposition}[\cite{bordenaveCAT}]
\label{prop:exppath}
For symmetric random permutations, there exists a  constant $c>0$ such that for any sequence of colored edges
$ (f_1, \ldots, f_{\tau})$, with $f_t  = [x_t,i_t,y_t]$, $\tau \leq \sqrt{n}$ and any $\tau_0 \leq \tau$, we have, 
$$
\ABS{ \dE \prod_{t= 1} ^{\tau_0}  (\underline S_{i_t} )_{x_t  y_{t}}  \prod_{t= \tau_0+1} ^{\tau}  ( S_{i_t} )_{x_t  y_{t}} }  \leq c \, 2^{b} \PAR{ \frac{1 }{ n }}^{e} \PAR{ \frac{  3\tau   }{ \sqrt{n} }}^{e_1}, 
$$ 
where $e = |\{ f_t : 1 \leq t \leq \tau \}|$,  $b$ is the number of inconsistent edges and $e_1$ is the number of  
$1 \leq t \leq \tau_0$ such that $[x_t, i_t , x_{t+1}]$ is consistent and has multiplicity one.\end{proposition}

\begin{proof}
Using the independence of the matrices $S_i$ (up to the involution), the claim is contained in \cite[Proposition 8]
{bordenaveCAT} for matchings and \cite[Proposition 25]{bordenaveCAT} for permutations. 
\end{proof}

Recall that the graph $G^\sigma$ is the colored graph with vertex set of $[n]$ and edges set of $[x,i,y]$ such that 
$\sigma_i (x) = y$ (and $\sigma_{i^*} (y) =x$).  

\begin{lemma}\label{le:tanglefree}
Let $A$ be as in \eqref{eq:defA} for symmetric random permutation. For some constant $c >0$, for any integer 
$1 \leq \ell \leq \sqrt n$, the expected number of cycles of length $\ell$ in $G^\sigma$ is bounded by $ c (d-1)^\ell$. 
The probability that $G^\sigma$ is $\ell$-tangled is at most $c \ell^3 (d-1)^{4\ell}  / n$. 
\end{lemma}

\begin{proof}
Let $H$ be a colored graph as in Definition \ref{def0},
with vertex set $V_H \subset [n]$ and edge set $E_H$. Let us say that $H$ is consistently 
colored if all its edges are consistent (as per Definition \ref{def:consistent}). 
If $H$ is not consistently colored then the probability that $H \subset G^\sigma$ is $0$. 
Assume that $H$ is consistently colored and that $E_H$ contains $e_i$ edges of the form $[x,i,y]$. If $i \ne i^*$, 
the probability these $e_i$ edges are present in $G^\sigma$ is
$$
\prod_{t=0} ^{e_i -1} \frac{1}{n-t} \leq \PAR{ \frac{1}{n-e_i +1}}^{e_i}.
$$
If  $i = i^*$, this probability is 
$$
\prod_{t=0} ^{e_i -1} \frac{1}{n-2t-1} \leq \PAR{ \frac{1}{n-2e_i +1}}^{e_i}.
$$
We  use that, for any integers, $k,\ell$ with $k \ell  \leq \alpha n$, $\ell \leq n/2$,  
\begin{equation*}
(n - \ell )^k \geq e^{-2\alpha} n^k,
\end{equation*}
(indeed, $(n-\ell)^k =  n^{k} \exp ( k \log ( 1 - \ell /n)) \geq n^{k} \exp ( - 2 k \ell / n)$ since $\log (1 -x) \geq - x / (1 - x)$ for 
$0 \leq x  <1$).  Using the independence of the permutations $\sigma_i$ (up to the involution), we deduce that, if 
$|E_H| \leq \sqrt n$
\begin{equation}\label{eq:subgraph}
\dP \PAR{ H \subset G^\sigma } \leq c \PAR{ \frac{1}{n}}^{|E_H|},
\end{equation}
for some constant $c>0$. 

Now, the number of properly consistently cycles in $[n]$ of length $\ell$ is at most
$$
n^\ell d (d-1)^{\ell -1},
$$
indeed, $n^\ell$ bounds the possible choices of the vertex set and $d (d-1)^{\ell -1}$ the possible colors of the edges. 
Since a cycle has $\ell$ edges, we get from \eqref{eq:subgraph} that the expected number of cycles of length $\ell$ 
contained in $G^\sigma$ is at most $c d(d-1)^{\ell-1}$ as claimed.

\begin{figure}[htb]
\begin{center}  
\resizebox{12cm}{!}{
\begin{tikzpicture}[main node/.style={circle, draw , fill = lightgray, text = black, thick}]
\node[main node] at (2,0) (6) {$x_6$} ;
\node[main node] at (1,1.7320508) (5) {$x_5$} ;
\node[main node] at (-1,1.7320508) (4) {$x_4$} ;
\node[main node] at (-2,0) (3) {$x_3$} ;
\node[main node] at (1,-1.7320508) (1) {$x_1$} ;
\node[main node] at (-1,-1.7320508) (2) {$x_2$} ;

\draw[orange!80, -,ultra thick] (1) to (2) ; 
\draw[magenta, -,ultra thick] (2) to (3) ; 
\draw[cyan, -,ultra thick] (3) to (4) ; 
\draw[magenta, -,ultra thick] (4) to (5) ; 
\draw[cyan, -,ultra thick] (5) to (6) ; 
\draw[magenta, -,ultra thick] (6) to (1) ; 
\draw[orange!80, -,ultra thick] (6) to (4) ;

\node[main node] at (10,0) (60) {$x_6$} ;
\node[main node] at (9,1.7320508) (50) {$x_5$} ;
\node[main node] at (7,1.7320508) (40) {$x_4$} ;
\node[main node] at (6,0) (30) {$x_3$} ;
\node[main node] at (9,-1.7320508) (10) {$x_1$} ;
\node[main node] at (7,-1.7320508) (20) {$x_2$} ;

\node[main node] at (12,0) (70) {$x_7$} ;
\node[main node] at (13.7320508,1) (80) {$x_8$} ;
\node[main node] at (13.7320508,-1) (90) {$x_9$} ;

\draw[orange!80, -,ultra thick] (10) to (20) ; 
\draw[magenta, -,ultra thick] (20) to (30) ; 
\draw[cyan, -,ultra thick] (30) to (40) ; 
\draw[magenta, -,ultra thick] (40) to (50) ; 
\draw[cyan, -,ultra thick] (50) to (60) ; 
\draw[magenta, -,ultra thick] (60) to (10) ; 
\draw[orange!80, -,ultra thick] (60) to (70) ; 
\draw[magenta, -,ultra thick] (70) to (80) ; 
\draw[orange!80, -,ultra thick] (80) to (90) ; 
\draw[cyan, -,ultra thick] (90) to (70) ;

\end{tikzpicture}
}
\caption{The involution $i^*$ is the identity. On the left hand side, a consistently colored $H_{6,1,4}$, on the right hand 
side a consistently colored $H_{6,3,1}$.} \label{fig:2cycles}

\end{center}\end{figure}

Similarly, if $G^\sigma$ is $\ell$-tangled, then there exists a ball of radius $\ell$ which contains two cycles. Depending on 
whether these two cycles intersect or not, it follows that $G^\sigma$ contains as a subgraph, either two cycles connected 
by a line segment or a cycle and a line segment (see Figure \ref{fig:2cycles}), where the line segment can be of length $0$. 
More formally, for integers $1 \leq s \leq k$ 
and $m \geq 1$ define $H_{k,m,s}$ as the colored graph with vertex set $\{x_t: 1 \leq t \leq k+m-1\}$ of size $k+m-1$ 
and colored edges, for $1 \leq t \leq k+m-1$, $[x_t,i_t,x_{t+1}]$, with  $x_{k+m} = x_s$, and $[x_k ,i_{k+m},x_1]$, where  
all $k+m$ edges are distinct. The graph $H_{k,m,s}$ depends implicitly on the choice of the $x_t$'s and $i_t$'s. Similarly, 
for integers $k,k'\geq 1$ and $m \geq 0$, let $H'_{k,k',m}$ be the colored graph with vertex set $\{x_t: 1 \leq t \leq k+k'+m-1\}$ 
of size $k+k'+m-1$ and colored edges $[x_t,i_t,x_{t+1}]$ for $ 1 \leq t \leq k+k'+m-1$ with $x_{k+k'+m} = x_{k+m}$, and the 
edge $[x_k ,i_{k+k'+m},x_1]$. Again, the graph $H'_{k,k',m}$ depends implicitly on the choice of the $x_t$'s and $i_t$'s. 
Then, if $G^\sigma$ is $\ell$-tangled either it contains as a subgraph, for some  $x_t$'s and $i_t$'s, a consistently 
colored graph $H_{k,m,s}$ with $k , m  \leq 2\ell$ or a consistently colored graph $H_{k,k',m}$ with $k ,  k' + m \leq 2\ell$.

The number of consistently colored graphs $H_{k,m,s}$ in $[n]$ is at most
$$
n^{k+m-1} d (d-1)^{k+m-1},
$$
and the number of consistently colored graphs $H'_{k,k',m}$ in $[n]$ is at most
$$
n^{k+k'+m-1} d (d-1)^{k+k'+m-1} ,
$$

From \eqref{eq:subgraph}, we deduce that the probability that $G^\sigma$ is $\ell$-tangled is at most 
$$
\sum_{k , s , m \leq  2 \ell} n^{k+m-1} d (d-1)^{k+m-1}  c \PAR{ \frac{1}{n}}^{k+m} +\sum_{k , k' + m \leq 2 \ell}n^{k+k'+m-1} d (d-1)^{k+k'+m-1}  c \PAR{ \frac{1}{n}}^{k+k'+m}.
$$
The latter is $O \PAR{ \frac{\ell^3 (d-1)^{4\ell}}{n} }$ as claimed.
\end{proof}

\subsection{Trace method of F\"uredi and Koml\'os}

\label{subsec:FK}
\subsubsection{Norm of $\uB^{(\ell)}$}

Here, we give a sharp bound on the operator norm of the matrices $\uB^{(\ell)}$ for symmetric random permutations. 
In this subsection, we fix a collection $(a_i), i \in [d],$ of matrices such that $\max_i ( \| a_i \| \vee  \| a_i ^{-1} \|^{-1} ) \leq \veps^{-1}$ for some $\veps >0$. Then $B_\star$ is 
the corresponding non-backtracking operator in the free group. The constants may depend implicitly on $r$, $d$ and $\veps$.

\begin{proposition} \label{prop:normB} 
Let $\veps > 0$. If $1 \leq \ell \leq  \log n$, then the event
\begin{equation*}
 \| \uB^{(\ell)} \| \leq ( \log n) ^{20} ( \rho (B_\star) + \veps ) ^\ell,
\end{equation*}
holds with the probability at least $1 -  c e^{-\frac{\ell \log n}{ c \log \log n}} $ where $c >0$ depends on $r$, $d$ and $\veps$.

\end{proposition}

The proof relies on the method of moments.
Let $m$ be a positive integer. With the convention that $f_{2m + 1} = f_1$, we get 
\begin{eqnarray}
\|\uB^{(\ell)}  \| ^{2 m} = \|\uB^{(\ell)}{\uB^{(\ell)}}^*  \| ^{m} & \leq & \tr \BRA{ \PAR{  \uB^{(\ell)}{\uB^{(\ell)}}^*}^{m}  } \nonumber\\
& = & \sum_{(f_1, \ldots, f_{2m})\in E^{2m}} \tr  \prod_{j=1}^{m}  (\uB^{(\ell)}) _{f_{2j-1} , f_{2 j}} ( {\uB^{(\ell)}}^* )_{f_{2j} , f_{2 j+1}} \nonumber \\
& =  &  \sum_{\gamma \in W_{\ell,m} }\prod_{j=1}^{2m}   \prod_{t=1}^{\ell} (\underline S_{i_{j,t}} )_{x_{j,t}  x_{j,t+1}}  \tr \prod_{j=1}^{2m} a( \gamma_j)^{\veps_j}   ,  \label{eq:trBl}
\end{eqnarray}
where $a(\gamma_j)^{\veps_j} = a(\gamma_j)$ or $a(\gamma_j)^*$ depending on the parity of $j$ and $W_{\ell,m}$ is the set of  $\gamma = ( \gamma_1, \ldots, \gamma_{2m})$ such that $\gamma_j = (\gamma_{j,1} , \ldots, \gamma_{j,\ell+1}) \in F^{\ell+1}$, $\gamma_{j,t} = ( x_{j,t}, i_{j,t})$ and for all $j = 1, \ldots, m$, 
\begin{equation}\label{eq:bound}
\gamma_{2j,1} = \gamma_{2j+1, 1} \quad \hbox{ and } \quad  \gamma_{2j-1,\ell+1} = \gamma_{2j, \ell+1},
\end{equation}

with the convention that $\gamma_{2m+1} = \gamma_{1}$, see Figure \ref{fig:fleur}. The proof of Proposition \ref{prop:normB} 
is based on an upper bound on the expectation of the right hand side of \eqref{eq:trBl}. We write 
\begin{eqnarray}
\dE \|\uB^{(\ell)}  \| ^{2 m} & \leq  &   \sum_{\gamma \in W_{\ell,m} } |w(\gamma)| \,   \tr | a ( \gamma ) |  ,  \label{eq:trBl2}
\end{eqnarray}
where we have set 
$$
w(\gamma) = \dE \prod_{j=1}^{2m}  \prod_{t=1}^{\ell} (\underline S_{i_{j,t}} )_{x_{j,t}  x_{j,t+1}} \AND a(\gamma) = \prod_{j=1}^{2m}  a ( \gamma_j)^{\veps_j}.
$$

\begin{figure}[htb]
\begin{center}  
\resizebox{10cm}{!}{
\begin{tikzpicture}[main node/.style={circle,fill , text = black, thick}]
\node  at (0,0) (1) {} ;
\node at (2,0) (2) {} ;
\node[right] at (3,1.73205) (3) {{\tiny $\gamma_{1,\ell+1} = \gamma_{2,\ell+1}$}}   ;
\node at (2,3.4641016) (4) {} ;
\node at (0,3.4641016) (5) {} ;
\node[left] at (-1,1.73205) (6) {{\tiny $\gamma_{2i-1,\ell+1} = \gamma_{2i,\ell+1}$}}  ;

\node at (1,0.8660254) (a) {} ;
\node[below] at (1.8660254,1.3660254) (b) {{\tiny \hspace{10pt}$\gamma_{1,1} = \gamma_{12,1}$}} ;
\node  at (1.8660254,2.23205) (c) {} ;
\node at (1,2.5980762) (d) {} ;
\node  at (0.1339746,2.23205) (e) {} ;
\node[below ]  at (0.1339746,1.3660254) (f) {{\tiny\hspace{-5pt} $\gamma_{2i,1} = \gamma_{2i+1,1}$}}  ;

\draw[cyan , ->,thick] (1.8660254,1.3660254)  [out = 0 , in = 180] to (3) {} ;  \node at (2.3,1.46) {{\small $\gamma_1$}} ; 
\draw[cyan ,<-,thick] (3)  [out = 180 , in = 0] to (1.8660254,2.23205) {} ;  \node at (2.3,2) {{\small $\gamma_2$}}  ; 
\draw[cyan ,->,thick] (1.8660254,2.23205)  [out = 60 , in = -120] to (4)   ; 
\draw[cyan ,-,thick] (4)  [out = -120 , in = 60] to (1,2.5980762)  ; 
\draw[cyan ,->,thick] (1,2.5980762)  [out = 120 , in = -60] to (5)   ; 
\draw[cyan ,-,thick] (5)  [out = -60 , in = 120] to (0.1339746,2.23205)   ; 
\draw[cyan ,->,thick] (0.1339746,2.23205) [out = 180 , in = 0] to (6)   ;  \node at (-0.339746,2) {{\small $\gamma_{2i-1}$}} ; 
\draw[cyan ,-,thick] (6)  [out = 0 , in = 180] to (0.1339746,1.3660254)  ; \node at (-0.339746,1.46) {{\small $\gamma_{2i}$}} ; 
\draw[cyan ,->,thick] (0.1339746,1.3660254)  [out = -120 , in = 60] to (1)   ; \node at (0,0.8) {{\small $\gamma_{2i+1}$}} ; 
\draw[cyan ,-,thick] (1)  [out = 60 , in = -120] to (1,0.8660254)   ; 
\draw[cyan ,->,thick] (1,0.8660254)  [out = -60 , in = 120] to (2) {} ; \node at (2.05,0.8)  {{\small $\gamma_{12}$}}  ; 
\draw[cyan ,-,thick] (2)  [out = 120 , in = -60] to (1.8660254,1.3660254)    ;

\end{tikzpicture}
}
\caption{A path $\gamma = (\gamma_1, \ldots , \gamma_{12})$ in $W_{\ell,6}$, each $\gamma_i$ is tangle-free.} \label{fig:fleur}

\end{center}\end{figure}

First, to deal with this large sum, we partition  $W_{\ell,m}$ in isomorphism classes. Permutations on $[n]$ and $[d]$ act 
naturally on $W_{\ell,m}$.  We consider the isomorphism class $\gamma \sim \gamma'$ if there exist $\sigma \in \cS_n$ and 
$(\tau_x)_x \in (\cS_d) ^n$ such that, with $\gamma'_{j,t}  = (x'_{j,t} , i'_{j,t})$, for all $1 \leq j \leq 2m$, $1 \leq t \leq \ell+1$,
 $x'_{j,t} = \sigma ( x_{j,t})$, $i'_{j,t}  = \tau_{x_{j,t}}( i_{j,t})$ and $(i'_{j,t}  ) ^*  = \tau_{x_{j,t+1}}( (i_{j,t}) ^* )$.  
 For each $\gamma \in W_{\ell,m}$, we define $G_\gamma$ as in Definition \ref{def1}, 
 $V_\gamma = \cup_j V_{\gamma_j} = \{ x_{j,t}:  1 \leq j \leq 2m , 1 \leq t \leq \ell +1\}$ and 
 $E_\gamma = \cup_j E_{\gamma_j} = \{ [ x_{j,t} , i_{j,t} , x_{j,t+1} ]:   1 \leq j \leq 2m , 1 \leq t \leq \ell \}$ are the sets of 
 visited vertices and visited pairs of colored edges along the path.  Importantly, $G_\gamma$ is connected.   
 We may then define a canonical element in each isomorphic class as follows. We say that a path $\gamma \in W_{\ell,m}$ 
 is {\em canonical} if $\gamma$ is minimal in its isomorphism class for the lexicographic order 
 ($x$ before $x+1$ and $(x,i)$ before $(x,i+1)$), that is $\gamma_{1,1} = (1,1)$ and $\gamma_{j,t}$ minimal over all 
 $\gamma'_{j,t}$ such that $\gamma' \sim \gamma$ and $\gamma'_{k,s} = \gamma_{k,s}$ for all $(k,s) \prec (j,t)$.  
 Our first lemma bounds the number of isomorphism classes. This lemma is  a variant of \cite[Lemma 17]{BLM} 
 and \cite[Lemma 13]{bordenaveCAT}. It relies crucially on the fact that an element $\gamma \in W_{\ell,m}$ is 
 composed of $2m$ tangle-free paths.  

\begin{lemma}\label{le:enumpath}
Let $\cW_{\ell,m} (v,e) $ be the subset of canonical paths with $|V_\gamma| = v$ and $|E_\gamma |= e$. We have 
$$
| \cW _{\ell,m} (v,e) | \leq   (2 d \ell  m )^{6 m \chi  + 10 m },
$$
with $\chi = e - v +1 \geq 0$. 
\end{lemma}

\begin{proof}
We bound $| \cW_{\ell,m} ( v, e) | $ by constructing an encoding of the canonical paths (that is, an injective map from 
$\cW_{\ell,m}(v,e)$ to a  set whose cardinality is easily upper bounded).  For $i \leq i \leq 2m$ and $1 \leq t \leq \ell$, 
let $e_{j,t} = ( x_{j,t},i_{j,t}, x_{j,t+1} )$ and $[e_{j,t}]=   [ x_{j,t},i_{j,t}, x_{j,t+1} ] \in E_\gamma$ the corresponding colored edge. 
We explore the sequence $(e_{i,t})$ in lexicographic order denoted by $\preceq$ (that is $(j,t)\preceq (j+1,t')$ and 
$(j,t)\preceq(j,t+1)$). We think of the index $(j,t)$ as a time. 
We define $(j, t)^-$ as the largest time smaller than $(j, t)$, i.e. $(j, t)^- = (j, t - 1)$ if 
$t\geq 2$, $ (j, 1)^- = (j- 1, \ell) $ if $j \geq 2$ and, by convention, $(1, 1)^- = (1, 0).$

We denote by $G_{(j,t)}$ the graph spanned by the edges $ \{ [e_{j',t'} ]: (j',t')\preceq (j,t)\}$. The graphs $G_{(j,t)}$ are 
non-decreasing over time and by definition $G_{(2m,\ell)} = G_\gamma$. We may define a growing spanning forest 
$T_{(j,t)}$ of $G_{(j,t)}$ as follows: $T_{(1,0)}$ has no edge and a single vertex, $1$. Then, $T_{(j,t)}$ is obtained 
from $T_{(j,t)^-}$ by adding the edge $[e_{j,t}]$ if its addition does not create a cycle in $T_{(j,t)^-}$. 
We then say that $[e_{j,t}]$ is a {\em tree edge}. By construction $T_{(j,t)}$ is spanning forest of $G_{(j,t)}$ 
and $T_\gamma = T_{(2m,\ell)}$ is a spanning tree of $G_\gamma$. An edge $[e_{j,t}]$ in 
$G_\gamma \backslash T_\gamma$ is called an {\em excess edge}. Since  $T_\gamma$ has $v-1$ edges, we have
\begin{equation}\label{eq:defchi}
\chi = \ABS{ \BRA{ f \in E_\gamma: \hbox{ $f$ is an excess edge}} } = e - v +1 \geq 0.
\end{equation}

\begin{figure}[htb]
\begin{center}  
\resizebox{14cm}{!}{
\begin{tikzpicture}[main node/.style={circle, draw , fill = lightgray, text = black, thick}]
\node[main node]  at (0,0) (1) {1} ;
\node[main node] at (2,0) (2) {2} ;
\node[main node] at (4,0) (3) {3} ;
\node[main node] at (3,1.7320508 ) (4) {4} ;
\node[main node] at (6,0) (5) {5} ;

\draw[cyan, -,ultra thick] (1) to (2) ; 
\draw[ magenta, -,ultra thick] (2) to (3) ; 
\draw[ cyan, -,ultra thick] (3) to (4) ;  
  \draw[orange!80,  -,ultra thick] (4) to (2) ;  
 \draw[cyan,  -, ultra thick] (3) to (5) ;

  \node[text = black!80] at (1,0) {1} ; 
    \node[text = black!80] at (3,0) {2,5,8,11} ; 
    \node[text = black!80] at (5,0) {12} ;   
       \node[text = black!80] at (3.6,0.8660254) {3,6,9} ; 
    \node[text = black!80] at (2.4,0.8660254) {4,7,10} ;

  \node[main node]  at (10,0) (10) {1} ;
\node[main node] at (12,0) (20) {2} ;
\node[main node] at (14,0) (30) {3} ;
\node[main node] at (13,1.7320508 ) (40) {4} ;
\node[main node] at (16,0) (50) {5} ;

\draw[cyan, -,ultra thick] (10) to (20) ; 
\draw[magenta, -,ultra thick] (20) to (30) ; 
\draw[cyan, -,ultra thick] (30) to (40) ;  
 \draw[cyan,  -, ultra  thick] (30) to (50) ;

\end{tikzpicture}
}
{\small $$\gamma_1 =  (1,1)     (2,2)     (3,1)    (4,3)    (2,2) (3,1)    (4,3)   (2,2)   (3,1)   (4,3)   (2,2) (3,1)     (5,2) $$} 

\vspace{-25pt}

\caption{A canonical path $\gamma_1 \in F^{13}$ and its associated spanning tree, the involution $i^*$ is the identity. 
The times $(1,t)$ with $t \in \{ 1,2,3,12\}$ are first times and  $ t= \{ 4,7, 10\} $ are important times, $(1,4)$ is the short 
cycling time, $(1,7), (1,10)$ are superfluous.  With the notation below, 
$t_1 = 4$, $t_0= 2$, $t_2 = 12$, $\tau = 13$.} \label{fig:impo}

\end{center}\end{figure}

Now, from \eqref{eq:bound}, for each $j$, there is a smallest time $(j,\sigma)$, which we call {\em the merging time}, 
such that $G_{j,\sigma}$ will be connected. By convention, if $x_{j,1} \in G_{(j,1)^-}$, we set $\sigma  =0$ 
(for example from \eqref{eq:bound} if $j$ is odd, $\sigma = 0$).  We say that $(j,t)$ is a {\em first time}, 
if it is not a merging time and if $[e_{j,t}]$ is a tree edge which has not been seen before (that is $e_{j,t} \ne e_{k, s}$ for 
all $(k,s) \preceq (j,t)$). We say that $(j,t)$ is an {\em important time} if $[e_{j,t}]$ is an excess edge (see Figure \ref{fig:impo}). 

By construction, since the path $\gamma_j$ is non-backtracking,  it can be decomposed by the successive repetition 
of (i) a sequence of first times (possibly empty), (ii) an important time or a merging time and (iii) a path on the forest defined 
so far (possibly empty). Note also that, if $t \geq 2$ and $(j,t)$ is a first time, then $i_t = p$ and $x_{j,t+1} = m +1$ 
where $m$ is the number of previous first times (including $(j,t)$) and $p$ is minimal over all $i\geq 1$, such that 
$i\ne i^*_{t-1}$. Indeed, since $\gamma$ is canonical, every time that a new vertex in $V_\gamma$ is visited its number 
has to be minimal, and similarly for the number of the color of a new edge. It follows that if $(j,t), \ldots, (j,t+s)$ are first 
times and $x_{j,t}$ is known then the values of $e_{j,t},\ldots, e_{j,t+s}$ can be unambiguously computed.

We can now build a first encoding of $\cW_{\ell,m}$. If $(j,t)$ is an  important time, we mark the time $(j,t)$  by the 
vector $(i_{j,t},x_{j,t+1},x_{j,\tau})$, where $(j,\tau)$ is the next time that $e_{j,\tau}$ will not be a tree edge of the forest 
$T_{j,t}$ constructed so far (by convention, if the path $\gamma_j$ remains on the tree, we set $\tau = \ell+1$). 
For $t=1$, we also add the {\em starting mark} $(x_{j,1},\sigma, x_{j,\tau})$ where $\sigma$ is the merging time and 
$(j,\tau) \geq (j,\sigma)$ is as above the next time that $[e_{j,\tau}]$ will not be a tree edge of the forest constructed so far.  
Since there is a unique non-backtracking path between two vertices of a tree, we can reconstruct 
$\gamma \in \cW_{\ell,m}$ from the starting marks and the position of the important times and their marks. 
It gives rise to a first encoding.

In this encoding, the number of important times could be large  (see Figure \ref{fig:impo}). We will now use  the assumption 
that each path $\gamma_j$ is tangle-free to partition important times into three categories, {\em short cycling}, 
{\em long cycling} and {\em superfluous} times. For each $j \in [2m]$, we consider the first occurrence of a time $(j,t_1)$ 
such that $x_{j,t_1+1} \in \{ x_{j,1},  \ldots, x_{j,t_1} \}$. If such $t_1$ exists, the last important time $(j,t_s) \preceq (j,t_1)$ 
will be called the short cycling time.  Let $1 \leq t_0 \leq t_1$ be such that $x_{j,t_0} = x_{j,t_1+1}$. 
By assumption, $C_j = (e_{j,t_0},\ldots, e_{j,t_1-1})$ will be the unique cycle visited by $\gamma_j$.  
We denote by $(j,t_2)$ the next $t_2 \geq  t_1$ such that $e_{j,t_2}$ in not in $C_j$ (by convention $t_2 = \ell+1$ if 
$\gamma_j$ remains in $C_j$). 
We modify the mark of the short cycling time $(j,t_s)$ as  $(i_{j,t_s},x_{j,t_s+1}, x_{j,t_1}, t_2 ,x_{j, \tau})$, where 
$(j,\tau) \succeq (j, t_2)$, is the next time that $[e_{j,\tau}]$ will not be a tree edge of the forest constructed so far. 
Important times $(j,t)$ with $1 \leq t < t_s$ or $\tau \leq t \leq \ell$ are called long cycling times, they receive the usual mark 
$(i_{j,t},x_{j,t+1},x_{j,\tau})$. The other important times are called superfluous. By convention, if there is no short cycling time, 
we call anyway, the last important time, the short cycling time. We observe that for each $j$, the number of long cycling times 
on $\gamma_j$ is bounded by  $\chi-1$ (since there is at most one cycle, no edge of $E_\gamma$ can be seen twice 
outside those of $C_j$, the $-1$ coming from the fact the short cycling time is an excess edge). 

We now have our second encoding. We can reconstruct $\gamma$ from the starting marks, the positions of the long cycling 
and the short cycling times and their marks. For each $j$, there are at most $1$ short cycling time and $ \chi-1$ long cycling 
times. There are at most $  (\ell +1)^{2m \chi}$ ways to position them. There are at most $d v^2$ different possible marks for a
long cycling time and $d v^3 (\ell+1)$ possible marks for a short cycling time. Finally, there are $v^2(\ell+1) $ possibilities for a 
starting mark. We deduce that    
$$
| \cW _{\ell,m} (v,e) | \leq    (\ell +1)^{2 m \chi} (v^2 (\ell+1) )^{2m}   (d v^2 ) ^{2m (\chi-1)}(d v^3 (\ell+1)) ^{2m}. 
$$
Using $v \leq 2 \ell m+1$ and $\ell +1 \leq 2 \ell$, we obtain the claimed bound. \end{proof}

Our second lemma bounds the sum of $a (\gamma)$ in an equivalence class.

\begin{lemma}\label{le:sumpath} Let $\rho = \rho( B_\star) + \veps$ and $k_0$ be a positive integer. 
Then, there exists a constant 
$c >0$ depending on $r$, $d$ and $\veps$ such that for any $\gamma\in \cW_{\ell,m}(v,e)$, 
\begin{equation}\label{eq:sumpath}
\sum_{\gamma':  \gamma' \sim \gamma} \tr |a (\gamma' ) | \leq  c^{m + \chi + e_1} n^v  \rho_0 ^{2(\ell m -v )}  \rho^{2v},
\end{equation}
where $\chi = e - v +1$, $e_1$ is the number of edges of $E_\gamma$ with multiplicity one and 
$$
\rho_0 = \max \NRM{ \prod_{s=1}^{k_0} a_{i_s} }^{\frac 1 {k_0}},
$$ 
and the maximum is over all non-backtracking sequences $(i_1, \cdots, i_{k_0})$, that is $i_{s+1} \ne i_{s ^*}$. Moreover, for all $k_0$ large enough, we have $\rho_0 \leq \rho$. \end{lemma}

\begin{proof}
We start by proving \eqref{eq:sumpath}. The proof relies on a decomposition of $G_\gamma$ where the path is split into 
sub-paths on the free group. Let $v_k$ (respectively $v_{\geq k}$) be the set of vertices of $G_\gamma$ of degree 
$k$ (respectively $\geq k$). We have 
$$
v_1 + v_2 + v_{\geq 3} = v \AND v_1 + 2 v_2 + 3 v_{\geq 3} \leq \sum_k k v_k = 2e. 
$$
Subtracting from the right hand side twice the left hand side, we deduce that 
\begin{equation*}\label{eq:boundv3}
v_{\geq 3} \leq 2( e - v) + v_1 \leq 2 \chi  + 2m - 2.
\end{equation*}
Indeed, at the last step the bound $v_1 \leq 2m$ follows from the observation that since each $\gamma_j$ is non-backtracking,
 only a vertex $x \in V_\gamma$ such that $x = x_{j,1}$ or $x = x_{j,\ell+1}$ for some $ 1 \leq j \leq 2m$ can be of degree $1$. 

Now, consider the set $V'_\gamma \subset V_\gamma$ formed by vertices of degree at least $3$ and vertices 
$x \in V_\gamma$ such that $x = x_{i,1}$ or $x = x_{i,\ell+1}$ for some $ 1 \leq i \leq 2m$. From what precedes, 
$$v' = |V'_\gamma | \leq 2 \chi + 4 m - 2.$$ 

We now build the graph $G'_\gamma$ on $V'_\gamma$ obtained from $G_\gamma$ by merging degree $2$ vertices  along 
edges. More formally, let $P_\gamma$ be the set of non-backtracking sequences 
$\pi = (y_1,i_1, \ldots, y_{k}, i_{k}, y_{k+1})$ with $[y_{s},i_s,y_{s+1}] \in E_\gamma$ for 
$1 \leq s \leq k$ and $y_1 , y_{k+1} \in V'_\gamma$, $y_2, \ldots, y_{k} \in V_\gamma \backslash V'_\gamma$. 
We set $\pi^* = (y_{k+1}, i_{k}^* , y_{k}, \ldots, y_1) \in P_\gamma$. Since all vertices not in $V'_\gamma$ have degree 
$2$, two distinct paths $\pi,\pi' \in P_\gamma$ are either disjoint (except the endpoints) or $\pi^* = \pi'$.  
As in Definition \ref{def0}, we define a (generalized) colored edge as an equivalence class $[\pi]$ of   $\pi$ in 
$P_\gamma$ endowed with the equivalence $\pi \sim  \pi' $ if $ \pi' \in \{  \pi ,  \pi^* \}$. 
Then $G'_\gamma = (V'_\gamma, E'_\gamma)$ is the colored graph with edge set $E'_\gamma$, the set of $[\pi]$ with 
$\pi =  (y_1,i_1, \ldots, y_{k}, i_{k}, y_{k+1}) \in P_\gamma$, $[\pi]$ being an edge between $y_1$ and $y_{k+1}$, see 
Figure \ref{fig:impo2}.  Let $e' = |E'_\gamma|$. We find easily that this operation of merging degree $2$ vertices preserves 
the Euler characteristic (if $\pi = (y_1,i_1, \cdots, y_{k}, i_k, y_{k+1})$ is in $P_\gamma$, it replaces $k$ edges and $k-1$ 
vertices of $G_\gamma$ by a single edge in $G'_\gamma$), that is
$$
e' - v' + 1 = e - v +1 = \chi.
$$
It follows that 
\begin{equation}\label{eq:eprime}
e' \leq 3 \chi  + 4m -3.
\end{equation}

\begin{figure}[htb]
\begin{center}  
\resizebox{14cm}{!}{
\begin{tikzpicture}[main node/.style={circle, draw , fill = lightgray, text = black, thick}]
\node[main node]  at (0,0) (1) {1} ;
\node[main node] at (2,0) (2) {2} ;
\node[main node] at (4,0) (3) {3} ;
\node[main node] at (3,1.7320508 ) (4) {4} ;
\node[main node] at (6,0) (5) {5} ;

\draw[cyan, -,ultra thick] (1) to (2) ; 
\draw[ magenta, -,ultra thick] (2) to (3) ; 
\draw[ cyan, -,ultra thick] (3) to (4) ;  
  \draw[orange!80,  -,ultra thick] (4) to (2) ;  
 \draw[cyan,  -, ultra thick] (3) to (5) ;
  \draw[magenta,-,ultra thick] [out = 30 , in = -30] (3) to (4.5,0.8660254 ) ; \draw[magenta,ultra thick] [out = 150 , in = 90]  (4.5,0.8660254 )  to (3) ; 
  \draw[magenta,-,ultra thick] [out = 150 , in = 210] (2) to (1.5,0.8660254 ) ; \draw[magenta,ultra thick] [out = 30 , in = 90]  (1.5,0.8660254 )  to (2) ;

  \node[main node]  at (10,0) (10) {1} ;
\node[main node] at (12,0) (20) {2} ;
\node[main node] at (14,0) (30) {3} ;
\node[circle,draw] at (13,1.7320508 ) (40) {4} ;
\node[main node] at (16,0) (50) {5} ;

\draw[cyan, -,ultra thick] (10) to (20) ; 
\draw[magenta, -,ultra thick] (20) to (30) ; 
  \draw[orange!80,  -,ultra thick] (40) to (20) ;
\draw[cyan, -,ultra thick] (30) to (40) ;  

 \draw[cyan,  -, ultra  thick] (30) to (50) ;

 \draw[magenta,-,ultra thick] [out = 30 , in = -30] (30) to (14.5,0.8660254 ) ; \draw[magenta,ultra thick] [out = 150 , in = 90]  (14.5,0.8660254 )  to (30) ; 
  \draw[magenta,-,ultra thick] [out = 150 , in = 210] (20) to (11.5,0.8660254 ) ; \draw[magenta,ultra thick] [out = 30 , in = 90]  (11.5,0.8660254 )  to (20) ;

\end{tikzpicture}
}
{\small$$ \gamma_1 =  (1,1)     (2,2)     (3,1)    (4,3)    (2,2)     (3,1)     (5,2) $$
$$
\gamma_2 = (1,1)      (2,2)   (2,3)   (4,1)    (3,2)     (3,1)     (5,2) 
$$} 
\vspace{-25pt}
\caption{A canonical path $\gamma = (\gamma_1,\gamma_2) \in \cW_{6,1}$ and its associated graphs $G_\gamma$ and $G'_\gamma$, the involution $i^*$ is the identity. We have $V_\gamma \backslash V'_\gamma = \{ 4\}$ and $[\pi]$ is an edge of $G'_\gamma$ with $\pi = (2,3,4,1,3)$. This edge has multiplicity $2$. The edge $[(1,1,2)]$ has multiplicity $2$.} \label{fig:impo2}

\end{center}\end{figure}

Now, we recall  the multiplicity introduced above Proposition \ref{prop:exppath}. If $[x,i,y] \in E_\gamma$, the {\em multiplicity} 
of $[x,i,y]$, denoted by $m_{[x,i,y]}$, is the number of times that $[\gamma_{j,s}, i_{j,s}, \gamma_{j,s+1} ] = [x,i,y]$.  Since 
$\gamma$ is non-backtracking along each edge of $E'_\gamma$, we observe that if $[\pi] \in E'_\gamma$ with 
$\pi  =  (y_1,i_1, \ldots, y_{k}, i_{k}, y_{k+1})$ then all edges $[y_s,i_s,y_{s+1}]$ have the same multiplicity. 
We may thus unambiguously define the  multiplicity $m_{[\pi]}$ of an edge $[\pi]  \in E'_\gamma$,  see Figure \ref{fig:impo2}.  Let $e_t$  be the number of edges of multiplicity equal to $t$. We have
\begin{equation}\label{eq:propet}
\sum_t e_t = e \AND  \sum_t t  e_t = 2\ell m. 
\end{equation}
We find
\begin{equation} \label{eq:bounde3}
\sum_{t} ( t -2 )_+ e_t = \sum_{t} ( t -2 ) e_t  + e_1 =  2(\ell m - e )+e_1.
\end{equation}

Since the path $\gamma_j$ is non-backtracking, we may decompose it into successive visits of the edges of $E'_\gamma$. 
More precisely, we decompose $\gamma_j $ as $\gamma_j = (p_{j,1}, p_{j,2}, \ldots , p_{j,k_j})$ where either (i) $p_{j,t}$ 
follows an edge of $E'_\gamma$ which is visited for the first or second time, or (ii) $p_{j,t}$ follows a sequence  of edges of 
$E'_\gamma$ which have been visited previously at least two times. By construction, in the decomposition of the whole path 
$\gamma$, there are at most $2e'$ subpaths $p_{j,t}$ of type $(i)$ and thus $2 e' + 4 m$ subpaths of type $(ii)$. 
We may then write 
\begin{eqnarray}\label{eq:agamma}
\| a (\gamma) \|   \leq   \prod_{j=1}^{2m} \NRM{ a ( \gamma_j ) }  \leq  \delta   \prod_{ p_{j,t} }   \NRM{ \prod_{s= 1}^{k} a_{i_s}} = \delta  \prod_{ p_{j,t}   \,  {\scriptsize \hbox{type $(i)$}}}   \NRM{ \prod_{s= 1}^{k} a_{i_s}}  \prod_{ p_{j,t}   \,  {\scriptsize \hbox{type $(ii)$}}}   \NRM{ \prod_{s= 1}^{k} a_{i_s}},
\end{eqnarray}
where in the above product $p_{j,t} = (y_1, i_1, y_2, i_2, \ldots , y_k,  i_{k})$ and 
$$
\delta =  \prod_{j=1}^{2m} \frac{\|  a_{i_{j,\ell+1}} \| }{ \| a_{i_{j,1}}\|}  \leq  \veps^{-4m}, 
$$
accounts for the boundary effects.  To estimate \eqref{eq:agamma}, we shall use the two rough bounds 
\begin{eqnarray}\label{eq:defe12}
 \NRM{ \prod_{s= 1}^{k} a_{i_s}} \leq  \NRM{ \prod_{s= 1}^{k} a_{i_s}}^2   \NRM{\PAR{ \prod_{s= 1}^{k} a_{i_s}}^{-1} } \leq \veps ^{-k}  \NRM{\prod_{s= 1}^{k} a_{i_s} }^2.
 \end{eqnarray}
We notice also, since $ \max_i \| a_i \|    \leq \veps^{-2} \rho _0$, 
\begin{eqnarray}\label{eq:defe13}
\NRM{ \prod_{s= 1}^{k} a_{i_s}} &\le & \PAR{ \prod_{s=1}^{\lfloor k / k_0 \rfloor} \| a_{i_{k_0 s-k_0 + 1}}\cdots  a_{i_{k_0 s}} \|} \| a_{i_{ k_0 \lfloor k / k_0 \rfloor+1} } \cdots a_{i_k}  \|  \\
& \leq & \rho_0 ^{\lfloor k / k_0 \rfloor} \max_i \| a_i\|^{ k - k_0 \lfloor k / k_0 \rfloor} \leq \veps^{-2k_0} \rho_0^{k},
\end{eqnarray}
(which  uses the non-backtracking condition $i_{s+1} \ne i^*_{s}$). 
Now, in \eqref{eq:agamma}, we decompose the product over $p_{j,t}$ of type $(i)$ and of type $(ii)$. {\em We first assume for simplicity that the symmetry condition \eqref{eq:sym} holds}. Then due to \eqref{eq:sym}, we note that the norm of product of $a_i$'s along an edge $[\pi ]$ does not  depend on whether we take the product along $\pi $ or $\pi^*$:
\begin{equation*}\label{eq:symconddede}
  \NRM{ \prod_{s= 1}^{k} a_{i_s}}  = \NRM{ \PAR{\prod_{s= 1}^{k} a_{i_s}}^* } =   \NRM{ \prod_{s= 1}^{k} a_{i_{k-s+1}^*}}. 
\end{equation*}
   Using  this and \eqref{eq:defe12} when $m_{[\pi]}= 1$, we arrive at 
\begin{equation}\label{eq:pjti}
 \prod_{ p_{j,t}   \,  {\scriptsize \hbox{type $(i)$}}}   \NRM{ \prod_{s= 1}^{k} a_{i_s}}  =   \prod_{[\pi ]  \in E'_\gamma}   \NRM{ \prod_{s= 1}^{k} a_{i_s}} ^{m_{[\pi]}  \wedge 2} \hspace{-20pt} \leq  \prod_{[\pi ]  \in E'_\gamma} ( \veps^{-k} )^ {\IND_{(m_{[\pi]}  = 1)} } \NRM{ \prod_{s= 1}^{k} a_{i_s}} ^{2}   =  \veps ^{-e_1}  \prod_{[\pi ]  \in E'_\gamma} \NRM{\prod_{s= 1}^{k} a_{i_s} }^2,
\end{equation}
where in the above product $\pi =  (y_1,i_1, \ldots, y_{k}, i_{k}, y_{k+1})$.  Similarly, for each $p_{j,t}$ of type $(ii)$, we use 
\eqref{eq:defe13}. Since there are at most $2e' + 4m$ subpaths  $p_{j,t}$ of type  $(ii)$ and since the sum of length of 
$p_{j,t}$ is equal to $\sum_{t} ( t -2 )_+ e_t =2(\ell m - e )+e_1$ by \eqref{eq:bounde3}, from \eqref{eq:defe13}, we find
$$
 \prod_{ p_{j,t}   \,  {\scriptsize \hbox{type $(ii)$}}}   \NRM{ \prod_{s= 1}^{k} a_{i_s}}  \leq    \prod_{ p_{j,t}   \,  {\scriptsize \hbox{type $(ii)$}}}   \veps^{-2k_0} \rho_0^{k}  \leq \veps^{ -2k_0 ( 2e' + 4m)} \rho_0^{2(\ell m - e )+e_1}.
$$
We finally plug the last two upper bounds into \eqref{eq:agamma}. Using \eqref{eq:eprime}, for some $c >0$, we arrive at  
\begin{equation*}\label{eq:boundag}
\| a (\gamma) \|   \leq  c ^{m+\chi + e_1 } \rho_0^{2(\ell m - e)}  \prod_{[\pi ]  \in E'_\gamma} \NRM{\prod_{s= 1}^{k} a_{i_s} }^2.
\end{equation*}
Thus, summing over all $\gamma' \sim \gamma$, we obtain
\begin{equation}\label{eq:boundag2} 
\sum_{\gamma':  \gamma' \sim \gamma} \|a (\gamma' ) \| \leq  c ^{m+ \chi + e_1} \rho_0^{2 (\ell m - e)} n^v \prod_{[\pi ]  \in E'_\gamma} \PAR{ \sum \NRM{\prod_{s= 1}^{k} a_{i_s} }^2},
\end{equation}
where, for $\pi  =  (y_1,i_1, \ldots, y_{k}, i_{k}, y_{k+1})$, the sum is over all non-backtracking sequence $(i_1, \ldots, i_{k})$.

Finally, in this last expression,  $  \sum \NRM{\prod_{s= 1}^{k} a_{i_s} }^2  $ can be bounded in terms of the spectral radius 
of the non-backtracking on the free group. Let $B_\star$ be the non-backtracking operator on the free group associated to 
$A_\star$ defined in \eqref{eq:defAfree}.  There exists $c>0$ such that for any $(g,i)\in X \times [d]$, for any integer $k \geq 0$, 
\begin{equation*}
\| B^k_\star \delta_{(o,i)} \|^2_2 = \sum \NRM{\prod_{s= 2}^{k+1} a_{i_s} }^2 \leq \max (  c  , \rho^{2k} ) ,
\end{equation*}
where the sum is over all non-backtracking sequence $(i_1, \ldots, i_{k+1})$ such that $i_1 = i$. Moreover from Lemma  
\ref{le:contrho}, the same constant $c$ may be taken for all $B_\star$ with weights such that $\max_i (\|a_i\|)  \leq \veps^{-1}$. 
Also, at the cost of changing the constant $c$ and taking $k_0$ large enough, we find in  \eqref{eq:boundag2}
\begin{equation}\label{eq:Bkstar}
 \sum \NRM{\prod_{s= 1}^{k} a_{i_s} }^2  \leq c \rho^{2k}  \AND  \sum \NRM{\prod_{s= 1}^{k_0} a_{i_s} }^2  \leq  \rho^{2k_0}.
\end{equation}
Since the sum of the length of all $[\pi] \in E'_\gamma$ is $e$, we get from \eqref{eq:boundag2},
$$
\sum_{\gamma':  \gamma' \sim \gamma} \|a (\gamma' ) \| \leq c ^{m+ \chi + e_1} \rho_0^{2(\ell m - e)} n^v  c^{e'} \rho^{2e}. 
$$
It remains to use again \eqref{eq:eprime}, $ e = v  + \chi - 1$ and adjust the constant $c$. Since 
$\tr |a(\gamma) | \leq d \| a(\gamma) \|$, we obtain \eqref{eq:sumpath}. 

This concludes the proof of \eqref{eq:sumpath} under the extra symmetry condition \eqref{eq:sym}. When this assumption does not hold, the bound \eqref{eq:pjti} is not accurate for the edges $[\pi] \in E'_\gamma$ with $m(\pi) \geq 2$ which are visited in both directions. In which case, in \eqref{eq:pjti}, the factor $\NRM{\prod_{s= 1}^{k} a_{i_s} }^2$ is replaced by 
$$
\NRM{\prod_{s= 1}^{k} a_{i_s} } \NRM{\prod_{s= 1}^{k} a_{i_{k-s+1}} }. 
$$
For those edges $[\pi]$ when we sum over all non-backtracking sequences $(i_1,\ldots,i_k)$ in \eqref{eq:boundag2}, from Cauchy-Schwarz inequality we get:
$$
\sum \NRM{\prod_{s= 1}^{k} a_{i_s} }\NRM{\prod_{s= 1}^{k} a_{i_{k-s+1}} } \leq \sqrt{ \sum \NRM{\prod_{s= 1}^{k} a_{i_s} }^2} \sqrt{\sum \NRM{\prod_{s= 1}^{k} a_{i_{k+1-s}} }^2} \leq c \rho^{2k},
$$
where we have used \eqref{eq:Bkstar} and the fact the $(i_1,\ldots,i_k)$ is non-backtracking is equivalent to $(i_k,\ldots,i_1)$ non-backtracking. The remainder of the argument leading to  \eqref{eq:sumpath} is identical.

Finally, the claimed lower bound, $\rho \geq \rho_0$ 
is a direct consequence of the right hand side of \eqref{eq:Bkstar} by considering only the non-backtracking sequence in the
 sum which maximizes $ \NRM{\prod_{s= 1}^{k_0} a_{i_s} }$. \end{proof}

Our final lemma gives a bound on $w(\gamma)$ defined below \eqref{eq:trBl2}. In the sequel, for an integer $n\in\mathbb{Z}$, we set $n_+$ to be its positive part, i.e. $n_+=\max (0,n)$.

\begin{lemma}\label{le:isopath}
There exists a constant $c >0$ such that for any $\gamma\in \cW_{\ell,m}(v,e)$ and $2 \ell m \leq \sqrt{  n}$, 
$$
 | w (\gamma) | \leq  c^{m + \chi}  \PAR{ \frac{1} {n} }^e  \PAR{\frac{ 6\ell m }{ \sqrt n} } ^{  (e_1 - 4 \chi - 4  m)_+}.  
$$
with $\chi = e - v +1$ and $e_1$ is the number of edges of $E_\gamma$ with multiplicity one. Moreover, 
$$
e_1 \geq 2 ( e - \ell m ). 
$$
\end{lemma}

\begin{proof}
We start by the last statement. Let $e_{\geq 2}$ be the number of edges of $E_\gamma$  of multiplicity at least $2$. From  \eqref{eq:propet}, we have
$$
e_1 + e_{\geq 2}  = e \quad  \AND  e_1 + 2 e_{\geq 2} \leq 2\ell m. 
$$ 
Therefore, $e_1 \geq 2( e - \ell m)$ as claimed. Let $b$ the number of inconsistent edges (recall the definition above 
Proposition \ref{prop:exppath}).  Using the terminology of the proof of Lemma \ref{le:enumpath}, a new inconsistent edge 
can appear at the the start of a sequence of first times, at a first visit of an excess edge or at the merging time.  Every such 
step can create $2$ inconsistent edges. Since each non-empty sequence of first times is followed either by a merging time or 
by a first visit of an excess edge, we deduce from \eqref{eq:defchi} that $b \leq 4 \chi + 4 m$. So finally, the number of 
consistent edges of $\gamma$ of multiplicity one is at least $(e_1 - 4 \chi - 4 m)_+$. It remains to apply  
Proposition \ref{prop:exppath}. 
\end{proof}

All ingredients are in order to prove Proposition \ref{prop:normB}.

\begin{proof}[Proof of Proposition \ref{prop:normB}]
For $n \geq 3$, we define 
\begin{equation}\label{eq:choicem}
m = \left\lfloor  \frac{ \log n }{13 \log (\log n)} \right\rfloor.
\end{equation}
For this choice of $m$, $n ^{  1 / (2m) } = o ( \log n )^{7}$ and $\ell m = o ( \log n) ^2$. Set $\rho = \rho ( B_\star) + \veps/2$, 
it suffices to prove that 
\begin{equation}\label{eq:boundS}
S = \sum_{\gamma \in W_{\ell,m} }   | w (\gamma) |  \tr  |a (\gamma) | \leq n  (c \ell m )^{10 m} \rho ^{\ell m}.
\end{equation} Indeed, Proposition \ref{prop:normB} follows immediately from \eqref{eq:trBl2}-\eqref{eq:boundS} and Markov 
inequality. Recall that $G_\gamma$ is connected for any $\gamma \in W_{\ell,m}$.  Hence, 
$|E_\gamma | \geq |V_\gamma| -1$ and
\begin{eqnarray*}
S & \leq & \sum_{v = 1}^\infty \sum_{ e  = v - 1} ^{\infty} |\cW_{\ell,m} (v,e) | \max_{ \gamma \in \cW_{\ell,m} (v,e)} \PAR{ | w(\gamma) | \sum_{\gamma' \sim \gamma}\tr  |a(\gamma')| }.
 \end{eqnarray*}
Let $\gamma \in \cW_{\ell,m} (v,e)$ with $e_1$ edges of multiplicity one and $\chi = e - v +1$. 
We use Lemma \ref{le:sumpath} with $\veps' = \veps /2$ and Lemma \ref{le:isopath}. 
Since $a \leq b + (a -b)_+$ and $e_1 \geq 2 ( e - \ell m ) $ (by Lemma \ref{le:isopath}), we find, 
\begin{eqnarray*}
| w(\gamma) | \sum_{\gamma' \sim \gamma} \tr |a(\gamma')|  &\leq  & n^v  c^{m + \chi +e_1} \rho_0^{2 (\ell m  - v)} \rho^{2v}  \PAR{ \frac{1} {n} }^v \PAR{\frac{ 8\ell m }{ \sqrt n} } ^{  (e_1 - 4 \chi - 4  m)_+}  \\
&  \leq &n  c^{5(m + \chi)}\rho_0^{2(\ell m  - v)} \rho^{2v}  \PAR{ \frac{1} {n} }^\chi \PAR{\frac{  8 c \ell m }{ \sqrt n} } ^{  (2  (v- \ell m -1)  - 2 \chi -  4m )_+}.
\end{eqnarray*}
We set $\alpha = ( 8 c \ell m )^2 /  n$ and $\ell' = \ell+2$. Since $\rho \geq \rho_0$ (if $k_0$ is chosen large enough in 
Lemma \ref{le:sumpath}), we deduce from Lemma \ref{le:enumpath} that, for some new constant $c >0$, 
\begin{eqnarray*}
S & \leq &  \sum_{v=1}^{\infty} \sum_{\chi =  0} ^{\infty}   n ( c \ell  m )^{6 m \chi  + 10 m }   \rho^{2\ell m}  \PAR{ \frac{1} {n} }^\chi  \alpha  ^{  (  v- \ell' m  -1 -   \chi)_+}. \nonumber\\
&  =  &  S_1 + S_2 + S_3, \label{eq:Sboundddd}
 \end{eqnarray*}
where $S_1$ is the sum over $\{ 1 \leq  v \leq \ell 'm, \chi \geq 0 \}$, $S_2$ over $\{ v  > \ell' m , 0  \leq  \chi <    v - \ell' m   \}$, 
and $S_3$ over $\{ v  >  \ell' m ,  \chi  \geq   v - \ell' m\}$. We find, 
\begin{eqnarray*}
S_1 & =  &    n (c \ell  m)^{10 m}  \rho^{2\ell m}  \sum_{v=1}^{ \ell' m}   \sum_{\chi = 0} ^{ \infty } \PAR{\frac{  (c \ell m) ^{6m} }{ n } }^{\chi} \\
& \leq & n (c \ell  m )^{10 m}(\ell' m)  \rho^{ 2  \ell m}  \sum_{\chi = 0} ^{ \infty } \PAR{\frac{  (c \ell m) ^{6m} }{ n } }^{\chi}.
\end{eqnarray*}
For our choice of $m$ in \eqref{eq:choicem}, for $n$ large enough, 
$$
\frac{   (c \ell m) ^{6m} }{ n } \leq \frac{( \log n )^{12m} }{n} \leq n^{-1/13}. 
$$
In particular, the above geometric series converges  and, adjusting the value of $c$,  the right hand side of 
\eqref{eq:boundS} is an upper bound for $S_1$ (since $\ell' m \leq c^m$ for $c >1$ and $n$ large enough). 
Similarly, since $\alpha =  ( 8 \ell m )^2 /  n $, for $n$ large enough, 
\begin{eqnarray*}
S_2 & = &   n  (c \ell  m )^{10 m}  \rho^{2\ell m}  \sum_{v= \ell' m+1}^{ \infty }    \alpha^{v -\ell' m-1} \sum_{\chi =0} ^{ v - \ell'  m - 1  } \PAR{\frac{  (c \ell m) ^{6m} }{ \alpha n } }^{ \chi}  \\
& \leq & n  (c \ell  m )^{10 m}  \rho^{2\ell m}  \sum_{v = \ell' m+1}^{ \infty }   \alpha^{v -\ell' m -1 } \,    2 \,   \PAR{\frac{  (c \ell m) ^{6m} }{ \alpha  n } }^{ v  - \ell'  m -1}  \\
& = &2  n (c \ell  m )^{10 m} \rho^{2 \ell m}  \sum_{t = 0}^{ \infty}    \PAR{ \frac {(c \ell m) ^{6m} }{ n } }^{ t}.
\end{eqnarray*}
Again, the geometric series is convergent and the right hand side of \eqref{eq:boundS} is an upper bound for $S_2$. 
Finally, the same manipulation gives for $n$ large enough, 
\begin{eqnarray*}
S_3 & = &   n (c \ell  m)^{10 m}  \rho^{2\ell m} \sum_{v = \ell' m+1}^{ \infty }  \sum_{\chi = v - \ell'  m  } ^{ \infty } \PAR{\frac{  (c \ell m) ^{6m} }{ n } }^{\chi}\\
& \leq & n (c \ell  m)^{10 m}  \rho^{2\ell m} \sum_{v = \ell' m+1}^{ \infty }   2  \PAR{\frac{  (c \ell m) ^{6m} }{ n } }^{v - \ell'  m }  \\
& = & 2 n (c \ell  m )^{10 m}  \rho^{2 \ell m }   \sum_{t = 1}^{ \infty}     \PAR{  \frac {(c \ell m) ^{6m} }{ n } }^{ t}.
\end{eqnarray*}
The right hand side of \eqref{eq:boundS} is again an upper bound for $S_3$. It concludes the proof.  \end{proof}

\subsubsection{Norm of $R_k^{(\ell)}$}

Here, we give a rough bound on the operator norm of the matrices $R_k^{(\ell)}$ for symmetric random permutations. 
In this subsection, we fix a collection $(a_i), i \in [d],$ of matrices such that $\max_i ( \| a_i \| ) \leq \veps^{-1}$ for some $\veps >0$. The constants may depend implicitly on $r$, $d$ 
and $\veps$.

\begin{proposition} \label{prop:normR} 
For any $1 \leq k, \ell \leq  \log n$, the event
$$ \| R_k^{(\ell)} \| \leq \PAR{\log n}^{40} \rho_1^\ell,$$
holds with the probability at least $1 -  c e^{-\frac{\ell \log n}{ c \log \log n}} $ where $c >0$ and $\rho_1 >0$  depend on 
$r$, $d$ and $\veps$.
\end{proposition}

The proof relies again on the method of moments. Proposition \ref{prop:normR} will be faster to prove than Proposition 
\ref{prop:normB} since we do not need a sharp estimate of $\rho_1$. Let $m$ be a positive integer. We argue as in 
\eqref{eq:trBl}-\eqref{eq:trBl2}, with the convention that $f_{2m + 1} = f_1$, we get 
\begin{eqnarray*}
\dE \| R_k^{(\ell)}  \| ^{2 m}  & \leq & \dE \tr \BRA{ \PAR{  R_k^{(\ell)}{R_k^{(\ell)}}^*}^{m}  } \nonumber\\
& = & \sum_{(f_1, \ldots, f_{2m})\in E^{2m}} \dE \tr  \prod_{j=1}^{m}  (R_k^{(\ell)}) _{f_{2j-1} , f_{2 j}} ( {R_k^{(\ell)}}^* )_{f_{2j} , f_{2j+1}} \nonumber \\
& \leq  &  \sum_{\gamma \in \widehat W_{\ell,m} } |\widehat w(\gamma)| \,   \tr | a ( \gamma ) |  ,  
\end{eqnarray*}
where $a(\gamma)$ is as in \eqref{eq:trBl2}, $\widehat W_{\ell,m}$ is the set of  $\gamma = ( \gamma_1, \ldots, \gamma_{2m})$ such that for any $1 \leq j \leq m$, $\gamma_j = (\gamma_{j,1} , \ldots, \gamma_{j,\ell+1}) \in F^{\ell+1}_k \backslash F^{\ell+1}$, $\gamma_{j,t} = ( x_{j,t}, i_{j,t})$, $\gamma$ with the boundary condition \eqref{eq:bound}, and we have set
$$
\widehat w(\gamma) = \dE \prod_{t=1}^{k-1} (\underline S_{i_{j,t}} )_{x_{j,t}  x_{j,t+1}}  \prod_{t=k+1}^{\ell} ( S_{i_{j,t}} )_{x_{j,t}  x_{j,t+1}}.
$$
Using that $\max_i \| a_i \| \leq \veps^{-1}$, we have $ \tr | a ( \gamma ) |\leq r \veps^{-\ell m}$ and thus, 
\begin{eqnarray}
\dE \| R_k^{(\ell)}  \| ^{2 m} 
& \leq  & c^{2\ell m} \sum_{\gamma \in \widehat W_{\ell,m} } |\widehat w(\gamma)|.    \label{eq:trRl}
\end{eqnarray}

To evaluate \eqref{eq:trRl}, we associate to each $\gamma \in \widehat W_{\ell,m}$, the graph $\widehat G_\gamma$ of
visited vertices and colored edges which appear in the expression $\widehat{w}(\gamma)$. More precisely, for each $j$, we 
set $\gamma'_j = (\gamma_{j,1}, \ldots, \gamma_{j,k}) \in F^{k}$ and 
$\gamma''_j = (\gamma_{j,k+1}, \ldots, \gamma_{j,\ell+1})\in F^{\ell+1 -k}$. Then, with the notation in Definition \ref{def1}, 
the vertex set of $\widehat G_\gamma$  is $V_\gamma  =\bigcup_j  V_{\gamma'_j} \cup V_{\gamma''_j}$ and the edge set is 
$\widehat E_\gamma  =\bigcup_j   E_{\gamma'_j} \cup E_{\gamma''_j}$ (for example, the black edge in 
Figure \ref{fig:Gamma3} is not part of $\widehat E_\gamma$).  The graph $\widehat G_\gamma$ may not be connected, 
however, due to the constraint on $\gamma$, it cannot have more vertices than edges. More precisely, let 
$\widehat G_{\gamma_j}$ denote the colored graph with vertex and edge sets $V_{\gamma'_j} \cup V_{\gamma''_j}$ 
and $E_{\gamma'_j} \cup E_{\gamma''_j}$, by the assumption that $\gamma_j \in F^{\ell+1}_k \backslash F^{\ell+1}$, it follows that either 
$\widehat G_{\gamma_j}$ is a connected graph with a cycle or it has two connected components which both contain a cycle. 
Notably, since $\widehat G_{\gamma}$ is the union of these graphs, any connected component of $\widehat G_{\gamma}$ 
has a cycle, it implies that 
\begin{equation}\label{eq:vehat}
|V_\gamma | \leq |\widehat E_\gamma |.
\end{equation}

Recall the definition of a canonical path above Lemma \ref{le:enumpath}. The following lemma bound the number of 
canonical paths in $\widehat W_{\ell,m}$.

\begin{lemma}\label{le:enumpathR}
Let $\widehat \cW_{\ell,m} (v,e) $ be the subset of canonical paths in $\widehat W_{\ell,m}$ with $|V_\gamma| = v$ and 
$|\widehat E_\gamma |= e$. We have 
$$
| \widehat \cW _{\ell,m} (v,e) | \leq   (2 d \ell  m )^{12 m \chi  + 20 m },
$$
with $\chi = e - v +1 \geq 1$. 
\end{lemma}

\begin{proof}
The proof is identical to the proof of Lemma \ref{le:enumpath} up to the minor modification that for each $j$, $\gamma'_j$ 
and $\gamma''_j$ are tangle-free and non-backtracking (instead of  simply $\gamma_j$). We use notation of Lemma 
\ref{le:enumpath}, for $1 \leq j \leq 2m$, $1\leq t \leq \ell$, $t \ne k$, we denote by $e_{j,t} = ( x_{j,t},i_{j,t}, x_{j,t+1} )$ and 
$[e_{j,t}]=   [ x_{j,t},i_{j,t}, x_{j,t+1} ] \in \widehat E_\gamma$ the visited edges. The graph $G_{(j,t)}$ is the graph spanned by
 the edges $ \{ [e_{j',t'} ]: (j',t')\preceq (j,t), t' \ne k\}$ and $T_{(j,t)}$ is its spanning forest. For each $j$, we set 
 $G_{(j,k)} = G_{(j,k-1)}$. The graphs $G_{(j,t)}$ are non-decreasing over time and by definition 
 $G_{(2m,\ell)} = \widehat G_\gamma$. 

Now, for each $j$ and $\gamma'_j$, $\gamma''_j$, the {\em merging times}, denoted by $(j,\sigma')$ and $(j,\sigma'')$, 
are the times  such that $\gamma'_j$ and $\gamma''_j$ merge into a previous connected component. More precisely,  
if $(j,t)$ with $1 \leq t \leq k-1$ (resp. $k+1 \leq t \leq \ell)$ is the smallest time such that $x_{j,t+1}$ is a vertex of 
$G_{(j,1)^-}$ (resp. $G_{(j,k+1)^-}$) then $\sigma' = t$ (resp. $\sigma'' = t$). By convention, if $x_{j,1} \in G_{(j,1)^-}$, 
we set $\sigma'  =0$ (for example from \eqref{eq:bound} if $j$ is odd, $\sigma' = 0$), and we set  
$\sigma'' = k$ if $x_{j,k+1} \in G_{(j,k)^-}$. Similarly, we set $\sigma' = k$ (resp. $\sigma'' = \ell+1$) if $\gamma'_j$ 
does not interest $G_{(j,1)^-}$ (resp. $\gamma''_j$ does not intersect $G_{(j,k)^-}$). {\em First times} 
and {\em important times} are defined as in Lemma \ref{le:enumpath}. 

We mark important times $(j,t)$ by the vector $(i_{j,t},x_{j,t+1},x_{j,\tau})$, where $(j,\tau)$ is the next time that $[e_{j,\tau}]$ 
will not be a tree edge of the forest $T_{j,t}$ constructed so far (by convention, if the path $\gamma'_j$ or $\gamma''_j$ 
remains on the forest, we set $\tau = k$ or $\tau = \ell+1$). For $t=1$ and  $t = k+1$, we also add the {\em starting mark} 
$(x_{j,1},\sigma', x_{j,\tau})$ and $(x_{j,k+1},\sigma'', x_{j,\tau})$) where $\sigma'$ and $\sigma''$ are the merging times 
and $(j,\tau) \geq (j,\sigma')$ or $(j,\tau) \geq (j,\sigma'')$ is, as above, the next time that $[e_{j,\tau}]$ will not be a tree edge 
of the forest constructed so far. As in Lemma \ref{le:enumpath}, it gives rise to a first encoding $\widehat \cW_{\ell,m} (v,e)$.

It can be improved by using that $\gamma'_j$ and $\gamma''_j$ are tangle-free. For each $1 \leq j \leq 2m$ and both for 
$\gamma'_j$ and $\gamma''_j$, we define {\em short cycling}, {\em long cycling} and {\em superfluous} times as in Lemma 
\ref{le:enumpath} and we modify the mark of the short cycling time $(j,t_s)$ as 
$(i_{j,t_s},x_{j,t_s+1}, x_{j,t_1}, t_2 ,x_{j, \tau})$, where $(j,t_1)$ is the closing time of the cycle, $(j, t_2)$ is the exit time of 
the cycle and $(j,\tau) \succeq (j, t_2)$, is the next time that $[e_{j,\tau}]$ will not be a tree edge of the forest constructed so far. 
For $\gamma'_j$ (resp. $\gamma''_j$), important times $(j,t)$ with $1 \leq t < t_s$ or $\tau \leq t \leq k-1$ 
(resp. $k+1 \leq t < t_s$ or $\tau \leq t \leq \ell$) are called long cycling times, they receive the usual mark 
$(i_{j,t},x_{j,t+1},x_{j,\tau})$. The other important times are called superfluous. By convention, if there is no short cycling time, 
we call anyway, the last important time, the short cycling time. As argued in Lemma \ref{le:enumpath}, there are at most 
$\chi-1$ long cycling time for $\gamma'_j$ and $\gamma''_j$. 

This is the second encoding: we can reconstruct uniquely $\gamma$ from the starting marks, the positions of the long cycling 
and the short cycling times and their marks. For each $j$, there are $2$ starting marks and at most $2$ short cycling times 
and $ 2(\chi-1)$ long cycling times. There are at most $  (\ell +1)^{4m \chi}$ ways to position them. There are at most 
$d v^2$ different possible marks for a long cycling time and $d v^3 (\ell+1)$ possible marks for a short cycling time. 
Finally, there are $v^2(\ell+1) $ possibilities for a starting mark. We deduce that    
$$
| \cW _{\ell,m} (v,e) | \leq    (\ell +1)^{4 m \chi} (v^2 (\ell+1) )^{4m}   (d v^2 ) ^{4m (\chi-1)}(d v^3 (\ell+1)) ^{4m}. 
$$
Using $v \leq 2 \ell m$ and $\ell +1 \leq 2 \ell$, we obtain the claimed bound. \end{proof}

We are ready to prove Proposition \ref{prop:normR}.

\begin{proof}[Proof of Proposition \ref{prop:normR}]
For $n \geq 3$, we define 
\begin{equation}\label{eq:choicemR}
m = \left\lfloor  \frac{ \log n }{25 \log (\log n)} \right\rfloor.
\end{equation}
For this choice of $m$, $\ell m = o ( \log n) ^2$. From Markov inequality and \eqref{eq:trRl}, it suffices to prove that for some constants $c, c_1 >0$,
\begin{equation}\label{eq:boundSR}
S = \sum_{\gamma \in \widehat W_{\ell,m} }   | \widehat w (\gamma) |  \leq  (c \ell m )^{32 m} c_1 ^{2 \ell m},
\end{equation}

From \eqref{eq:vehat},  $|V_\gamma| \leq |\widehat E_\gamma | \leq 2\ell m$ and
\begin{eqnarray*}
S & \leq & \sum_{v = 1}^{2 \ell m} \sum_{ e  = v } ^{\infty} |\widehat \cW_{\ell,m} (v,e) | \max_{ \gamma \in \cW_{\ell,m} (v,e)} \PAR{ | \widehat w(\gamma) | N(\gamma) },
 \end{eqnarray*}
where $N(\gamma)$ is the number of $\gamma'$ in $\widehat W_{\ell,m}$ such that $\gamma'\sim \gamma$. 
If $\gamma \in \cW_{\ell,m} (v,e)$, the following trivial bound holds:
$$
N(\gamma) \leq n^v d^e,
$$
(indeed, $n^v$ bounds the possible choices for the vertices in $V_\gamma$ and $d^e$ the possible choices for the colors 
of the edges in $\widehat E_\gamma$). Moreover, from Proposition  \ref{prop:exppath} (bounding the number of 
inconsistent edges by $e$), if $\gamma \in \cW_{\ell,m} (v,e)$, 
$$
\ABS{ \widehat w(\gamma) }  \leq c \PAR{\frac 9 n}^e.
$$
Using also Lemma \ref{le:enumpathR}, we find 
\begin{eqnarray*}
S & \leq & c \sum_{v = 1}^{2 \ell m} \sum_{ e  = v } ^{\infty}  (2 d \ell  m )^{12 m (e-v)  + 32 m } n^v d^e  \PAR{\frac 9 n}^e,
\\
& = & c  (2 d \ell  m )^{32m}  \sum_{v = 1}^{2 \ell m} \PAR{9d} ^{v}\sum_{ t = 0 } ^{\infty} \PAR{  \frac{ 9 d (2 d \ell  m )^{12 m } }{n} }^{t}.
 \end{eqnarray*}
For our choice of $m$, the geometric series is convergent. It follows that for some new constant $ c , c' > 0$, 
$$
S  \leq   c  (2 d \ell  m )^{32m}  \sum_{v = 1}^{2 \ell m} \PAR{9d} ^{v} \leq c' (2 d \ell  m )^{32m} \PAR{9d} ^{2 \ell m} 
$$
It concludes the proof of \eqref{eq:boundSR} with $c_1  = 9d$ (a finer analysis as done in Proposition \ref{prop:normB}  
leads to $c_1 =  d-1 + o(1)$).
  \end{proof}

\subsection{Proof of Theorem \ref{th:mainB}}
\label{subsec:net}

Let $0< \veps <1$. For a given collection of  weights $a = (a_i) \in M_r( \dC)^ d$, we denote by $B (a)$ the corresponding 
non-backtracking operator and by $\mathcal E_\veps(a)$ the event that $\rho( B(a)_{|K_0}) > \rho( B_\star(a)) + \veps$. 
It is sufficient to prove that for some $\beta >0$, 
\begin{equation}\label{eq:tbpB}
\dP \PAR{ \bigcup_{a \in \cS_\veps^d} \mathcal E_\veps(a) }  = O ( n^{-\beta}),
\end{equation}
where  $\mathcal B \subset M_r (\dC)$ is the unit ball for the operator norm $\| \cdot \|$ and 
$$\cS_\veps = \{ b \in  M_r (\dC): b \in \veps^{-1} \mathcal B  , b^{-1} \in   \veps^{-1} \mathcal B \}.$$ 
We  use a net argument on $\cS_\veps^d$. Due to the lack of uniform continuity of spectral radii, we perform the net 
argument with operator norms. To this end, we fix an integer valued sequence $\ell (n) \sim ( \log n ) / \kappa $ for some 
$\kappa > 1$ satisfying
$$ \kappa > \log   \PAR{  ( d-1)^4   \vee   \PAR{ \frac{4 \rho_1} {\veps}} },$$
where $\rho_1$ is as in Proposition \ref{prop:normR}. 
In order to lighten the notation, we will omit the dependence in $n$ of $\ell (n)$ whenever appropriate. 
We denote by  $\mathcal E'_\veps(a)$ the event  
$$
\sup_{g \in K_0, \| g \|_2 = 1 }    \| B^{\ell} (a)g \|_2 > \PAR{ \rho( B_\star(a)) + \veps }^{\ell}.
$$
From \eqref{eq:basicrho}, the inclusion $\mathcal E_\veps (a) \subset \mathcal E'_\veps(a)$ holds. Moreover, 
from \eqref{eq:defB}, 
$ \| B(a ) \| \leq (d-1) \| a \|
$, 
where 
$$
\| a \| = \sum_{i =1}^d \| a_i   \|. 
$$
We deduce that the map $a \mapsto B^\ell (a)$ satisfies a deviation inequality 
\begin{eqnarray}
\| B^{\ell} (a ) - B^{\ell} (a') \| & \leq &   \ell  \max( \| B (a) \| , \|B (a') \|  )^{\ell-1} \| B (a - a') \| \nonumber \\
&\leq  &\ell (d-1)^\ell \max( \| a \| , \| a' \| ) ^{\ell-1} \| a - a' \|. \label{eq:LIPB}
\end{eqnarray}

We now build our net of $\cS_\veps^d$. First, since all matrix norms are equivalent and $M_r( \dC) \simeq \dR^{2r^2}$, we 
can find a subset $N _\delta\subset \veps^{-1} \mathcal  B $ of cardinality at most $( c / (\veps \delta) )^{2r^2}$ such that for 
any $b \in \veps^{-1} \mathcal B$, there exists $b_0 \in N_\delta$ with $\| b - b_0 \| \leq \delta$ (the constant $c$ depends 
on $r$). Note that
$$
\| b_0 ^{-1} \| \leq \| b^{-1} \| + \| b_0 ^{-1}  - b^{-1} \| \leq  \| b^{-1} \|  + \| b^{-1} \|  \|b_0^{-1} \| \| b_0  - b \|.
$$
Hence, if $\| b^{-1} \|  \leq \veps^{-1}$ and $\delta < \veps / 2$, we find 
$$
\| b_0 ^{-1} \| \leq \frac{ \| b^{-1} \|   } { 1 -  \delta / \veps }  \leq  \frac 2 \veps. 
$$
We deduce that, if  $\delta < \veps / 2$, there exists $N'_\delta \subset \cS_{\veps/2} \cap N_\delta$ such for any 
$b \in \cS_\veps$, there exists $b_0 \in N'_\delta$ with $\| b - b_0 \| \leq \delta$. Consequently, if $\delta <  \veps d / 2 $, 
there exists a subset  $N''_\delta = (N'_{\delta / d} ) ^d \subset \cS_{\veps/2}^d$ of cardinal number at most 
$(c d / \veps \delta)^{2 r ^2 d}$  such that  for any $a \in \cS_\veps^d$, there exists $a_0 \in N''_\delta$ with 
$\| a - a_0 \| \leq \delta$.  Besides, from Lemma \ref{le:contrho}, for all $\delta$ small enough, 
\begin{equation}\label{eq:contrhoB}
\ABS{ \rho(B_\star(a)) - \rho(B_\star (a_0)) } \leq \frac \veps  3,
\end{equation}
and, from \eqref{eq:LIPB}, for some new constant $c >0$,
$$
\| B^{\ell} (a ) - B^{\ell} (a_0) \| \leq  \ell (d-1) \PAR{\frac{ d-1}{ \veps}}^{\ell -1}  \delta \leq c^\ell \delta, 
$$
If $\delta = (\veps / 3 c)^\ell$ and $\mathcal E'_{\veps / 3} (a_0) $ does not hold, we deduce, for $n$ large enough,
\begin{eqnarray*}
\sup_{g \in K_0, \| g \|_2 = 1 }  \| B^{\ell} (a)g \|_2 &  \leq & \sup_{g \in K_0, \| g \|_2 = 1 }   \| B^{\ell} (a_0)g \|_2 +  \| B^{\ell} (a ) - B^{\ell} (a_0) \| \\
& \leq &\PAR{ \rho(B_\star (a_0)) + \frac \veps 3 }^\ell +  \PAR{\frac \veps 3 }^\ell \\
& \leq & \PAR{ \rho(B_\star (a_0)) +\frac {2\veps} 3 }^\ell.
\end{eqnarray*}
Using \eqref{eq:contrhoB}, we find that, for our choice of $\delta$ and $n$ large enough,
$$
\bigcup_{a \in \cS_\veps^d} \mathcal E_\veps (a) \subset \bigcup_{a \in \cS_\veps^d} \mathcal E'_\veps (a) \subset  \bigcup_{a \in N''_\delta} \mathcal E' _{\frac \veps 3} (a) ,
$$
and, for some $c_1 >0$ (depending on $\veps$, $r$ and $d$),
\begin{equation}\label{eq:epsnet}
| N''_\delta | \leq c_1 ^\ell.
\end{equation}
We may now use the union bound to obtain an estimate of \eqref{eq:tbpB}. If $\Omega_0$ is the event that $G^\sigma$ is 
$\ell$-tangle free, we find, for $n$ large enough,
\begin{eqnarray*}
\dP \PAR{ \bigcup_{a \in \cS^d_\veps} \mathcal E_\veps (a) } & \leq & \sum_{a \in N''_\delta}  \dP\PAR{  \mathcal E' _{\frac \veps 3} (a) \cap  \Omega_0} + \dP\PAR{ \Omega_0^c}   \\
& \leq & \sum_{a \in N''_\delta}  \dP \PAR{ J (a)  \geq   \PAR{ \rho (B_\star(a) ) + \frac \veps 3  }^\ell } + O\PAR{ \frac{\ell^3 (d-1)^{4\ell} }{ n} } , 
\end{eqnarray*}
where at the second line, we have used Lemma \ref{le:decompBl}, Lemma \ref{le:tanglefree} and set  
$J(a) = \|  \uB^{(\ell)}(a) \|  + \frac 1{ n}   \sum_{k = 1}^\ell \| R_{k}^{(\ell)} (a)\|.$
For our choice of $\ell$, we note that  $\ell^3 (d-1)^{4\ell} / n = O(n^{-\beta})$ for some $\beta >0$. 
On the other end, by Propositions \ref{prop:normB}-\ref{prop:normR} applied to $\veps' = \veps / 4$, for any 
$a \in \cS_{\veps'}$, with probability at least $1 - c \ell \exp  (  -  \ell \log n   / (c \log \log n))$, we have 
\begin{align*}
J  (a) \leq ( \log n )^{20} \PAR{ \rho (B_\star(a) ) + \frac \veps 4 } ^ \ell +  \frac {1 } { n}   \sum_{k = 1}^\ell  (\log n)^{40} \rho_1^{\ell} \leq ( \log n )^{c}   \PAR{ \rho (B_\star(a) ) + \frac \veps 4 } ^ \ell,
\end{align*} since $\ell = O (\log n)$ and (for $n$ large enough) $\rho_1^\ell \leq n (\veps/4)^{\ell}$   thanks to our choice of $\ell$. Finally, since  $(\log n)^{c / \ell} = 1 + O ( \log \log n / \log n )$, it follows that the event $\{ J (a) \geq  \PAR{ \rho (B_\star(a) ) + \frac \veps 3  }^\ell\}$ holds with probability most $c \ell \exp  (   - \ell  \log n    / (c \log \log n))$. Using \eqref{eq:epsnet}, we obtain, 
$$
\dP \PAR{ \bigcup_{a \in \cS^d} \mathcal E_\veps (a) }  =  O \PAR{   \ell e^{ - \frac{\ell  \log n   }{  c \log \log n} } c_1  ^\ell + n^{-\beta}  } = O \PAR{ n^{-\beta} } .
$$
The bound \eqref{eq:tbpB} follows.  \qed

\section{Proof of Theorem \ref{th:mainT}}
\label{sec:mainT}

We start with the inclusion \eqref{eq:lowboundA} with $A^{(2)}$ in place of $A$. Note that $\ell^2 (X^2)$ can be decomposed 
as the direct sum  $\ell^2 (X^2) = \ell^2 (X^2_=) \oplus  \ell^2 (X^2_{\ne})$  where $X^2_=  = \{ (x,x) : x \in X \}$ and 
$X^2_{\ne} = \{ (x,y) : x \ne y  \in X \}$. Moreover,  $A_{|\ell^{2} (X^2_{=})}$ can be identified with $A$. It follows that the 
spectrum of $A^{(2)}$ contains the spectrum of $A$, and thus \eqref{eq:lowboundA} holds also for $A^{(2)}$ thanks to 
Section \ref{sec:AB}. 

We turn to the inclusion \eqref{eq:upboundA} with $A^{(2)}$ in place of $A$. Recall that the vector space $V$ is spanned by 
$I$ and $J$ defined by \eqref{eq:defIJ}. We set $$K^{(2)}_0 = \dC^r  \otimes V^\perp \otimes\dC^d.$$ 

Arguing exactly as in the proof of Theorem \ref{th:main}, Theorem \ref{th:mainT} is a consequence of the following statement 
on the non-backtracking operators
$$
B = \sum_{j\ne i^*} a_j \otimes S_i \otimes S_i  \otimes E_{ij}.
$$

\begin{theorem}\label{th:mainBT}
Theorem \ref{th:mainB} holds with $A$ replaced by $A^{(2)}$ defined by \eqref{eq:defAT} and $K_0$ replaced by $K_0^{(2)}$.
\end{theorem}

The proof of Theorem \ref{th:mainBT} follows essentially from  the proof of Theorem \ref{th:mainB}. We now explain how 
to adapt the above argument. We shall follow the same steps and only highlight the differences.

The first observation is that part of Theorem \ref{th:mainBT} is already contained in Theorem \ref{th:mainB}. More precisely, 
as already pointed, we have the direct sum  $\ell^2 (X^2) = \ell^2 (X^2_=) \oplus  \ell^2 (X^2_{\ne})$.  Then, it is immediate 
to check that, for any  permutation operator $S$  on $\ell^2 (X)$, $S \otimes S$ decomposes orthogonally on 
$\ell^2 (X^2_=) \oplus  \ell^2 (X^2_{\ne})$. Hence, $B$ decomposes orthogonally on  
$(\dC^r  \otimes \ell^2 (X^2_=) \otimes\dC^d) \oplus (\dC^r  \otimes \ell^2 (X^2_{\ne}) \otimes\dC^d)$. Also, since 
$J \in \ell^2 (X^2_{\ne})$ and $I \in \ell^2 (X^2_=)$, the operator $B_{|K^{(2)}_0}$ decomposes orthogonally on  
$K^{I}_0 \oplus K^{J}_0$ with 
$$
K^I_0  =  \dC^r  \otimes  ( I^\perp \cap \ell^2 (X^2_=) )  \otimes\dC^d \AND  K^J_0  =  \dC^r  \otimes (J^\perp \cap \ell^2 (X^2_{\ne})) \otimes\dC^d .
$$
Finally, $B_{|K^{I}_0}$ can be identified with $B'_{|K_0}$ where $B' = \sum_{j\ne i^*} a_j \otimes S_i \otimes E_{ij}$. 
The spectral radius of  $B'_{|K_0}$ can be bounded using Theorem \ref{th:mainB}.  As a byproduct, it remains to prove 
Theorem \ref{th:mainBT} with $K_0^{(2)}$ replaced by $K^J_0$.

\subsection{Path decomposition}
\label{subsec:PDT}

We follow Subsection \ref{subsec:PD} and use the same notation. We set $X = [n]$ and let $A^{(2)}$ be as in 
\eqref{eq:defAT}. We now denote by $\bm B$ the non-backtracking matrix of $A^{(2)}$ restricted to 
$\dC^r  \otimes \ell^2 (X^2_{\ne}) \otimes\dC^d$. Our goal is to derive the analog of Lemma  
\ref{le:decompBl}
for 
$\rho (\bm B_{|K^{J}_0})$. We define $\bm E = X^2_{\ne} \times [d]$. We may write $\bm B$  as a matrix-valued matrix on 
$\bm E$: for $e,f \in \bm E$,  $e = (x,i)$, $f = (y,j)$, $x = (x^- , x^+)$, $y = (y^-, y^+)$, 
$$
\bm B_{e f}  = a_j \IND( \sigma_i(x^-) =  y^-)\IND( \sigma_i(x^+) =  y^+)  \IND ( j \ne i^*) = a_j (S_i \otimes S_i)_{xy}  \IND ( j \ne i^*).
$$

The next definitions extend Definitions \ref{def1}-\ref{def2}.  
We revisit Definition \ref{def0}, where we replace $X$ by  $ X^2_{\ne}$, and $\sigma_i$ by $\sigma_i\otimes \sigma_i , i\in [d]$.
We may define
a colored edge $[x,i,y]$ with $x,y \in X^2_{\ne}$, $i \in [d]$.
\begin{definition}\label{defT}
Let  $\gamma = (\gamma_1, \ldots, \gamma_k)$ in $\bm E^k$, $\gamma_t = (x_t, i_t)$, $x_t = (x^-_t, x^+_t )$. 
\begin{enumerate}[-]
\item 
We set $\gamma^{\pm} = (\gamma_{1}^\pm, \ldots ,  \gamma_k ^\pm)$ with $\gamma^\pm _t = (x_t^\pm, i_t) \in E$.

\item 
The  {\em weight} of $\gamma$ is  $a(\gamma) = a(\gamma^\pm) = \prod_{t=2}^k a_{i_t}.$ 

\item 
We set $\bm V_\gamma = \{x_t  : 1 \leq t \leq k \}$, $\bm E_\gamma = \{ [x_t, i_t, x_{t+1}] : 1 \leq t \leq k \} $, 
$V_\gamma = V_{\gamma^-} \cup  V_{\gamma^+} = V_{(\gamma^-,\gamma^+)}$ and 
$E_\gamma = E_{\gamma^-} \cup  E_{\gamma^+}  = E_{(\gamma^-,\gamma^+)}$. We define the colored graphs  
$G_\gamma = (V_\gamma,E_\gamma)$ and $\bm G_\gamma = (\bm V_\gamma,\bm E_\gamma)$, see Figure \ref{fig:impoT}.

\item 
If $e , f$ are in $\bm{E}$, we define $\bm\Gamma^{k} _{ e f}$ is the subset of $\gamma$ in $\bm E^k$,  such that 
$(\gamma^-, \gamma^+) \in \Gamma^{k} _{e^- f^-} \times \Gamma^{k} _{e^+ f^+}$.  The sets $\bm \Gamma^k$, 
$\bm F^k$, $\bm F_k ^{\ell+1}$ are defined in the same way from the sets $ \Gamma^k$, $F^k$, $F_k ^{\ell+1}$. For 
example,  $\bm F^k$ is the set $\gamma$ in $\bm \Gamma^k$ such that both $(\gamma^-,\gamma^+)$ are tangle-free. 
\end{enumerate}

\end{definition}

Note that according to this definition, if $\gamma\in \bm F^k$, then $\bm G_{\gamma}$ is necessarily tangle-free, while
$G_{\gamma}$, not necessarily. See Figure \ref{fig:impoT} for an example.

\begin{figure}[htb]
\begin{center}  
\resizebox{15cm}{!}{
\begin{tikzpicture}[main node/.style={circle, draw , fill = lightgray, text = black, thick}]

\node[main node]  at (2,0) (a1) {1} ;
\node[main node] at (4,0) (a2) {2} ;

\node[] at (3,1) {$G_{\gamma^-}$};
\draw[cyan, -,ultra thick][out = -40 , in = 220] (a1) to (a2) ; 
\draw[magenta, -,ultra thick][out = 40 , in = -220] (a1) to (a2) ;

\node[main node]  at (11,0) (b3) {3} ;
\node[main node] at (13,0) (b4) {4} ;
\node[main node] at (14,1.7320508 ) (b5) {5} ; 
\node[main node]  at (15,0) (b1) {1} ;

\node[] at (12.2,1.2) {$G_{\gamma^+}$};
\draw[cyan, -,ultra thick] (b3) to (b4) ; 
\draw[magenta, -,ultra thick] (b4) to (b5) ; 
\draw[cyan, -,ultra thick] (b5) to (b1) ; 
\draw[magenta, -,ultra thick]  (b1) to (b4) ;

\node[main node]  at (0,-4) (c3) {3} ;
\node[main node] at (2,-4) (c4) {4} ;
\node[main node] at (3,-2.2679492) (c5) {5} ; 
\node[main node]  at (4,-4) (c1) {1} ;
\node[main node]  at (6,-4) (c2) {2} ;

\node[] at (1.2,-2.8) {$G_{\gamma}$};
\draw[cyan, -,ultra thick] (c3) to (c4) ; 
\draw[magenta, -,ultra thick] (c4) to (c5) ; 
\draw[cyan, -,ultra thick] (c5) to (c1) ; 
\draw[magenta, -,ultra thick]  (c1) to (c4) ; 
\draw[magenta, -,ultra thick][out = 40 , in = -220] (c1) to (c2) ; 
\draw[cyan, -,ultra thick][out = -40 , in = 220]
 (c1) to (c2) ;

\node[main node]  at (11,-4) (d1) {13} ;
\node[main node] at (13,-4) (d2) {24} ;
\node[main node] at (14,-2.2679492 ) (d3) {15} ; 
\node[main node]  at (15,-4) (d4) {21} ;

\node[] at (12.2,-2.8) {$\bm G_{\gamma}$};
\draw[cyan, -,ultra thick] (d1) to (d2) ; 
\draw[magenta, -,ultra thick] (d2) to (d3) ; 
\draw[cyan, -,ultra thick] (d3) to (d4) ; 
\draw[magenta, -,ultra thick]  (d4) to (d2) ;

\end{tikzpicture}
}
{$$ \gamma =  ((1,3),1)     ((2,4),2)     ((1,5),1)    ((2,1 ),2)    ((1, 4), 1) $$
} 
\vspace{-25pt}
\caption{A path $\gamma  \in \bm  E^5$ and its associated graphs $G_{\gamma^\pm}$, $G_\gamma$ and 
$\bm G_\gamma$ (whose vertices have been written $x^-x^+$ instead of $(x^-,x^+)$), the involution 
$i^*$ is the identity.} \label{fig:impoT}

\end{center}\end{figure}

For $e,f$ in $\bm E$, we find that
$$
(\bm B^{\ell}) _{e f}=  \sum_{\gamma \in \bm\Gamma^{\ell+1} _{e  f}} a(\gamma)  \prod_{t=1}^{\ell} ( S_{i_t} \otimes S_{i_t} )_{x_{t} x_{t+1} }  .
$$

The orthogonal projection of $S_i \otimes S_i$ on $J^\perp \cap \ell^2 (X^2_{\ne})$ is given by 
$$
(\underline{S_{i } \otimes S_i})   = ( S_{i} \otimes S_i ) - P_J.
$$
where, for any $g \in \ell^2 (X^2_{\ne})$,  $P_J(g) = J \langle J , g \rangle / (n (n-1))$. Alternatively, in matrix form, for any 
$x,y$ in $X^2_{\ne}$
\begin{equation}\label{eq:defST}
(\underline{S_{i } \otimes S_i})_{xy}   = ( S_{i} \otimes S_i )_{xy} - \frac 1 {n(n-1)}.
\end{equation}

Now, recall the definition of the colored graph $G^\sigma$ 
with vertex set $[n]$ in Definition \ref{def0}. Obviously, if $G^\sigma$ is 
$\ell$-tangle-free  and $0 \leq k \leq 2\ell$ then 
\begin{equation}\label{eq:BkBkT}
\bm B^{k}   = \bm B^{(k)},
\end{equation}
where 
$$
(\bm B^{(k)}) _{e f}=  \sum_{\gamma \in \bm F^{k+1} _{e f}} a(\gamma)  \prod_{t=1}^{k} ( S_{i_t}  \otimes S_{i_t} )_{x_{t} x_{t+1} } .
$$
We define similarly the matrix $\underline B^{(k)}$ by 
\begin{equation}\label{eq:defDT}
(\underline{\bm B}^{(k)}) _{e f}=  \sum_{\gamma \in \bm F^{k+1} _{e f}} a(\gamma)  \prod_{t=1}^{k} ( \underline{ S_{i_t}  \otimes S_{i_t} })_{x_{t} x_{t+1} } .
\end{equation}

We may now use the telescopic sum decomposition performed in Subsection \ref{subsec:PD}. We denote by  
$\overline {\bm B}$ the matrix on $\dC^r \otimes \dC^{\bm E}$ defined by 
$$\overline {\bm B} = \sum_{j \ne i^*}  a_j \otimes P_J  \otimes  E_{ij}.$$  
We also set for all $e,f \in \bm E$, 
\begin{equation}\label{eq:defR}
(\bm R^{(\ell)}_{k} )_{ef}   =  \sum_{\gamma  \in \bm F^{\ell+1} _{k,e f} \backslash \bm F^{\ell+1} _{e f}} a(\gamma)  \PAR{ \prod_{t=1}^{k-1} ( \underline {S_{i_t} \otimes S_{i_t}} )_{x_t x_{t+1}}}  \PAR{\prod_{t = k+1}^\ell (S_{i_t} \otimes S_{i_t} )_{x_t x_{t+1}}  }.
\end{equation}
Arguing as in  Subsection \ref{subsec:PD}, we have  
\begin{eqnarray*}
\bm  B^{(\ell)}   &= & \underline{ \bm B}^{(\ell)}   +    \frac 1 {n(n-1)}  \sum_{k = 1} ^{\ell}  \underline{ \bm B}^{(k-1)} \overline {\bm B}  \bm B^{(\ell - k)}  -   \frac 1 {n(n-1)}     \sum_{k = 1}^{\ell} \bm R^{(\ell)}_{k}.
\end{eqnarray*}
Now, if $G^\sigma$ is $\ell$-tangle free, then, from  \eqref{eq:BkBkT}, 
$\overline {\bm  B} \bm B^{(\ell - k)} = \overline {\bm B} {\bm  B}^{\ell -k}$. 
Since the kernel of $\overline {\bm B}$ contains $K^J_0$ and $\bm B^{\ell-k} K^J_0 \subset K^J_0$, we find that $\overline {\bm B} {\bm  B}^{\ell -k}= 0$ on $K^J_0$. So finally, if $G^\sigma$ is $\ell$-tangle free, for any $g \in K^J_0$,  
\begin{eqnarray*}
\bm  B^{(\ell)} g  &= & \underline{ \bm B}^{(\ell)} g  -   \frac 1 {n(n-1)}     \sum_{k = 1}^{\ell} \bm R^{(\ell)}_{k}g.
\end{eqnarray*}
We get the following lemma. 

\begin{lemma}\label{le:decompBlT}
Let $\ell \geq 1$ be an integer and $A^{(2)}$ as in \eqref{eq:defAT} be such that $G^\sigma$ is $\ell$-tangle free. Then,
$$
\rho (\bm B_{|K^J_0}) \leq  \PAR{\|  \underline{ \bm B}^{(\ell)} \|  +  \frac 1{ n(n-1)}   \sum_{k = 1}^\ell \| \bm R^{(\ell)}_{k} \|} ^{1/\ell}.$$
\end{lemma}

\subsection{Novel estimate on random permutations}

In this subsection, we prove the analog of Proposition \ref{prop:exppath} for the tensors $S_i \otimes S_i$ for the 
symmetric random permutations. To this end, we need to adapt the proof of this proposition in \cite{bordenaveCAT}.

Consider a sequence of colored edges $(f_1, \ldots, f_\tau)$ with $f_t = [x_t,i_t,y_t]$ and $x_t= (x^-_t,x^+_t),y_t = (y^-_t,y^+_t) \in X_{\ne}$. 
We set $f^\pm _t = (x^\pm_t,i_t,y^\pm_t)$. We say an edge $ [x ,i,y] \in \{ f_t : 1 \leq t \leq \tau \}$ is {\em consistent} if  
$[x^- ,i,y^-]$  \emph{and} $[x^+ ,i,y^+]$ are consistent in the colored graph  spanned by the edge $\{ f^-_t, f^+_t : 1 \leq t \leq \tau \}$ 
(where the definition of a consistent edge is given in Definition \ref{def:consistent}). It is inconsistent otherwise. 
Its {\em multiplicity} is the pair of multiplicities 
$(m^- , m^+)$ of $[x^- ,i,y^-]$ and $[x^+ ,i,y^+]$ in the sequence $(f^+_1, \ldots, f^+_\tau,f^-_1, \ldots, f^-_\tau)$ 
(that is, $m^\pm = \sum_t \IND ( f^-_t = [x^\pm,i,y^\pm] ) + \IND ( f^+_t = [x^\pm,i,y^\pm] )$). 

\begin{proposition}
\label{prop:exppathT}
For symmetric random permutations, there exists a   constant $c>0$ such that for any sequence $(f_1, \ldots, f_\tau)$, with
$f_t  = [x_t,i_t,y_t]$, $x_t,y_t \in X_{\ne}$, $4 \tau \leq n^{1/3}$ and any $\tau_0 \leq \tau$, we have, 
$$
\ABS{ \dE \prod_{t= 1} ^{\tau_0}  (\underline{S_{i_t} \otimes S_{i_t}} )_{x_t  y_{t}}  \prod_{t= \tau_0+1} ^{\tau}  ( S_{i_t}\otimes S_{i_t} )_{x_t  y_{t}} }  \leq c \, 9 ^b  \PAR{ \frac{1 }{ n }}^{e} \PAR{ \frac{  6 \tau   }{ \sqrt{n} }}^{e_1}, 
$$ 
where $e = | \{ f^+_t, f^-_t : 1 \leq t \leq \tau \} |$, $b$ is the number of inconsistent edges of  and $e_1$ is the number of  
$1 \leq t \leq \tau_0$ such that $f_t$ is consistent and has multiplicity $(1,1)$.
\end{proposition}

We will use the Pochammer symbol, defined for non-negative integers $a,b$,
$$
(a)_b = \prod_{p=0}^{b-1} (a-p),
$$
(recall the convention that a product over an empty set is $1$). We will use the following Lemma whose proof is postponed to Section \ref{sec:aux}. 

\begin{lemma}\label{le:binT}
Let $z\geq 1$, $k\geq 1$ be an integer, $0 < p , q \leq 1/4$ and $N$ be a $\BIN ( k , p)$ variable, if  
$8 ( 1 - p - p/q)^2 \leq  4 z k ^2 \sqrt q \leq 1$, we have 
$$
\ABS{ \dE (-1)^N \prod_{n=0}^{2N-1} \PAR{ \frac 1 {\sqrt{q} } - z n } } \leq  8 ( 3  \sqrt {2z} k   q^{1/4}) ^k. 
$$
\end{lemma}

\begin{proof}[Proof of Proposition \ref{prop:exppathT}] 
We adapt the proof of \cite[Proposition 8]{bordenaveCAT}. Using the independence of the matrices $S_i$ (up to the involution), 
it is enough to consider the case of a single permutation matrix. We set $S = S_i$ and $\sigma = \sigma_i$. The colored edge 
$[x,i,y]$ of $ \{ f_t : 1 \leq t \leq \tau \}$ is simply denoted by $(x,y)$.  Note that we may view $(x,y)$ as an oridented edge 
from $x$ to $y$ as $i$ is fixed.
We treat the case of $S$ uniformly sampled random  
permutation (the case of uniform matching is similar, see final comment below). We will repeatedly use that, if 
$1 \leq k \leq a \sqrt n$ then, for some $c = c(a) >0$,
$$
(n-k)^{-k} \leq c n^{-k}.
$$

\noindent  {\em We first assume that all edges are consistent}. 
Let $X = \{x^\veps_t : t \in [\tau],  \veps \in \{ -,+\} \}$, 
$Y =\{y^\veps_t : t \in [\tau],\veps \in \{ -,+\} \}$ and $\{ g_1, \ldots , g_e  \} = \{ (x^\veps_t,y^\veps_t)  : t \in [\tau] , \veps \in \{ -,+\}\}$ 
with $g_{s} = ( u_s , v_s  )$ be the distinct edges of $\{ f^\veps_t : t \in [\tau], \veps \in \{ -,+\}\}$. 
Let $T$ be the set $1 \leq t \leq \tau_0$ such that 
$(x_t,y_t)$ has multiplicity $(1,1)$. Let $T_2 \subset T$ be the subset of $t$ in $T$ such that 
$(\sigma (x^-_t),\sigma(x^+_t) ) \in \{ (y_t^-,y_t^+) \} \cup ( [n] \backslash Y)^2$ and 
$(\sigma^{-1} (y^-_t),\sigma^{-1} (y^+_t) ) \in \{(x_t^-,x_t^+)\} \cup ([n] \backslash X)^2$. In words, elements in $T_2$ are 
either matched perfectly ($(\sigma (x^-_t),\sigma(x^+_t) = (y^+,y^-)$)  or have their image and preimage outside of 
$\gamma$. We set $T_1  = T \backslash T_2$. By construction, if $t \in T_1$ then 
$$
(\underline{S \otimes S} )_{x_t  y_{t}}  = (-m)^{-1}.
$$
with $m = n(n-1)$. We may thus write
\begin{equation}\label{eq:djei}
P = \prod_{t= 1} ^{\tau_0}  (\underline{S \otimes S} )_{x_t  y_{t}}  \prod_{t= \tau_0+1} ^{\tau}  ( S\otimes S  )_{x_t  y_{t}} =  (-m) ^{ - |T_1|}  \times P_2 \times P_3, 
\end{equation}
where 
$$
P_2 = \prod_{t \in T_2} (\underline{S \otimes S} )_{x_t  y_{t}} \AND P_3 = \prod_{ t \in [\tau_0] \backslash T_1 \cup T_2} (\underline{S \otimes S} )_{x_t  y_{t}}  \prod_{t= \tau_0+1} ^{\tau}  ( S\otimes S  )_{x_t  y_{t}} .
$$

Let $\cF$ be the filtration generated by the variables $T_2$ and 
$\{\sigma(x^\pm_t) , \sigma^{-1}(y^\pm_t)\}_{t\in [\tau ]  \backslash T_2}$. 
By construction, the variable $T_1$ and $P_3$ are $\cF$-measurable. If $\dE_\cF [ \cdot ]$ denotes the conditional 
expectation given $\cF$, it follows that 
\begin{equation}\label{eq:decomP}
\ABS{ \dE P } = \ABS{ \dE \SBRA{ (-m)^{-|T_1|}  P_3 \dE_\cF [  P_2 ]  }   }  \leq 
c \, \dE \SBRA{ n^{-2|T_1|}     \cdot       \dE \SBRA{      |P_3|   | T_1   } \cdot    |\dE_\cF [  P_2 ] | }. 
\end{equation}

We start by estimating $P_3$ in \eqref{eq:decomP}.  Consider the graph, say $\Gamma$, with vertex set 
$\{g_1, \ldots, g_e \}$ and whose  $\tau$ edges 
are $ \{  (x^-_t,y^-_t) , (x^+_t, y^+_t) \} $ (it may have loops and multiple edges). 
Let $L$ be set of $g_s$, $1 \leq s \leq e$, such that  $\sigma (u_s) \ne v_s$ and $g_s \ne  (x^\veps_t,y^\veps_t)$ for all $ t \in T$, 
$ \veps \in \{-,+\}$. Let 
$K$ be the set of edges of $\Gamma$ adjacent to a vertex in $L$.  
We have 
\begin{equation*}
|P_3 | \leq m^{-|K|},
\end{equation*}
(if the $t$-th edge is adjacent to an element in $L$ then $|(\underline{S \otimes S} )_{x_t  y_{t}}| = m^{-1}$ and 
$(S \otimes S )_{x_t  y_{t}} = 0$). We claim that 
$$
|K| \geq \frac{2|L|}{3}.
$$
Indeed, consider the subgraph spanned by the edges in $K$. This subgraph  has vertex set $L'\supseteq L$. Consider a 
connected component of this graph, with vertex set $L'_0$ and edge set $K_0$. Set $L_0 = L \cap L'_0$. It is sufficient to 
check that $|K_0| \geq 2 |L_0| /3$.  If  $|L'_0| \geq 3$ then  the claimed bound follows from $|K_0| \geq |L'_0| - 1$. 
Similarly, if $|L'_0| = 1$ we have $L'_0 = L_0$ and $|K_0| \geq 1$. If $|L'_0 | =2$ and $|K_0| \geq 2$, the bound holds trivially. 
The last remaining case is $|L'_0| = 2$ and $|K_0| = 1$. This case follows from the claim $|L_ 0 | = 1$. 
Indeed,  since $|K_0| = 1$, if both vertices, say $g_a,g_b$ of this connected component are in $L$, then by construction, 
there is a unique $t_0$ such that $g_a$ or $g_b$ are in 
$\{ (x^\veps_{t_0}, y^\veps_{t_0}) : \veps \in \{ -,+ \}\}$ and $\{ g_a, g_b \} = \{ (x^-_{t_0}, y^-_{t_0}) , (x^+_{t_0}, y^+_{t_0}) \} $. 
It implies that $(x_{t_0},y_{t_0})$ has multiplicity $(1,1)$. It contradicts the definition of $L$ (which contains no $g_t$ with 
multiplicity $(1,1)$). So finally, we have proven that 
\begin{equation}\label{eq:P3L}
|P_3 | \leq m^{-2 |L|/3}.
\end{equation}

We now estimate the law of the random variable $|L|$. 
It follows from Equation \eqref{eq:decomP} that we have to estimate $\dP (  |L| =  x | T_1)$.
For $t \in [e]$, let $\cF_t$ be the the filtration generated by the variables 
$\sigma(u_s) , \sigma^{-1}(v_s), s\in [e]\backslash \{t\}$.  
We have 
\begin{equation}\label{eq:bcmarg}
\dP (  S_{u_t v_t}  = 1  | \cF_t ) \leq \frac{1}{\hat n},
\end{equation}
where we introduced the notation $\hat n= n-\tau +1$.
Hence, if $e_2 = e - 2 e_1$, it follows that for any integer $0 \leq x \leq e_2$, 
$$
\dP (  |L| =  x | T_1) \leq {e_2 \choose x } (\hat n) ^{x- e_2} \leq (2\tau)^{x} (\hat n) ^{x- e_2} .
$$
From \eqref{eq:P3L} and using $4 \tau \leq n^{1/3}$, we get for some $c >0$,
\begin{equation}\label{eq:P3f}
\dE [  |P_3| |  T_1 ] \leq \sum_{x =  0}^{\infty}  (2\tau )^{x}  (\hat n) ^{x - e_2} m^{  - \frac{2 x}{3} } \leq (\hat n)^{-e_2} \sum_{x =  0}^{\infty}  2^{-x} \PAR{  \frac{n^{4/3} }{m^{2/3}} }^x  \leq c n^{-e_2} .
\end{equation}

We now give an upper bound on $\dE_\cF [  P_2 ] $ in \eqref{eq:decomP}. This is where Lemma \ref{le:binT} is used.  Let $\tau_2$ 
be the number of $t \in T_2$ such that $(\sigma(x^-_t) ,\sigma(x^+_t) ) \ne (y^-_t ,y^+_t)$, we have 
\begin{eqnarray*}
\dE_\cF  P_2    & =&  \dE_{\cF}\PAR{ 1 - \frac 1 m}^{|T_2 |  -  \tau_2} \PAR{ - \frac 1 m}^{ \tau_2}. 
\end{eqnarray*}
Let $\bar n = n -  |Y| -   \sum_{t \notin T_2} \IND ( \si(u_t) \notin Y) $. 
By a direct counting argument, the law of $\tau_2$ given $\cF$ is given, for $0 \leq x \leq |T_2|$, by
\begin{equation}\label{eq:lawN}
\dP_{\cF} ( \tau_2 = x) = \frac{{|T_2| \choose x} ( \bar  n )_{  2 x} }{Z}\quad \hbox{ with } \quad Z = \sum_{x = 0}^{|T_2|} {|T_2| \choose x} ( \bar  n )_{2x}.
\end{equation}
Indeed, we use the fact that a uniform law conditioned by an event remains uniform. In turn, the
term ${|T_2| \choose x}$ accounts for the possible choices of $t \in T_2$ such that 
$(\sigma(x^-_t) ,\sigma(x^+_t) ) \ne (y^-_t ,y^+_t)$. 
Once these $t$ have been chosen, we use that for all $t \in T_2$, 
$(\sigma (x^-_t),\sigma(x^+_t) ) \in \{ (y_t^-,y_t^+) \} \cup ( [n] \backslash Y)^2$ and 
$(\sigma^{-1} (y^-_t),\sigma^{-1} (y^+_t) ) \in \{(x_t^-,x_t^+)\} \cup ([n] \backslash X)^2$.  
There are $( \bar  n )_{  2 x}$ such choices.
It is immediate to estimate $Z$. 
Indeed, since $\bar n \geq n - 4 k$, $|T_2| \leq k$ and $ k \ll  \sqrt{n}$, we find, for some $c >0$,
$$
Z \geq  \sum_{x = 0}^{|T_2|} {|T_2| \choose x}   (n - 6k)^{2x}   = [ (n - 6k)^2+1]^{|T_2|} \geq  c n^{2|T_2|}. 
$$
Also, we deduce that
\begin{eqnarray*}
\dE_\cF P_2  & = &  \frac 1 Z \sum_{x = 0} ^{|T_2|} {|T_2| \choose x}  (-1)^x \prod_{y=0}^{2 x-1} ( \bar n - y ) \PAR{ 1 - \frac 1 m}^{|T_2 |  - x} \PAR{ \frac 1 m}^{x} \\
& = &   \frac 1 Z \dE (-1)^{\tau_2} \prod_{y=0}^{2\tau_2-1} ( \bar n - y),
\end{eqnarray*}
where $N$ has distribution $\mathrm{Bin}(|T_2 | , 1 / m)$. By Lemma \ref{le:binT} applied to $z = 1$, $k = |T_2|$ 
(which is at most $k$), $p = 1/m$, $q = 1/ \bar n^2$, we deduce that, with $\veps = 3 \sqrt 2 k  / \sqrt n$, for some $c >0$,
\begin{equation}\label{eq:P2f}
 \ABS{ \dE_{\cF}   P_2  }\leq c \PAR{ \frac{\veps}{n^2} }^{|T_2|}.
\end{equation}

Since $|T_2| + |T_1| = e_1$ and $ 2 e_1 + e_2 = e$, we obtain in \eqref{eq:decomP} from \eqref{eq:P3f} and \eqref{eq:P2f} that, for some $c >0$,
$$
\ABS{ \dE P } \leq c  n^{-e} \veps^{e_1} \dE \veps^{-|T_1|}. 
$$

It thus remains to show that $ \dE \veps^{-|T_1|}$ is of order $1$. If $|T_1| = x$, then there are least $\lceil x / 2 \rceil$ distinct 
$x^\veps_t$ with $t \in [k], \veps \in \{-,+\}$ such that $\sigma(x^\veps_t) \in Y$. From \eqref{eq:bcmarg}, we find that 
$$
\dP ( |T_1| = x ) \leq { 2k \choose \lceil x / 2 \rceil} { |Y| \choose \lceil x / 2 \rceil} (\hat n)^{-\lceil x /2 \rceil} \leq c \PAR{ \frac{2k}{\sqrt n} }^{2\lceil x / 2 \rceil} \leq c \PAR{ \frac{2k}{\sqrt n} }^{x} .
$$
We deduce
$$
\dE \veps^{-|T_1|} \leq c \sum_{x = 0}^\infty \veps^{-x}  \PAR{ \frac{2k}{\sqrt n} }^{ x}  =  c \sum_{x = 0} ^\infty (\sqrt 2/3)^x. 
$$
The latter series is convergent and it concludes the proof when all edges are consistent. 

{\em We now extend to the case of inconsistent edges}. Let us say that $x \in \bm X = \{ x_t : t \in [\tau]\}$ is  inconsistent with degree 
$\delta \geq 2$, if there are $y_{1}, \cdots, y_{\delta}$ distinct elements of $\bm Y = \{ y_t : t \in [\tau]\}$ such that for any 
$1 \leq s \leq \delta$,  $(x,y_s)$ is  in $\{ f_t  : t \in [\tau] \}$.  We define similarly the degree of an inconsistent vertex $\bm Y$. 
The (vertex) inconsistency of $\bm f = (f_1, \cdots , f_\tau)$ is defined as the sum of 
the degrees of inconsistent vertices in $\bm X \cup \bm Y$. The  inconsistency of $\bm f$, say $\hat b$, is at most $2b$. Assume that $x$ is an 
inconsistent vertex of degree $\delta \geq 2$ and $(x,y_1)$ and  $(x,y_2)$ are in $\{ f_t  : t \in [\tau] \}$. Observe that 
$(S \otimes S)_{x y_{1}} (S \otimes S)_{x y_2} = 0$, hence, for any integers $p_1,p_2 \geq 1$,
\begin{eqnarray*}
(\underline{S \otimes S})_{x y_1}^{p_{1}}  (\underline{S \otimes S})_{x y_2}^{p_{2}}   &= &\PAR{ (S \otimes S)_{x y_{1}} - \frac 1 m }^{p_{1}}  \PAR{(S \otimes S)_{x y_2} - \frac 1 m }^{p_2} \\
&= & (\underline{S \otimes S})_{x y_1}^{p_{1}} \PAR{-\frac{1}{m}}^{p_2}   +  \PAR{-\frac{1}{m}}^{p_1}   (\underline{S \otimes S})_{x y_2}^{p_{2}}  -  \PAR{-\frac{1}{m}}^{p_1+p_2}.
\end{eqnarray*}
Similarly, if $q_1 \geq 1$,
\begin{eqnarray*}
(\underline{S \otimes S})_{x y_1}^{p_{1}} ({S \otimes S})_{x y_1}^{q_{1}} (\underline{S \otimes S})_{x y_2}^{p_{2}} &= & (\underline{S \otimes S})_{x y_1}^{p_{1}} ({S \otimes S})_{x y_1}^{q_{1}} \PAR{-\frac{1}{m}}^{p_2}.
\end{eqnarray*}
If $p_i$, $q_i$ is the number of occurences of $(\underline{S \otimes S})_{x y_i}$ and  $({S \otimes S})_{x y_i}$ in the product 
\eqref{eq:djei}, we thus have decomposed \eqref{eq:djei} into at most $3$ terms of the same form as \eqref{eq:djei} up to a 
factor $(-m)^{ - p'}$. Each of these terms is associated to a new sequence $\bm f'$ of colored edges which is a subsequence 
of $\bm f$ with $e'\leq e $ distincts elements
and $e'_1 \geq  e_1$ elements of multiplicity $(1,1)$. Moreover we have $p'/2 + e'\geq e$ and the inconsistency of 
$\bm f'$ is at most $\hat b -1$. By repeating this for all inconsistent vertices, we may decompose \eqref{eq:djei} into at 
most than $3^{\hat b} \leq 9^b$ terms of the form \eqref{eq:djei} with $e' \leq e$, $e'_1 \geq e_1$ where all edges are 
consistent, multiplied by a factor $(- m)^{-p'}$ with $p' / 2  + e' \geq e$. Applying the first part of the proposition to each term, 
the conclusion follows. 

{\em Case where $\sigma$ is a uniform matching}. The proof follows from the same line. The only noticable difference is for 
the distribution of the random variable $\tau_2$ in \eqref{eq:lawN}. We will apply Lemma \ref{le:binT} with $z = 2$ 
(see \cite[Proposition 8]{bordenaveCAT}).
\end{proof}

\subsection{Trace method of F\"uredi and Koml\'os}

We adapt here the content of Subsection \ref{subsec:FK} to tensors of permutation matrices. 
We explain here how to perform this adaptation.

\label{subsec:FKT}

We fix a collection $(a_i), i \in [d],$ of matrices such that
$\max_i ( \| a_i \| \vee  \| a_i ^{-1} \|^{-1} ) \leq \veps^{-1}$ for some $\veps >0$. Then, $B_\star$ is the corresponding 
non-backtracking operator in the free group. The constants may depend implicitly on $r$, $d$ and $\veps$. 
We have the following analogs of Proposition \ref{prop:normB} and Proposition \ref{prop:normR}.

\begin{proposition} \label{prop:normBT} 
Let $\veps > 0$. If $1 \leq \ell \leq  \log n$, then the event
\begin{equation*}
 \| \underline{ \bm B}^{(\ell)} \| \leq ( \log n) ^{50} ( \rho (B_\star) + \veps ) ^\ell,
\end{equation*}
holds with the probability at least $1 -  c e^{-\frac{\ell \log n}{ c \log \log n}} $ where $c >0$ depends on $r$, $d$ and $\veps$.
\end{proposition}

\begin{proposition} \label{prop:normRT} 
For any $1 \leq k, \ell \leq  \log n$, the event
$$ \| \bm  R_k^{(\ell)} \| \leq \PAR{\log n}^{100} \rho_1^\ell,$$
holds with the probability at least $1 -  c e^{-\frac{\ell \log n}{ c \log \log n}} $ where $c >0$ and $\rho_1 >0$  depend 
on $r$, $d$ and $\veps$.
\end{proposition}

We only explain the differences arising in the proof of Proposition \ref{prop:normBT}, the case of Proposition 
\ref{prop:normRT} being identical.  Let $m$ be a positive integer. The computation leading to \eqref{eq:trBl2} gives 
\begin{eqnarray}
\dE \| \underline{ \bm B}^{(\ell)}  \| ^{2 m} & \leq  &   \sum_{\gamma \in \bm W_{\ell,m} } |w(\gamma)| \,   \tr | a ( \gamma ) |  ,  \label{eq:trBl2T}
\end{eqnarray}
where $\bm W_{\ell,m}$ is the set of  $\gamma = ( \gamma_1, \ldots, \gamma_{2m})$ such that $\gamma_j = (\gamma_{j,1} , \ldots, \gamma_{j,\ell+1}) \in \bm F^{\ell+1}$, $\gamma_{j,t} = ( x_{j,t}, i_{j,t})$ and for all $j = 1, \ldots, m$, the boundary condition $
\gamma_{2j,1} = \gamma_{2j+1, 1}$   and $ \gamma_{2j-1,\ell+1} = \gamma_{2j, \ell+1}$, with the convention that $\gamma_{2m+1} = \gamma_{1}$. In \eqref{eq:trBl2T}, we have also set 
$$
w(\gamma) = \dE \prod_{t=1}^{\ell} (\underline{S_{i_{j,t}} \otimes S_{i_{j,t}} } )_{x_{j,t}  x_{j,t+1}} \AND a(\gamma) = \prod_{j=1}^{2m}  a ( \gamma_j)^{\veps_j}.
$$
and $a ( \gamma_j)^{\veps_j}$ is $a ( \gamma_j)$ or $a ( \gamma_j)^{*}$ depending on the parity of $j$.

Let $X =  [n]$. Exactly as in the proof of Proposition \ref{prop:normB}, we define the isomorphism class 
$\gamma \sim \gamma'$ if there exist permutations $\sigma \in \mathcal S_{X^2_{\ne}}$ and 
$(\tau_x)_x \in (\mathcal S_d) ^{X^2_{\ne}}$ such that, with $\gamma'_{j,t}  = (x'_{j,t} , i'_{j,t})$, for all 
$1 \leq j \leq 2m$, $1 \leq t \leq \ell+1$, $x'_{j,t} = \sigma ( x_{j,t})$, $i'_{j,t}  = \tau_{x_{j,t}}( i_{j,t})$ and $(i'_{j,t}  ) ^*  = \tau_{x_{j,t+1}}( (i_{j,t}) ^* )$.  
For each $\gamma \in W_{\ell,m}$, we define $\gamma_-$, $\gamma_+$ and $G_\gamma$ as in Definition \ref{defT}, the 
vertex set of $G_\gamma$ is $V_\gamma = \cup_j V_{\gamma_j}$ and its  edge set is $E_\gamma = \cup_j E_{\gamma_j}$).   
We also define $\bm V_\gamma $, $\bm E_\gamma$ and $\bm G_\gamma$ as in Definition \ref{defT}. The graph 
$G_\gamma$ is the union of the graphs $G_{\gamma^-}$ and $G_{ \gamma^+}$. Since $G_{\gamma^\pm}$ is connected, 
$G_\gamma$ has at most two connected components, it follows that $|E_\gamma| - |V_\gamma| + 2 \geq 0$.  We may also 
define a canonical element in each isomorphic class as in the proof of Proposition \ref{prop:normB}.  The analog of Lemma 
\ref{le:enumpath} is the following.

\begin{lemma}\label{le:enumpathT}
Let $\bm \cW_{\ell,m} (v,e) $ be the subset of canonical paths with $|V_\gamma| = v$ and $|E_\gamma |= e$. We have 
$$
| \bm \cW _{\ell,m} (v,e) | \leq   (4 d \ell  m )^{12 m \chi  + 20 m },
$$
with $\chi = e - v +2 \geq 0$. 
\end{lemma}

\begin{proof}
We repeat the proof of Lemma \ref{le:enumpath} where we replace $\gamma$ by the sequence $(\gamma^-,\gamma^+)$, it 
then essentially amounts to replace $2m$ by $4m$ (the merging time of $\gamma^+_{1}$ may be empty if $G_{\gamma^-}$ 
and $G_{\gamma^+}$ are disjoint).  
\end{proof}

There is also an analog of Lemma \ref{le:sumpath}.

\begin{lemma}\label{le:sumpathT} Let $\rho = \rho( B_\star) + \veps$ and any postive integer $k_0$, there exists a constant 
$c >0$ depending on $r$, $d$ and $\veps$ such that for any $\gamma\in \bm \cW_{\ell,m}(v,e)$, 
\begin{equation}\label{eq:sumpathT}
\sum_{\gamma':  \gamma' \sim \gamma} \tr |a (\gamma' ) | \leq  c^{m + \chi + e_1} n^v   \rho_0^{2(\ell m - v) } \rho^{2\ell m},
\end{equation}
where $\chi = e - v +2$, $e_1$ is the number of edges of $\bm E_\gamma$ with multiplicity $(1,1)$ and 
$$
\rho_0 = \max \NRM{ \prod_{s=1}^{k_0} a_{i_s} }^{\frac 1 {k_0}},
$$ 
and the maximum is over all non-backtracking sequences $(i_1, \cdots, i_{k_0})$, that is $i_{s+1} \ne i_{s ^*}$. Moreover, for 
all $k_0$ large enough, we have $\rho_0 \leq \rho$. \end{lemma}

\begin{proof}
In order to adapt the proof of Lemma \ref{le:sumpath}, we may consider the graph $\bm G_\gamma$ with vertex set 
$\bm V_\gamma$ and colored edges $\bm E_\gamma$. An issue is that edges visited once on this graph by $\gamma$ are 
not necessarily edges of multiplicity $(1,1)$: for example, in Figure \ref{fig:impoT} all edges of $\bm E_\gamma$ are visited 
exactly once but none of them is of multiplicity $(1,1)$.

To circumvent this difficulty, we introduce a new graph.  Consider the following equivalence class on $\bm V_{\gamma}$: 
$x \sim x'$  if there exists a sequence $y_0,\ldots,y_t$ in $\bm V_\gamma$ such that $y_0 = x$, $y_t = x'$ and for all $s \in [t]$, $\{y^-_{s-1},y^+_{s-1} \} \cap \{y^-_{s},y^+_{s} \} \ne \emptyset$  
(in words, we glue iteratively together elements of $\gamma$ which share some common vertices of $G_\gamma$).  
If $\widetilde{V_\gamma}$ is the set of equivalence classes for this equivalence relation, we may define the graph 
$\widetilde{G_\gamma}= (\widetilde{{V_\gamma}}, \widetilde{ E_\gamma})$ as the quotient graph of $\bm G_{\gamma}$. 
More precisely, $\widetilde{E_\gamma}$ is obtained from $\bm E_{\gamma}$ by identifying two edges $[u,i,u']$ and $[v,i,v']$  if there exists a sequence $[x_0,i,y_0],\ldots,[x_t,i,y_t]$ in $\bm E_\gamma$ such that $(x_0,i,y_0) = (u,i,u')$, $(x_t,i,y_t)= (v,i,v')$ and for all $s \in [t]$, $\{(x^-_{s-1},i,y^-_{s-1}),(x^+_{s-1},i,y^+_{s-1}) \} \cap \{(x^-_{s},i,y^-_{s}),(x^+_{s},i,y^+_{s}) \} \ne \emptyset$. For example if  $\gamma  \in \bm  E^5$ is the path of Figure \ref{fig:impoT}, then $\widetilde{V_\gamma}$ has a single vertex with two colored loops attached. If $\widetilde e = | \widetilde E_\gamma|$, $\widetilde v = | \widetilde V_\gamma|$ and 
$\widetilde \chi = \widetilde e - \widetilde v +1$,  by iteration on the successive gluings of vertices of $G_\gamma$, we find easily that $\widetilde \chi \leq \chi$.

Similarly, we define $\widetilde \gamma = ((\widetilde x_{j,t} , i_{j,t} ))_{j,t}$ from the original path 
$\gamma = ((x_{j,t} , i_{j,t} ))_{j,t}$. Then, by construction, the number of edges of multiplicty $1$ in 
$\widetilde \gamma$, say $\widetilde e_1$, is at most $e_1$, the number of edges of muliplicity $(1,1)$ in $\gamma$.
We may then simply repeat the proof of Lemma \ref{le:sumpath}  with $\widetilde \gamma$ and $\widetilde G_\gamma$ 
in place of $\gamma$ and $G_\gamma$.  
\end{proof}

We finally give the analog of Lemma \ref{le:isopath} for $w(\gamma)$ defined below \eqref{eq:trBl2T}. 

\begin{lemma}\label{le:isopathT}
There exists a constant $c >0$ such that for any $\gamma\in \bm \cW_{\ell,m}(v,e)$ and $8 \ell m \leq  n ^{1/3}$, 
$$
 | w (\gamma) | \leq  c^{m + \chi}  \PAR{ \frac{1} {n} }^e  \PAR{\frac{ 12 \ell m }{ \sqrt n} } ^{  (e_1 - 4 \chi - 8  m)_+}.  
$$
with $\chi = e - v +2$ and $e_1$ is the number of edges of $\bm E_\gamma$ with multiplicity $(1,1)$. Moreover, 
$$
e_1 \geq 2 ( e -  \ell m)  - 4 \chi - 8m. 
$$
\end{lemma}

\begin{proof}
Consider the graph $\widetilde G_\gamma$ defined in Lemma \ref{le:sumpathT}. Let $\widetilde e_1$ be the the number 
of edges of multiplicty $1$ in $\widetilde \gamma$. Arguing as in the proof of Lemma \ref{le:isopath}, 
$\widetilde e_1 \geq 2(e - \ell m)$. The last statement is thus a consequence of \eqref{eq:te1}. The first statement is 
a consequence   Proposition \ref{prop:exppathT} and that the number of inconsistent edges of $\gamma$ is at most in 
$4 \chi + 8m$ (as already used in the proof of Lemma \ref{le:sumpathT}). 
\end{proof}

All ingredients are in order to prove Proposition \ref{prop:normBT}.

\begin{proof}[Proof of Proposition \ref{prop:normBT}]
For $n \geq 3$, we define 
\begin{equation*}\label{eq:choicemT}
m = \left\lfloor  \frac{ \log n }{25 \log (\log n)} \right\rfloor.
\end{equation*}
We may then repeat the proof of Proposition \ref{prop:normB} with the exponents slightly modified. \end{proof}

\subsection{Proof of Theorem \ref{th:mainBT}}
\label{subsec:netT}

With Proposition \ref{prop:normBT} and Proposition \ref{prop:normRT}, The net argument used in Subsection \ref{subsec:net} to prove Theorem \ref{th:mainB} can be applied exactly in the same way to prove Theorem \ref{th:mainBT}.

\section{Proof of Theorem \ref{thm:strong} and Theorem \ref{thm:strongT}}
\label{sec:cor}

The proof of Theorem \ref{thm:strong} and Theorem \ref{thm:strongT} has become standard in the last decade, it is 
based on the the  linearization trick. Let us outline it here. 
Let $U=(U_1, \ldots , U_d)$ be unitary elements of a unital $C^*$-algebra $A$ and $V=(V_1, \ldots , V_d)$ be unitary  elements of a unital 
$C^*$-algebra $B$. Let $P$ be a non-commutative polynomial in $d$ free variables and their adjoints.
Then, from Pisier \cite[Proposition 6]{MR1401692}, the following are equivalent, 
\begin{enumerate}[(i)]
\item For any $P$, $||P(U)||=||P(V)||$.
\item For any polynomial $P$ with \emph{matrix coefficients} with matrices of any size, $||P(U)||=||P(V)||$.
\item For any integer $r$, and $r\times r$ matrices $a_0,\ldots, a_d$, 
$$||a_0\otimes 1+ a_1\otimes U_1+ \ldots + a_d\otimes U_d||=||a_0\otimes 1+ a_1\otimes V_1+ \ldots + a_d\otimes V_d||.$$
\item 
For any integer $r$, and $r\times r$ matrices $a_0,\ldots, a_d$ such that 
$a_0\otimes 1+ a_1\otimes U_1+ \ldots + a_d\otimes U_d$ and $a_0\otimes 1+ a_1\otimes V_1+ \ldots + a_d\otimes V_d$
are self-adjoint, 
$$||a_0\otimes 1+ a_1\otimes U_1+ \ldots + a_d\otimes U_d||=||a_0\otimes 1+ a_1\otimes V_1+ \ldots + a_d\otimes V_d||.$$
\end{enumerate}
Here, all norms are $C^*$-algebra norms. For a closely related linearization trick, we refer to \cite[p256]{MR3585560}  (Exercise 1 following  Proposition 4). 

Back to our restricted permutation matrices $(S_i)_{| \IND^\perp}$ and $(S_i \otimes S_i)_{|V^\perp}$,  the important point is 
that instead of treating general non-commuting polynomials in  $(S_i)_{| \IND^\perp}$ (resp.  $(S_i \otimes S_i)_{|V^\perp}$) as
in (i) and (ii), we treat degree one self-adjoint polynomials with matrix values as in (iv).  In Theorem \ref{th:main}, Theorem 
\ref{th:mainT}, we prove the asymptotic convergence of operator norms of operators of the form \eqref{eq:defA} for any integer 
$r$, and $r\times r$ matrices $a_0,\ldots, a_d$,  therefore by the above criterion (iv), it implies the result for any polynomial $P$ 
 in $(S_i)_{| \IND^\perp}$ (resp.  $(S_i \otimes S_i)_{|V^\perp}$)
with matrix coefficients, as in (ii).

\section{Proofs of auxiliary results}

\label{sec:aux}
\subsection{Proof of Proposition \ref{prop:inclusion}}

We note that if the matrices $a_i$ have non-negative entries then an analog of the Alon-Boppana lower bound 
\eqref{eq:AB} holds in this context, see \cite{CECCHERINISILBERSTEIN2004735}. The proof of Proposition 
\ref{prop:inclusion} is based on the notion of spectral measure. If $A$ is of the form \eqref{eq:defA} and the symmetry 
condition \eqref{eq:sym} holds, then $A$ is a bounded self adjoint operator. If $\phi \in  \dC^r \otimes \ell^{2} (X)$, we denote by 
$\mu^\phi_A$ the spectral measure of $A$, that is 
$$
\mu^\phi_ A (\cdot) = \langle \phi , E (\cdot ) \phi \rangle.
$$ 
where $E$ is the spectral resolution of the identity of $A$.
We have for any integer $k \geq 0$, 
\begin{equation}\label{eq:momsm}
\int \lambda^k d \mu_A^\phi ( \lambda) = \langle \phi, A^k \phi \rangle. 
\end{equation}
If $x \in X$, we also define the spectral measure 
$$
\mu_A ^{x} (\cdot )  =  \frac 1 r \tr \BRA{ E(\cdot)_{xx} }= \frac 1 r \sum_{j=1}^{r} \mu_A^{f_j \otimes \delta_x} (\cdot),
$$ 
where we used the notation \eqref{eq:matrixvalued} for $E(\cdot)_{xx}$  and $(f_1, \ldots, f_r)$ is an orthonormal basis of 
$\dC^r$. 
Moreover if $X = [n]$ is finite and $(\psi_1, \ldots, \psi_{nr})$ is an orthonormal basis of eigenvectors of $A$ with 
eigenvalues $\lambda_1, \ldots, \lambda_{nr}$, we have
$$
\mu_A ^x = \frac 1 {r} \sum_{k=1}^{nr}  \| (\psi_k)(x) \|^2 _2 \delta_{\lambda_k}.
$$
where $\psi(x) \in \dC^r$ is the projection of $\psi$ on  $\dC^r \otimes \{ \delta_x \}$. Recall the orthogonal decomposition of 
$H_0 \oplus H_1$ of $\dC^r \otimes \ell^2(X)$. Let us assume without loss of generality that $\psi_1, \ldots , \psi_r$ is an  
orthonormal basis  of $H_1$. Then, for $1 \leq k \leq r$, $\psi_k = (f_k \otimes \IND) / \sqrt n $  where $(f_1, \ldots, f_r)$ is an
 orthonormal basis of $\dC^r$. We set 
$$
\mu_{A_{|H_0}} ^x = \frac 1 {r} \sum_{k=r+1}^{nr}  \| (\psi_k)(x) \|^2 _2 \delta_{\lambda_k}.
$$
It follows that 
$$
| \mu_A ^x  - \mu^x_{A_{|H_0}} | (\dR) = \frac 1 {r} \sum_{k=1}^{r}  \| (\psi_k)(x) \|^2 _2 = \frac 1 n,
$$
and by construction 
$$
\SUPP (\mu_{A_{|H_0}}^x ) \subset \sigma(A_{|H_0}).
$$
We readily deduce the following lemma. 
\begin{lemma}\label{le:projdede}
Assume that $X = [n]$ and that the symmetry condition \eqref{eq:sym} holds. If $f: \dR \to \dR$ is a function uniformly 
bounded by $1$ and $\int f d\mu_A^x > 1/n$ then $\sigma( A_{|H_0}) \cap \SUPP ( f) \ne 0$. 
\end{lemma}

We may now prove Proposition \ref{prop:inclusion}. 
\begin{proof}[Proof of Proposition \ref{prop:inclusion}]
Since $\sigma(A_\star)$ is compact, it is sufficient to prove that for any $\lambda  \in \sigma(A_\star)$ there exists an integer 
$h = h(\veps,\lambda)$ such that if $(G^\sigma,x)_{h }$ contains no cycle for some $x \in [n]$ then
$
 \sigma (A_{|H_0}) \cap [\lambda -\veps ,\lambda + \veps]\ne 0. 
$ 

Let $S = \{ g_i: i \in [d] \}$ be the symmetric generating set of the free group $X_\star$ and let $o$ be its unit. Let $M$ be a 
bounded operator on $\dC^r \otimes \ell^2 ( X_ \star )$ in the $C^\star$-algebra generated by finite linear combinations of 
operators of the form $b \otimes \lambda(g)$. We introduce the standard tracial state $\tau$ defined in \eqref{eq:deftau} 
(with $r$ instead of $k$). Then 
$$
\int \lambda^k d \mu^{o} _{A_\star} (\lambda)  =\tau (A^k_\star)
$$
and $$\SUPP ( \mu_{A_\star}^o) = \sigma(A_\star).$$ In particular, if $\lambda \in  \sigma(A_\star)$ and 
$f (x) = \max(0,1 - |x  -  \lambda| / \veps)$, we have
$$
\eta =  \int f \, d\mu_{A_\star}^o > 0. 
$$
Set $T = \sum_i \|a_i\|$ and $I = [ - T , T ]$. From Stone-Weierstrass Theorem, there exists a polynomial $p$ of degree $m$ 
such that for any $x \in I$,
$$
| f (x) - p(x) | < \eta / 4. 
$$
Since the norms of $A$ and $A_\star$ are bounded by $T$, we deduce that 
$$
\ABS{ \int f \,  d \mu^x_A  - \int f\,  d \mu^{o} _{A_\star} } < \ABS{ \int p \, d \mu^x_A  - \int p \, d \mu^{o} _{A_\star} } + \eta / 2. 
$$
However, from \eqref{eq:momsm}, $
 \int \lambda^ k \, d \mu^x_A (\lambda)
$ is a function of $(G^\sigma,x)_h$. More precisely, we have 
$$
 \int \lambda^ k \, d \mu^x_A (\lambda) = \frac 1 r  \sum_{\gamma} \tr \prod_{t=1}^{k} a_{i_t}, 
$$
where the sum is over all  closed paths $\gamma = (x_1, i_1, x_2, i_2, \ldots, x_{k+1})$ in $G^\sigma$ of length $k$ with 
$x_1 = x_{k+1} = x$. We deduce that if $(G^\sigma,x)_h$ contains no cycle then, for any integer $0 \leq k \leq 2h$, 
$$
 \int \lambda^ k \, d \mu^x_A (\lambda) = \int \lambda^k  \, d \mu^{o} _{A_\star} (\lambda).
$$
Hence, if $2h \geq m$, we obtain
$$
\int p \, d \mu^x_A  = \int p \, d \mu^{o} _{A_\star} ,
$$
and consequently
$$
 \int f \,  d \mu^x_A  >  \int f\,  d \mu^{o} _{A_\star}   - \ABS{ \int p \, d \mu^x_A  - \int p \, d \mu^{o} _{A_\star} } - \eta / 2 > \eta / 4. 
$$
If $n \geq 4 / \eta$ then we may conclude using Lemma \ref{le:projdede} (note that the condition $n \geq \eta / 4$ is included in the condition $(G^\sigma,x)_{\lceil \eta / 4 \rceil}$ contain no cycle).
\end{proof}

\subsection{Proof of Theorem \ref{th:PF}}

\begin{proof}
Let us give a proof of \eqref{PF2} which is the result that we have actually used. We only prove \eqref{PF1} when $L$ itself is 
positive semi-definite. We may argue as in Gross \cite[Theorem 2]{MR0339722}. From Lemma \ref{le:stabnn}, $L$ maps 
hermitian and skew-hermitian matrices to hermitian and skew-hermitian matrices.  We deduce that if $\lambda$ is a (non-
negative) eigenvalue of $L$ with eigenvector $x$ then both its hermitian and skew-hermitian parts are eigenvector. At least one 
of the two parts is non-trivial and it follows that there exists an hermitian eigenvector $x$ with eigenvalue $\rho$. We write 
$x = a - b $ with $a,b$ positive semi-definite. We have $|x| = a +b$ and
\begin{align*}
\rho \langle |x| , |x| \rangle & =  \rho \tr ( x^2 ) = \langle   Lx, x \rangle \\
& = \langle L  a ,  a  \rangle + \langle L b  ,  b \rangle - 2   \langle L a  ,  b \rangle \\
& \leq  \langle L  a ,  a  \rangle + \langle L b  ,  b \rangle + 2   \langle L a  ,  b \rangle = \langle L |x| , |x| \rangle, 
\end{align*}
where at the last line, we have used that if $a,c$ are positive semi-definite, then $c^{1/2} a  c^{1/2}$ is positive semi-definite 
and $\langle a , c \rangle = \tr ( a c) = \tr (  c ^{1/2}  a  c^{1/2}) \geq 0$. Since $\rho$ is the operator norm of $L$, we get 
$$
\rho \langle |x| , |x| \rangle \leq \langle L |x| , |x| \rangle \leq \rho \langle |x| , |x| \rangle.
$$
Hence $\langle (\rho - L) |x| , |x| \rangle =0$. Since $\rho - L $ is positive semi-definite, we thus have proved that $|x|$ is 
an eigenvector of $L$ with eigenvalue $\rho$. It concludes the proof of \eqref{PF1} when $L$ is positive semi-definite.

We may then prove \eqref{PF2}. From Lemma \ref{le:stabnn}, $ (L^n )^* L^n $ is positive semi-definite and of negative type. 
Let $\rho_n^{2n}$ be the operator norm of $(L^n )^* L^n $, Gelfand's formula implies that $\rho_n$ converges to $\rho$. 
Moreover, from what precedes, $ (L^n )^* L^n $ has a positive semi-definite eigenvector $y_n$ with $\| y_n \|_2 = 1$ with
 eigenvalue $\rho_n^{2n}$.  From the spectral theorem, we have
\begin{equation}\label{eq:lnlnx}
\rho_{n}^{2n} \| x \|^2_{2} \geq \| L^n x \|_{2}^2 = \langle x , (L^n )^* L^n  x \rangle \geq \rho_{n}^{2n} \ABS{\langle x , y_n \rangle}^2.
\end{equation}
However, since for any positive semi-definite $x,y$, $\tr ( x y )  = \tr ( y^{1/2} x y^{1/2} ) \leq \tr ( x ) \| y \|$ (where $\| y \|$ is the 
operator norm), we deduce that 
$$
1 = \tr ( y_n^2  ) \leq \tr ( y_n  )  =  \tr (x^{-1/2} x^{1/2} y_n x^{1/2} x^{-1/2} ) \leq \tr (  x^{1/2} y_n x^{1/2} ) \| x^{-1} \| = \langle x , y_n \rangle \| x^{-1} \| .  
$$
Hence $\ABS{\langle x , y_n \rangle}^2$ is lower bounded uniformly in $n$ by some $\delta > 0$. Taking the power $1/(2n)$ in 
\eqref{eq:lnlnx} concludes the proof of \eqref{PF2}.
\end{proof}

\subsection{Proof of Lemma \ref{le:binT}}

\begin{proof}
We adapt the proof of \cite[Lemma 9]{bordenaveCAT}. Let $f(x) =  \dE (-1)^N \prod_{n=0}^{2N-1} \PAR{ 1/ \sqrt{x} - z n}$,  $\delta  = - z \sqrt q $ and $\veps = (1 -  p - p/q)/ ( 1 - p)$.  We write
$$
f(q) = \sum_{t = 0} ^k {k \choose t} p^{t} ( 1- p)^{k-t} (-1)^t \prod_{n=0}^{2t-1} \PAR{ \frac 1 {\sqrt{q}} - zn  }  = \PAR{ 1 - p}^k   \sum_{t = 0} ^k {k \choose t} ( -1 + \veps)^{t}  \prod_{n=0}^{2t-1} \PAR{ 1 + \delta n}.$$

We will use that by assumption, $p \leq 1/4$ and $|\veps| \leq  (4/3) ( 1 / \sqrt 8 ) \leq 1/2$.  We write
$$
  \prod_{n=0}^{2t-1} \PAR{ 1 + \delta n} = 1 + \sum_{s=1}^{2t-1} \delta^s \sum_{(s)}  \prod_{i=1}^s n_i  = 1 +  \sum_{s=1}^{2t-1} \delta^s P_s(2t),
$$
where $\sum_{(s)}$ is the sum over all $(n_i)_{1\leq i \leq s}$ all distinct and $1 \leq n_i \leq 2t-1$. We observe that 
$t \mapsto P_s(t)$ is a polynomial of degree $2s$ in $t$, which vanish at integers $0 \leq t \leq s$  and for integer $t \geq s+1$, 
$0 \leq P_s (t) \leq ( \sum_{1 \leq n \leq t-1} n )^s \leq (t^{2}/2)^ s$.  Setting $P_0 (t) =1$, we have
\begin{eqnarray}\label{eq:fpabs}
|f(q)| \leq \ABS{  \sum_{s=0}^{2k-1}   \delta^s   \sum_{t=0} ^k  {k \choose t} (-1+\veps)^{t}  P_s(t)}. 
\end{eqnarray}
We may then repeat verbatim the proof of \cite[Lemma 9]{bordenaveCAT}. In \eqref{eq:fpabs}, for large values of $s$ we 
have the rough bound 
\begin{eqnarray*}
  \sum_{s =  \lfloor \frac{k -1}{2} \rfloor +1 }^{2k-1}   | \delta| ^s   \sum_{t=0} ^k  {k \choose t} \ABS{ -1 +  \veps}^t \ABS{  P_s(2t)}  \leq    \sum_{s =   \lfloor \frac{k -1}{2} \rfloor +1}^\infty | \delta| ^s  3 ^{k} k^{2s}2 ^{s} 
 \leq   2 . 3^k \PAR{ 2 | \delta|  k^2 }^{\frac k 2} = 2 ( 3 \sqrt 2  k  z^{1/2} q^{1/4})^k,
\end{eqnarray*}
where we have used that $|1 - \veps| \leq 2$, $\sum_{t=0}^k {k \choose t} 2^t  = 3^k$, $\ABS{  P_s(2t)} \leq (2 k^{2}) ^s$ and, 
at the third step, that $2 |\delta| k ^2 = 2 z  k ^2 \sqrt q \leq 1/2$ and $\sum_{s \geq r}  x^k \leq 2x^r$ if $0 \leq x \leq 1/2$.

For $1 \leq s \leq  \lfloor (k -1)/2 \rfloor$, there are some algebraic cancellations in \eqref{eq:fpabs}. Consider the derivative 
of order $m$ of $( 1+ x)^k = \sum_{t=0}^{k} {k \choose t} x^t$.  It vanishes at $x = -1$ for any $0 \leq m \leq k-1$. We get that 
for any $0 \leq m \leq k-1$,
$$
0 =  \sum_{t=0}^{k} {k \choose t} (-1)^t (t)_m.
$$
Since $Q_m ( x )  = (x)_m$ is a monic polynomial of degree $m$, $Q_0, \ldots , Q_{k-1}$ is a basis of $\dR_{k-1}[x]$, the 
real polynomials of degree  at most $k-1$. Hence, by linearity that for any  $P \in \dR_{k-1} [x]$, 
\begin{equation}\label{eq:PPPP}
0 =  \sum_{t=0}^{k} {k \choose t} (-1)^t P(t).
\end{equation}

Now, using again the binomial identity, we write
$$
\PAR{1 - \veps}^t  = \sum_{r=0}^{t} (-\veps)^r {t \choose r} =  T_{k,  s} (t) + R_{k,s} (t),
$$
where $T_{k,s} (t) = \sum_{r=0}^{k-1 - 2s} (-\veps)^r {t \choose r}$ is a polynomial in $t$ of degree $k-1 - 2s $. Using 
$0 \leq \veps \leq 1/2$ and, if $t \geq 2$, $0 \leq r \leq t$, $
{t \choose r} \leq 2^{t-1}$ (from Pascal's identity), 
$$
|R_ {k,s} (t)|   =    \ABS{ \sum_{r=k - 2s}^{t}   (-\veps)^r  {t \choose r} } \leq 2^{t-1} \sum_{r=k - 2s}^{t}   \veps^r    \leq 2^{t} \veps^{k - 2 s}.
$$
If $t \in \{0,1\}$ this last inequality also holds. Since $P_s$ is a polynomial of degree $2s$, we get from \eqref{eq:PPPP} that
\begin{eqnarray*}
  I = \sum_{s=0}^{ \lfloor \frac{k -1}{2} \rfloor}   \delta^s   \sum_{t=0} ^k  {k \choose t} (-1+\veps)^{t}  P_s(2t) =   \sum_{s=0}^{ \lfloor \frac{k -1}{2} \rfloor}   \delta^s   \sum_{t=0} ^k  {k \choose t} ( -1) ^t  R_{k,s} (t) P_s(2t).
\end{eqnarray*}
Taking absolute values and using again $|P_s(2 t)| \leq 2^s t^{2s} $, we find
\begin{eqnarray*}
\ABS{ I }   \leq   \sum_{s=0}^{ \lfloor \frac{k -1}{2} \rfloor}     |\delta|^s   \sum_{t=0} ^k  {k \choose t} 2^{t} \veps^{k-2s} k^{2s}2 ^{s}  =   ( 3 \veps) ^{k}  \sum_{s=0}^{ \lfloor \frac{k -1}{2} \rfloor}   \PAR{ \frac{2 |\delta| k^2}{  \veps^2}} ^s \leq  2 ( 3 \veps) ^{k}  \PAR{ \frac{2|\delta| k^2}{   \veps^{2} }} ^{\frac k 2 }, 
\end{eqnarray*}
where at the last step, we use that $ |\delta| k^2 /  \veps^2 \geq 1 $ and $\sum_{s=0}^r x^s \leq 2 x^r$ if $x \geq 2$. We obtain $|I| \leq 2 ( 3 \sqrt 2  k  z^{1/2} q^{1/4})^k$. It concludes the proof of the lemma. 
\end{proof}

\bibliographystyle{abbrv}
\bibliography{bib}

\bigskip
\noindent
 Charles Bordenave \\ 
Institut Math\'ematiques de Marseille . \\
39 rue Fr\'ed\'eric Joliot Curie.  13013 Marseille, France. \\
\noindent
{E-mail:}\href{mailto:charles.BORDENAVE@univ-amu.fr}{charles.bordenave@univ-amu.fr} \\

\noindent
Beno\^it Collins \\
Department of Mathematics, Graduate School of Science,  \\
Kyoto University, Kyoto 606-8502, Japan. \\
\noindent
{E-mail:}\href{mailto:collins@math.kyoto-u.ac.jp}{collins@math.kyoto-u.ac.jp} \\

\end{document}